\documentclass[11pt]{article} 
\usepackage[utf8]{inputenc}
\usepackage[T1]{fontenc}
\usepackage[english]{babel}
\makeatletter
\newcommand{\note}[1]{{\textcolor{red}{[#1]}}\@latex@warning{Note: #1}}
\makeatother
\usepackage{amsmath,amsfonts,amssymb,amsthm,mathtools}
\usepackage{enumitem}

\usepackage[mono=false]{libertine}
\usepackage[cmintegrals,libertine]{newtxmath}
\usepackage[cal=euler, scr=boondoxo]{mathalfa}

\useosf
\usepackage[a4paper,vmargin={3.5cm,3.5cm},hmargin={2.5cm,2.5cm}]{geometry}
\usepackage[font={small,sf}, labelfont={sf,bf}, margin=1cm]{caption}
\usepackage[pdftex,colorlinks=true]{hyperref}
\usepackage[pdftex]{color,graphicx}

\usepackage{stackrel}
\usepackage{soul}
\usepackage{hyperref}

\theoremstyle{plain}
\newtheorem{theorem}{Theorem}
\newtheorem{corollary}[theorem]{Corollary}
\newtheorem{proposition}[theorem]{Proposition}
\newtheorem{lemma}[theorem]{Lemma}
\theoremstyle{definition}

\newcommand{\ind}{\mathbf{1}}

\newcommand{\R}{\mathbb{R}}
\newcommand{\C}{\mathbb{C}}
\newcommand{\N}{\mathbb{N}}
\newcommand{\Z}{\mathbb{Z}}

\newcommand{\rmd}{\mathrm{d}}

\newcommand{\Var}{\operatorname{Var}}
\newcommand{\breakline}{\nonumber\\&\quad\quad\quad}
\newcommand{\Mod}{\mathcal{M}}
\newcommand{\WP}{\mathrm{WP}}
\newcommand{\tree}{\mathfrak{t}}
\newcommand{\polytope}{\mathcal{A}_\tree}
\newcommand{\dd}{\mathop{}\!\mathrm{d}}
\newcommand{\vertex}{\mathsf{v}}
\newcommand{\edge}{\mathsf{e}}
\newcommand{\treeset}{\mathfrak{T}}
\newcommand{\vsetinner}{\mathsf{V}}
\newcommand{\vsetboundary}{\mathsf{B}}
\newcommand{\edgeset}{\mathsf{E}}
\newcommand{\cornerset}{\mathsf{C}}
\newcommand{\corner}{\mathsf{c}}
\newcommand{\spinebijection}{\mathsf{Spine}}
\newcommand{\bvertex}{\mathsf{b}}

\usepackage{tikz}

\newcommand{\hypwedge}{%
  \mathrel{\vcenter{\hbox{%
    \begin{tikzpicture}[scale=0.25, line join=round, line cap=round]
      \coordinate (A) at (0,0.9);
      \coordinate (B) at (1,0.9);
      \coordinate (C) at (0.5,0.);
      \draw[thick]
        (A) .. controls (0.3,0.5) and (0.35,0.6) .. (C)
        (C) .. controls (0.65,0.6) and (0.7,0.5) .. (B)
        (B) .. controls (0.5,1.1) and (0.5,1.1) .. (A);
    \end{tikzpicture}%
  }}}%
}
\newcommand{\hypdiamond}{%
  \mathrel{\vcenter{\hbox{%
    \begin{tikzpicture}[scale=0.29, line join=round, line cap=round]
      \coordinate (A) at (0.1,0.5);
      \coordinate (B) at (0.9,0.5);
      \coordinate (C) at (0.5,0);
	  \coordinate (D) at (0.5,1);
      \draw[thick]
        (A) .. controls (0.25,0.55) and (0.45,0.75) .. (D)
		(A) .. controls (0.25,0.45) and (0.45,0.25) .. (C)
        (B) .. controls (0.75,0.55) and (0.55,0.75) .. (D)
		(B) .. controls (0.75,0.45) and (0.55,0.25) .. (C)
		(A) -- (B);
    \end{tikzpicture}%
  }}}%
}
\newcommand{\triangulation}{%
  \mathrel{\vcenter{\hbox{%
    \begin{tikzpicture}[scale=0.4, line join=round, line cap=round]
      \coordinate (A) at (0,0.05);
      \coordinate (B) at (1,0.05);
      \coordinate (C) at (0.5,0.95);
      \draw[thick]
        (A) .. controls (0.4,0.4) and (0.4,0.6) .. (C)
        (C) .. controls (0.6,0.6) and (0.6,0.4) .. (B)
        (B) .. controls (0.6,0.2) and (0.4,0.2) .. (A);
    \end{tikzpicture}%
  }}}%
}

\makeatletter
\DeclareRobustCommand{\cev}[1]{%
  \mathpalette\do@cev{#1}%
}
\newcommand{\do@cev}[2]{%
  \fix@cev{#1}{+}%
  \reflectbox{$\m@th#1\vec{\reflectbox{$\fix@cev{#1}{-}\m@th#1#2\fix@cev{#1}{+}$}}$}%
  \fix@cev{#1}{-}%
}
\newcommand{\fix@cev}[2]{%
  \ifx#1\displaystyle
    \mkern#23mu
  \else
    \ifx#1\textstyle
      \mkern#23mu
    \else
      \ifx#1\scriptstyle
        \mkern#22mu
      \else
        \mkern#22mu
      \fi
    \fi
  \fi
}
  
\makeatother

\begin{document}

\title{\bf A tree bijection for the moduli space of\\genus-$\boldsymbol{0}$ hyperbolic surfaces with boundaries}

\author{\textsc{Timothy Budd}\footnote{Email: \texttt{\href{mailto:t.budd@science.ru.nl}{t.budd@science.ru.nl}}}, \, \textsc{Thomas Meeusen} \, and \, \textsc{Bart Zonneveld} \\[5mm]
{\small IMAPP, Radboud University, Nijmegen, The Netherlands.}}
\date{\today}
\maketitle

\vspace{-6mm}
\begin{abstract}
The Weil-Petersson volume of genus-$g$ hyperbolic surfaces with geodesic boundaries is known since work of Mirzakhani to be polynomial in the boundary lengths.
We provide a bijective proof of this fact in the genus-$0$ case in the presence of a distinguished cusp.
It is based on a generalization of a recent tree bijection, by the first author and Curien, to the setting with geodesic boundaries, requiring an extension of the Bowditch-Epstein-Penner spine construction.
As an application of our tree bijection we establish an explicit formula for the distance-dependent three-point function, which records an exact metric statistic measuring the difference of two geodesic distances among a triple of distinguished cusps in a Weil-Petersson random surface.
We conclude with a discussion of the relevance of this function to the topological recursion of Weil-Petersson volumes and metric properties of Weil-Petersson random surfaces with many boundaries or cusps.  
\end{abstract}
\vspace{-3mm}
\begin{figure}[h!]
	\centering
	\includegraphics[width=\linewidth]{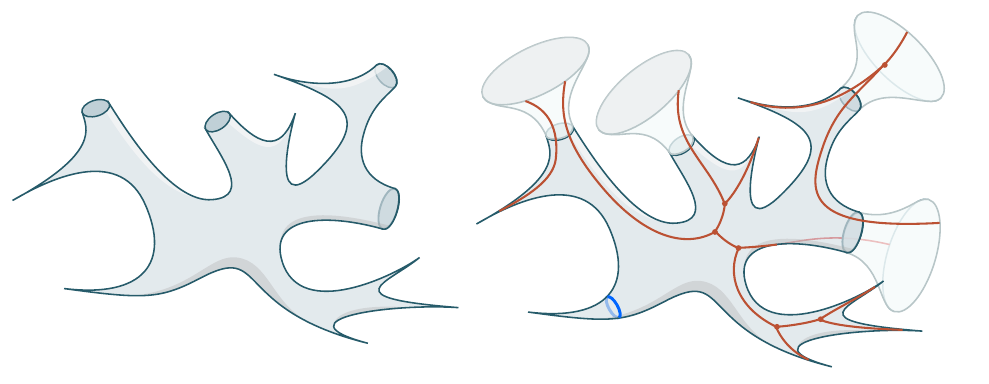}
	\caption{Illustration of a genus-$0$ hyperbolic surface $\mathsf{X}$ with cusps and geodesic boundaries (left). The bijection in this work considers the spine construction in the extended hyperbolic surface $\check{\mathsf{X}}$ (right), obtained by attaching funnels to the geodesic boundaries of $\mathsf{X}$. The tree in question corresponds to the cut locus of points that have multiple length-minimizing geodesics to a horocycle around the distinguished cusp, called the origin.\label{fig:bijection}}
\end{figure}

\tableofcontents

\section{Introduction}

Many families of random metrics on a fixed topological surface have been studied in the literature using equally diverse mathematical techniques.
Since these families often share universal characteristics, or, in the language of physics, can each be viewed as a version of \emph{two-dimensional quantum gravity}, it is natural to investigate whether any of the employed techniques are transferable.

One prominent family of models, with a long history in both mathematics and physics, is that of random combinatorial maps (or ribbon graphs).
The extensive toolkit used to study their enumeration, as well as statistical properties of uniformly sampled maps, includes analytic combinatorics, matrix integrals, integrable hierarchies, topological recursion, probabilistic exploration processes, and more. 
Yet there is one family of techniques, intrinsic to the combinatorial structure of maps, that has proved particularly powerful in recent decades: the \emph{bijective method}.
The idea is simple: a bijection between two combinatorial classes not only establishes their equinumeration, but also relates uniform random samples in both classes.

This is particularly effective when one of the classes consists of (decorated) combinatorial trees, because many combinatorial and probabilistic methods are readily available to analyse them.
Starting with the bijections of Cori--Vauquelin \cite{Cori_Planar_1981} and Schaeffer \cite{Schaeffer_Conjugaison_1998} (and even earlier work of Mullin \cite{Mullin_enumeration_1967}), a rich collection of such tree bijections has been identified in the planar-map literature, see for instance \cite{Bouttier_Planar_2004,Bernardi_bijection_2012,Albenque_generic_2015}.
As we will discuss in more detail below, these tree bijections have played a central role in establishing scaling limits for the metric of large random planar maps, notably in the proofs of convergence to the Brownian sphere \cite{LeGall_scaling_2010,Miermont_Brownian_2013}. 

The present work aims to transfer aspects of this bijective method from planar maps to the continuous setting of moduli spaces of planar hyperbolic surfaces equipped with their Weil-Petersson volume measures.
One may view the computation of Weil-Petersson volumes as a continuous analogue of the enumeration of maps, and the Weil-Petersson random hyperbolic surface as an analogue of the uniform random map.
The utility of a tree bijection in this context was already demonstrated by the first author and Curien \cite{budd2025random} by proving that the Weil-Petersson punctured hyperbolic sphere has the same metric scaling limit, in the form of the Brownian sphere, as many models of large random planar maps (see below for discussions).
The main goal of this work is to extend the bijective method to planar hyperbolic surfaces with geodesic boundaries.

\subsection{Main results}\label{sec:mainresults}

We start by explaining our main bijective result.
Recall that a hyperbolic surface with geodesic boundaries is a surface with a Riemannian metric of constant (Gaussian) curvature $-1$ such that the boundaries are closed hyperbolic geodesics. 
For $2g-2+n>0$ and $\mathbf{L} = (L_1,\ldots,L_n) \in \R_{\geq 0}^n$,  $\mathcal{M}_{g,n}(\mathbf{L})$ is the moduli space of hyperbolic surfaces with $n$ labeled boundaries of specified lengths $L_1,\ldots,L_n$, viewed modulo orientation-preserving (and boundary preserving) isometries.
Some (or all) of the boundaries are allowed to have zero length, in which case they correspond to cusps.
The moduli space $\mathcal{M}_{g,n}(\mathbf{L})$ comes with a natural volume measure $\operatorname{Vol}_\WP$ arising from its Weil-Petersson symplectic structure.
The volume of the full moduli space is finite and known as the Weil-Petersson volume $V_{g,n}(\mathbf{L})=\operatorname{Vol}_\WP(\mathcal{M}_{g,n}(\mathbf{L}))$.

In this work we focus exclusively on the spherical case $g=0$ and we will always assume that the first boundary is a cusp, so that the moduli space of interest is $\mathcal{M}_{0,1+n}(0,\mathbf{L})$ for $n\geq 2$ and $\mathbf{L}\in \R_{\geq0}^n$.
This special cusp will be called the \emph{origin} and the remaining boundaries will be considered to be labeled $1,\ldots,n$.
\begin{figure}[h!]
    \centering
    \includegraphics[width=.95\linewidth]{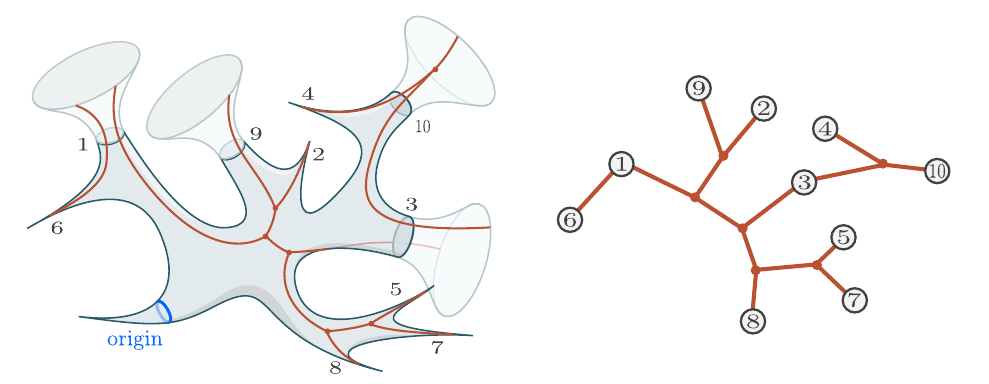}
    \caption{Example of the tree $\tree$ associated to a surface $\mathsf{X} \in \mathcal{M}_{0,1+n}(\mathbf{L})$, with in this case $n=10$ boundaries (of which $6$ are cusps and $4$ are geodesic) besides the origin cusp. The spine of the extended hyperbolic surface $\check{\mathsf{X}}$ has the structure of a bicolor plane tree $\tree \in \treeset_n$.\label{fig:bijectiontree}}
\end{figure}

\paragraph{Tree bijection.}
Inspired by the spine construction of Bowditch \& Epstein \cite{Bowditch_Natural_1988} and Penner \cite{Penner_decorated_1987}, already employed in \cite{budd2025random}, as well as tree bijections for planar maps (see the discussion in Section~\ref{sec:discussion}), it is natural to consider the \emph{spine} or \emph{cut-locus} associated to the origin in a hyperbolic surface $\mathsf{X} \in \mathcal{M}_{0,1+n}(0,\mathbf{L})$.
This is the subset of points of $\mathsf{X}$ that have more than one length-minimizing geodesic to the origin.
Due to the spherical topology the spine generally has the structure of a forest, i.e.\ a disjoint union of real trees.
To turn it into a useful combinatorial tree, one should consider the spine in the extended hyperbolic surface $\check{\mathsf{X}}$ obtained by attaching infinite hyperbolic funnels to $\mathsf{X}$ (Figure~\ref{fig:bijection}).
As will be explained in detail in Section~\ref{sec:bijection}, joining the endpoints of the infinite arcs of the spine that run into the same funnel, gives the spine the combinatorial structure of a plane tree $\tree$, see Figure~\ref{fig:bijectiontree}.
This tree together with a decoration by real labels, carefully chosen  to make the encoding bijective, is denoted $\mathsf{Spine}(\mathsf{X})$.

To make this precise, we consider the collection $\treeset_n$ of (red and white) bicolored plane trees $\tree$ with $n$ white vertices labeled $1,\ldots,n$, also called \emph{boundary} vertices, and an arbitrary number of red vertices of degree $3$, called \emph{inner} vertices.
We let $\vsetinner(\tree)$, $\vsetboundary(\tree)$ and $\edgeset(\tree)$ be the set of inner vertices, boundary vertices and edges of $\tree$.
We also denote by $\vec{\edgeset}(\tree)$ the oriented edges and for convenience we denote by $\vec{\edge},\cev{\edge}\in\vec{\edgeset}(\tree)$ the two orientations of an edge $\edge\in \edgeset(\tree)$.
We write $\vec{\edge}_{\vertex,1}, \vec{\edge}_{\vertex,2}, \vec{\edge}_{\vertex,3} \in \vec{\edgeset}(\tree)$ respectively $\vec{\edge}_{\bvertex,1}, \ldots, \vec{\edge}_{\bvertex,\deg(\bvertex)} \in \vec{\edgeset}(\tree)$ for the sequence of oriented edges starting at an inner vertex $\vertex$, respectively boundary vertex $\bvertex$, where it is assumed that at each vertex a counterclockwise ordering of the edges is chosen in an arbitrary deterministic fashion.
If $\mathbf{L} \in \R_{\geq 0}^n$, we will write $L_{\bvertex} = L_i$ in case the boundary vertex carries label $i$.
To properly accommodate cusps among the boundaries, we further introduce the restricted family of trees $\treeset_n(\mathbf{L}) \subset \treeset_n$, by requiring $\deg(\bvertex) = 1$ in case $L_\bvertex = 0$. 

To $\tree\in\treeset_n(\mathbf{L})$ we associate a $(2n-4)$-dimensional polytope $\mathcal{A}_{\tree}(\mathbf{L})\subset \R^{6n-6}$ as follows.
For $k\geq 1$ and $y >0$ we denote by $\Delta_k^{(y)} \coloneqq \{ \mathbf{x} \in (0,\infty)^{k} : x_1 + \cdots + x_{k} = y \}$ the open $(k-1)$-dimensional simplex of size $y$, and let $\Delta_1^{(0)} = \{0\}$ by convention.
Then
\begin{align}
	\mathcal{A}_{\tree}(\mathbf{L}) \coloneqq \left\{ \varphi : \vec{\edgeset}(\tree) \to \R_{\geq 0} \middle|\begin{array}{l}
	\varphi(\vec{\edge}_{\bvertex,i}) = 0\text{ for }\bvertex\in\vsetboundary(\tree), 1\leq i\leq\deg(\bvertex),\\
	\varphi(\vec{\edge}_{\vertex,i}) > 0\text{ for }\vertex\in \vsetinner(\tree),  1\leq i\leq 3,\\
	\varphi(\vec{\edge}_{\vertex,1})+\varphi(\vec{\edge}_{\vertex,2})+\varphi(\vec{\edge}_{\vertex,3}) = \pi\text{ for }\vertex\in \vsetinner(\tree), \\ 
	\varphi(\vec{\edge}) + \varphi(\cev{\edge}) < \pi\text{ for }\edge\in \edgeset(\tree)\end{array}\right\} \times \prod_{\bvertex\in \vsetboundary(\tree)} \left(\Delta^{(L_\bvertex/2)}_{\deg(\bvertex)}\right)^2.\label{eq:polytopedefgeneric}
\end{align}
This polytope is naturally equipped with a volume measure given by $2^{n-2}$ times the $(2n-4)$-dimensional Lebesgue measure
\begin{align}
	\operatorname{Vol}_\tree = 2^{n-2} \prod_{\vertex\in \vsetinner(\tree)}\rmd\varphi(\vec{\edge}_{\vertex,1})\rmd\varphi(\vec{\edge}_{\vertex,2}) \prod_{\bvertex\in \vsetboundary(\tree)} \prod_{j=1}^{\deg(\bvertex)-1}\rmd w_{\bvertex,j}\rmd v_{\bvertex,j}.\label{eq:measurepolytope}
\end{align}
Here $(w_{\bvertex,1},\ldots,w_{\bvertex,\deg(\bvertex)},v_{\bvertex,1},\ldots,v_{\bvertex,\deg(\bvertex)})\in \left(\Delta^{(L_\bvertex/2)}_{\deg(\bvertex)}\right)^2$ parametrize the pair of simplices associated to boundary vertex $\bvertex$.
Our main bijective result can then be formulated as follows.

\begin{theorem}\label{thm:introbijection}
	For $n\geq 2$, there exists an open subset $\mathcal{M}^\circ_{0,n+1}(0,\mathbf{L}) \subset \mathcal{M}_{0,1+n}(0,\mathbf{L})$ of full Weil-Petersson measure and a bijection
	\begin{align*}
		\mathsf{Spine} : \mathcal{M}^\circ_{0,n+1}(0,\mathbf{L}) \to \bigsqcup_{\tree\in\treeset_n(\mathbf{L})} \mathcal{A}_\tree(\mathbf{L})
	\end{align*}
	such that the pushforward measure $\mathsf{Spine}_* \operatorname{Vol}_\WP$ agrees with $\operatorname{Vol}_\tree$ on $\mathcal{A}_\tree(\mathbf{L})$. 
\end{theorem}
In fact, we will see in Theorem~\ref{thm:fullbijection} that $\mathsf{Spine}$ makes sense a bijection for the full moduli space $\mathcal{M}_{0,1+n}(0,\mathbf{L})$ by considering an extended family of trees. But the polytopes associated to these extra trees are of dimension smaller than $2n-4$ and therefore do not contribute to the volume.

As a direct consequence of Theorem~\ref{thm:introbijection} we obtain a bijective interpretation, in the special case of $g=0$ with the first boundary a cusp, of the well-known fact \cite[Theorem~1.1]{Mirzakhani2007} due to Mirzakhani that $V_{g,n}(\mathbf{L})$ is a rational homogeneous polynomial of degree $3g-3+n$ in $\pi^2,L_1^2,\ldots,L_n^2$.

\begin{corollary}\label{cor:volpolynomial}
	For $n \geq 2$, the Weil--Petersson volume $V_{0,n+1}(0,\mathbf{L})$ is the rational homogeneous polynomial of degree $n-2$ in $\pi^2,L_1^2,\ldots,L_n^2$ given by the finite sum
	\begin{align}
		V_{0,n+1}(0,\mathbf{L}) = \sum_{\tree\in\treeset_n(\mathbf{L})} \operatorname{Vol}_\tree(\mathcal{A}_\tree(\mathbf{L})), \qquad
		\operatorname{Vol}_\tree(\mathcal{A}_\tree(\mathbf{L})) \in \pi^{2|\vsetinner(\tree)|} \cdot \prod_{\bvertex\in \vsetboundary(\tree)} L_\bvertex^{2\deg(\bvertex)-2} \cdot \mathbb{Q}. \label{eq:volpolytopesum}
	\end{align}
\end{corollary}
\begin{proof}
	For each $\tree\in\treeset_n(\mathbf{L})$, the volume $\operatorname{Vol}_\tree(\mathcal{A}_\tree(\mathbf{L}))$ factorizes over the Cartesian product in \eqref{eq:polytopedefgeneric}.
	The angular part is a rational polytope of dimension $2|\vsetinner(\tree)|$ rescaled by $\pi$, while the simplex $\Delta^{(L_\bvertex/2)}_{\deg(\bvertex)}$ is a rational polytope of dimension $\deg(\bvertex)-1$ rescaled by $L_{\bvertex}$.
	Therefore $\operatorname{Vol}_\tree(\mathcal{A}_\tree(\mathbf{L}))$ is a rational monomial of degree $|\vsetinner(\tree)| + \sum_{\bvertex\in\vsetboundary(\tree)} (\deg(\bvertex)-1) = n-2$ in $\pi^2,L_1^2,\ldots,L_n^2$.
\end{proof}

Computing the volumes $\operatorname{Vol}_\tree(\mathcal{A}_\tree(\mathbf{L}))$ is not entirely straightforward, due to the \emph{Delaunay condition} $\varphi(\vec{\edge}) + \varphi(\cev{\edge}) < \pi$ for $\edge\in \edgeset(\tree)$.
Luckily it turns out that the computation simplifies drastically if one uses an inclusion-exclusion principle to reverse the inequality, meaning that one instead considers the \emph{anti-Delaunay condition} $\varphi(\vec{\edge}) + \varphi(\cev{\edge}) > \pi$ for some edges.
The result is that the contributions to the volume in \eqref{eq:volpolytopesum} can be repackaged in terms of an adapted collection of trees, that we call \emph{anti-Delaunay trees}. 
They form a superset $\widetilde{\treeset}_n(\mathbf{L}) \supset \treeset_n(\mathbf{L})$ of the bicolor trees introduced above, the only difference being that $\tree \in \widetilde{\treeset}_n(\mathbf{L})$ is allowed to have inner vertices of arbitrary degree $3$ or larger.

\begin{figure}[h!]
    \centering
    \includegraphics[width=\linewidth]{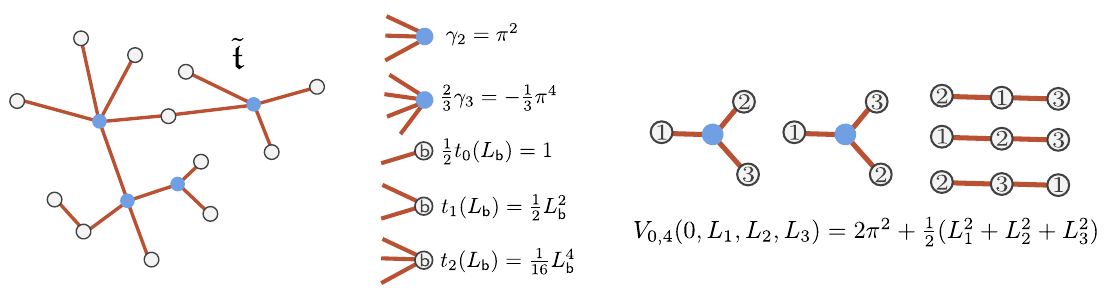}
    \caption{Left: Example of an anti-Delaunay tree $\tilde{\tree}\in \widetilde{\treeset}_{12}(\mathbf{L})$. Middle: The contribution to the Weil-Petersson volume for several low-degree vertices. Right: all five trees in $\widetilde{\treeset}_4(\mathbf{L})$. \label{fig:antitreeweights}}
\end{figure}

\begin{theorem}\label{thm:antidelaunay}
	For $n\geq 2$ and $\mathbf{L}\in\R_{\geq 0}^n$, the Weil-Petersson volume $V_{0,1+n}(0,\mathbf{L})$ can be expressed as a sum over anti-Delaunay trees as
	\begin{align}
		V_{0,1+n}(0,\mathbf{L})
		&=\sum_{\tree\in\widetilde{\treeset}_n(\mathbf{L})} \prod_{\vertex\in\vsetinner(\tree)} \frac{2^{\deg(\vertex)-2}}{(\deg(\vertex)-1)!}\,\gamma_{\deg(\vertex)-1}\prod_{\bvertex \in \vsetboundary(\tree)} \frac{2^{\deg(\bvertex)-2}}{(\deg(\bvertex)-1)!}\,t_{\deg(\bvertex)-1}(L_\bvertex),
	\end{align}
	where
	\begin{align}
		\gamma_k=\frac{(-1)^k\pi^{2k-2}}{(k-1)!}, \quad t_k(L)=\frac{2}{k!}\left(\frac{L}{2}\right)^{2k}.
	\end{align}
\end{theorem}

See Figure~\ref{fig:antitreeweights} for an illustration of $V_{0,4}(0,L_1,L_2,L_3) = 2\pi^2 + \tfrac12(L_1^2+L_2^2+L_3^2)$.

Readers familiar with intersection numbers on the moduli space of curves and the KdV hierarchy, may recognize $t_k$ as the \emph{times} appearing in the KdV differential equations and the \emph{time shifts} $\gamma_k$ originating from expressing the cohomology class of the Weil-Petersson symplectic form in terms of $\psi$-classes \cite{Witten_Two_1991}.
See Mirzakhani's proof of Witten's conjecture \cite{Mirzakhani2007a} and \cite{Manin_Invertible_2000,Mulase_Mirzakhanis_2008,Liu_Recursion_2009}, or the overview in \cite[Section~4]{budd2020irreducible}.

\paragraph{Generating functions.}
In this context it is well-known that the genus-$0$ intersection numbers, and therefore also the genus-$0$ Weil-Petersson volumes, are completely determined by the \emph{string equation} \cite{Witten_Two_1991,Kontsevich1992}.
As a consequence of Theorem~\ref{thm:antidelaunay}, our bijection allows to reinterpret the string equation as the characteristic equation for the generating function of the combinatorial anti-Delaunay trees $\widetilde{\treeset}_n(\mathbf{L})$.

To make this precise, it is necessary to switch from working at fixed number of boundaries to generating functions of Weil-Petersson volumes.
Since the Weil-Petersson volumes depend on an $n$-tuple of boundary lengths, some care is required to formalize a notion of generating function.
A versatile definition was given by the first and third author in \cite[Section~1.4]{budd2023topological}.
Informally, the generating function $F_g[\mu]$ of genus-$g$ Weil-Petersson volumes is the formal power series
\begin{align}
	F_g[\mu] = \sum_{n\geq 0}\frac{1}{n!}\int_{\R_{\geq 0}^n} V_{g,n}(\mathbf{L}) \rmd\mu(L_1)\cdots \rmd\mu(L_n), \label{eq:partfun}
\end{align}
which has a functional dependence\footnote{Using physicist's notational conventions square brackets indicate functional dependence.} on a \emph{weight} $\mu$, which one may think of as a measure on $\R_{\geq 0}$ or, more generally, as a linear function on the ring of even, real polynomials (see Section~\ref{sec:stringeq} for details). 
For us the important generating function is that of double-cusped hyperbolic spheres
\begin{align}
	R[\mu]=\sum_{n\geq1}\frac{1}{n!}\int V_{0,2+n}(0,0,\mathbf{L})\, \rmd{\mu(L_1)}\dots\rmd{\mu(L_n)} = \int \rmd \mu(L_1) + O(\mu^2) \label{eq:Rseriesdef},
\end{align}
which can also be understood as a second derivative of $F_0[\mu]$, see also \cite{budd2020irreducible,budd2023topological}.
Theorem~\ref{thm:antidelaunay} translates into a bijective proof of the known string equation \cite[Equation~(21)]{budd2023topological} for $R[\mu]$.

\begin{corollary}[String equation]\label{cor:string}
	This generating function satisfies the \emph{string equation}
	\begin{align}
		R[\mu]&=\sum_{d\geq0}\frac{2^{d-1}}{d!}\left(t_{d}[\mu]+\ind_{\{d\geq2\}}\gamma_{d}\right)R[\mu]^{d},\quad t_k[\mu]\coloneqq \int t_k(L) \rmd\mu(L).
	\end{align}
	or equivalently $Z(R[\mu];\mu] = 0$, where
	\begin{align}
		Z(r;\mu] = \frac{\sqrt{r}}{\sqrt{2}\pi}J_1\left(2\pi\sqrt{2r}\right)-\int I_0\left(L\sqrt{2r}\right)\dd{\mu(L)}
	\end{align}
	is a formal power series in $r$ involving the Bessel functions $J_1$ and $I_0$.
\end{corollary}

As explained in detail in \cite{budd2023topological}, the generating function $R[\mu]$ plays a central role in the topological recursion of Weil-Petersson volumes of hyperbolic surfaces with weights on the boundaries.
For an introduction to topological recursion \cite{Eynard_Invariants_2007} we refer the reader to \cite{Eynard_short_2014,eynard2016counting,Borot_Topological_2020}.
To be precise, we considered the formal power series defined in terms of $R[\mu]$ and $Z(r;\mu]$ via 
\begin{align}
	\eta(u;\mu] = \sum_{p=0}^\infty \frac{u^{2p}}{(2p+1)!!} \frac{\partial^{p+1} Z}{\partial r^{p+1}}(R[\mu];\mu] = \frac{\sin(2\pi u)}{2\pi u} + \int \frac{\cos(2\pi u)-\cosh(L_1u)}{u^2}\rmd\mu(L_1) + O(\mu^2)\label{eq:etadef}
\end{align}
Then the invariants $\omega_{g,n}(\mathbf{u})$ of the spectral curve
\begin{align}\label{eq:spectralcurve}
	\begin{cases}
		x(u) = u^2,\\ y(u) = 2u\, \eta(u;\mu],
	\end{cases}
\end{align}
in the sense of \cite{Eynard_Invariants_2007}, were shown to be related to generating functions of Weil-Petersson volumes of genus-$g$ surfaces with $n$ distinguished boundaries and an arbitrary number of extra weighted boundaries.  
To be precise, according to \cite[Lemma~20]{budd2023topological} we have the formal power series identity
\begin{align}
	\sum_{p=0}^\infty \frac{1}{p!} \int_{\R_{>0}^{n+p}} V_{g,n+p}(\mathbf{L},\mathbf{K}) \prod_{i=1}^n L_i e^{-z_i L_i}\rmd L_i  \prod_{j=1}^p  \rmd \mu(K_j) = \omega_{g,n}(\mathbf{u}) \prod_{i=1}^n \frac{z_i}{u_i}, \qquad u_i \coloneqq \sqrt{z_i^2 - 2R[\mu]}.
\end{align}
Upon setting the weight $\mu$ to zero, one finds $R[0]=0$ and $\eta(u;0] = \sin(2\pi u)/(2\pi u)$, this reduces to the standard Mirzakhani recursion of Weil-Petersson volumes \cite{eynard2007weil,Mirzakhani2007}. 

\paragraph{Metric space statistics.}
Paralleling a similar phenomenon in planar map combinatorics \cite{Bouttier_Bijective_2022,Bouttier_Enumeration_2024}, it was shown in \cite[Theorem~2]{budd2023topological} that this generalization of Mirzakhani's recursion is most naturally phrased in the setting of hyperbolic surfaces with certain ``tight'' boundaries.
Based on an, as of yet heuristic, geometric interpretation via a ``tight'' version of the pants decomposition underlying the result of Mirzakhani, it was conjectured in \cite[Section~1.6]{budd2023topological} that the spectral curve hides precise metric statistics (see also Section~\ref{sec:discussion}).
To be precise, for $n\geq 0$, $\mathbf{L} \in \R_{\geq 0}^n$, we let $\mathsf{X} \in \mathcal{M}_{0,3+n}(0,0,0,\mathbf{L})$ be a triply-cusped hyperbolic sphere and consider the unique unit-length horocycles $h_1,h_2,h_3$ around these cusps.
Denoting the hyperbolic metric on $\mathsf{X}$ by $d_{\mathsf{X}} : \mathsf{X}\times \mathsf{X}\to \R_{\geq 0}$, we let the \emph{distance difference} be $D(\mathsf{X}) = d_{\mathsf{X}}(h_1,h_2) - d_{\mathsf{X}}(h_1,h_3)$, see Figure~\ref{fig:threepoint}.

\begin{figure}[h!]
    \centering
    \includegraphics[width=.6\linewidth]{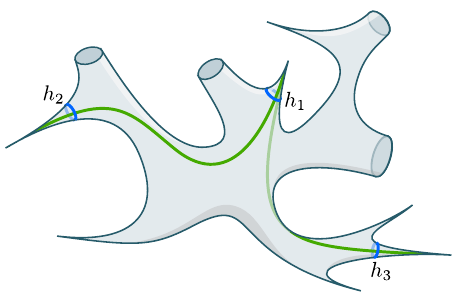}
    \caption{A hyperbolic sphere $\mathsf{X}$ with three distinguished cusps equipped with their unit-length horocycles $h_1,h_2,h_3$. The distance difference $D(\mathsf{X})$ measures the length difference between the two length-minimizing geodesics.\label{fig:threepoint}}
\end{figure}

This defines a continuous mapping
\begin{align}
	D : \mathcal{M}_{0,3+n}(0,0,0,\mathbf{L}) \to \R.
\end{align}
We may thus consider the push-forward $D_*\operatorname{Vol}_\WP$ of the Weil-Petersson measure to $\R$, which for $n > 0$ is absolutely continuous with respect to the Lebesgue measure on $\R$ with a density that we denote $X_n(x;\mathbf{L}) \rmd x$.
We let $\hat{X}_n(u;\mathbf{L})$ be its two-sided Laplace transform,
\begin{align}
	\hat{X}_n(u;\mathbf{L}) \coloneqq \int_{-\infty}^\infty e^{2u x} X_n(x;\mathbf{L}) \rmd x = \int_{\mathcal{M}_{0,3+n}(0,0,0,\mathbf{L})} e^{2u D(\mathsf{X})} \rmd\operatorname{Vol}_\WP(\mathsf{X}).\label{eq:Xhatn}
\end{align}
It was predicted in \cite[Section~1.6]{budd2023topological} that the generating function of $\hat{X}_n(u;\mathbf{L})$, which we call the \emph{distance-dependent three-point function}, is directly related to the spectral curve \eqref{eq:spectralcurve}.
With the help of the tree bijection in this paper we can settle this conjecture.

\begin{theorem}[distance-dependent three-point function]\label{thm:dist}
	The formal power series
	\begin{align}
		\hat{X}(u;\mu] = \sum_{n \geq 0} \frac{1}{n!} \int \hat{X}_n(u;\mathbf{L}) \rmd \mu(L_1)\cdots \rmd \mu(L_n).
	\end{align}
	is related to \eqref{eq:etadef} via
	\begin{align}
		\hat{X}(u;\mu] = \frac{\sin(2\pi u)}{2\pi u\, \eta(u;\mu]}.\label{eq:Xhatdef}
	\end{align}
\end{theorem}
As far as we are aware, this constitutes the first exact global metric space statistic for Weil-Petersson random surfaces of fixed genus and fixed number of boundaries.
Many works deal with statistics of lengths of closed geodesics on such surfaces \cite{Mirzakhani_Growth_2013,Guth_Pants_2011}, but such geodesics are not guaranteed to be geodesics in the metric space sense of being (concatenations of a few) length-minimizing curves.
An exception\footnote{A pair of points positioned oppositely on a systole of length $L$ are necessarily at hyperbolic distance $L/2$, contrary to the case of a typical closed geodesic.} is the length of the \emph{systole}, i.e.\ the shortest closed geodesic on the surface, but exact statistics are only known in large-$g$ and/or large-$n$ limits \cite{mirzakhani2017lengths,Hide_Short_2025}. 
The Gromov-Hausdorff convergence of random genus-$0$ hyperbolic surfaces with cusps to the Brownian sphere, shown by Curien and the first author in \cite{budd2025random} (see below), allows to import a host of known metric space statistics from the latter, but again only in the large-$n$ limit.
Finally, we note that some metric space statistics arising from lengths of ``tight'' closed geodesics can be extracted from the already mentioned work \cite{budd2023topological}, see also \cite{Budd_tight_2025}.

To illustrate the non-trivial content of the identity \eqref{eq:Xhatdef}, let us examine the case $n=0,1$. 
With the help of \eqref{eq:Rseriesdef} and \eqref{eq:etadef}, we see that to first order in $\mu$ this power series is given by 
\begin{align}
	\hat{X}(u;\mu] = 1 + \int \frac{2\pi}{u} \frac{\cosh(L_1 u)-\cos(2\pi u)}{\sin(2\pi u)}\rmd\mu(L_1) + \cdots.
\end{align} 
The constant term $\hat{X}_0(u) = 1$ for $n=0$ reflects that the triply-punctured sphere deterministically satisfies $d_{\mathrm{hyp}}(h_1,h_2) = d_{\mathrm{hyp}}(h_1,h_3)$.
For $n=1$, one can perform the inverse Laplace transform of the integrand to find the explicit distribution for the distance difference on a hyperbolic surface with three cusps and a single boundary of length $L_1 \geq 0$,
\begin{align*}
	X_1(x;L_1) = 2\log\left(\frac{\cosh(x)+\cosh(L_1/2)}{\cosh(x)-1}\right).
\end{align*}
Obtaining a similar closed form $X_2(x;L_1,L_2)$ already appears to be challenging.

Theorem~\ref{thm:dist} can also be used to derive asymptotics as $n\to\infty$.
Let us demonstrate this for the special case of the random punctured sphere $\mathcal{S}_n \in \mathcal{M}_{0,1+n}(\mathbf{0})$ sampled with probability measure $\operatorname{Vol}_\WP / V_{0,n+1}(\mathbf{0})$, considered in \cite{Hide_Short_2025,budd2025random}.
As a consequence of Theorem~\ref{thm:dist} we prove the following.

\begin{corollary}\label{cor:distvar}
	If $\mathcal{S}_n$ is the Weil-Petersson random $(n+1)$-cusped sphere, then the variance of the distance difference satisfies the asymptotics
	\begin{align}
		\lim_{n\to \infty}\Var\left(\frac{D(\mathcal{S}_n)}{c_\WP \,n^{1/4}}\right) = \sqrt{\frac{\pi}{8}}, \qquad c_\WP = \frac{2\pi}{\sqrt{3c_0}} = 2.3392\ldots,
	\end{align}
	where $c_0$ is the first positive zero of the Bessel function $J_0$.
\end{corollary}

According to \cite[Theorem~2]{budd2025random} the random surface $\mathcal{S}_n$ equipped with the normalized hyperbolic metric $n^{-1/4} d_{\mathcal{S}_n}(\,\cdot\,,\,\cdot\,)$ converges to a constant multiple of the Brownian sphere \cite{Marckert_Limit_2006,LeGall_Uniqueness_2013,Miermont_Brownian_2013}.
In fact, Corollary~\ref{cor:distvar} follows directly from \cite[Proposition~40]{budd2025random} once we know that $\sqrt{\pi/8}$ is the variance of the distance difference for a triple of uniform points in the standard Brownian sphere.
The latter is easily verified: if $(\mathbf{e},Z)$ is the Brownian snake (see e.g.\ \cite{LeGall_Uniqueness_2013}), then this variance is equal to $\mathbb{E}[(Z_a-Z_b)^2]=\mathbb{E}[d_{\mathbf{e}}(a,b)]$ for two uniform points $a$ and $b$, where $d_{\mathbf{e}}(\cdot,\cdot)$ is the pseudometric of the continuum random tree encoded by the Brownian excursion $\mathbf{e} : [0,1] \to \R_{\geq 0}$.
Hence, $\mathbb{E}[d_{\mathbf{e}}(a,b)] = \sqrt{\pi/8}$ is equal to the expected area under the Brownian excursion $\mathbf{e}$, see for instance \cite[Section~2]{Janson_Brownian_2007} and the references there for the value $\sqrt{\pi/8}$.
The virtue of our new proof of Corollary~\ref{cor:distvar} is that it provides an alternative, concise derivation\footnote{Compare to \cite[Section~6.6]{budd2025random}, where the main technical challenge in computing $c_\WP$ came down to establishing the asymptotic variance of a particular Markov chain on $[0,\pi)$. Of course, the most technical part of our derivation is hidden in the proof of Theorem~\ref{thm:dist}.} of the scaling constant $c_\WP = \frac{2\pi}{\sqrt{3c_0}}$.

\begin{proof}[Proof of Corollary~\ref{cor:distvar}]
	Taking the weight to be $\mu = x \delta_0$, with $\delta_0$ the delta measure at zero, turns $\hat{X}(u;\mu]$ into an exponential generating function of $\hat{X}_n(u;\mathbf{0})$ with variable $x$.
	Theorem~\ref{thm:dist} implies that the variance of the distance difference between a uniform triple of unit-length horocycles is given by
\begin{align}
	\Var D(\mathcal{S}_n) = \frac{\hat{X}_n''(0;\mathbf{0})}{4\hat{X}_n(0;\mathbf{0})} = \frac{[x^n]\hat{X}''(0;x\delta_0]}{4 [x^n]\hat{X}(0;x\delta_0]}.
\end{align}
It follows from \eqref{eq:Xhatdef} and \eqref{eq:etadef} that
\begin{align}
	\hat{X}(0;x\delta_0] = \frac{1}{M_0(x)},\qquad \hat{X}''(0;x\delta_0] = - \frac{2 M_1(x)}{3 M_0(x)^2} - \frac{4\pi^2}{3} \hat{X}(0;x\delta_0],\qquad M_p(x)\coloneqq \frac{\partial^{p+1} Z}{\partial r^{p+1}}(R[x\delta_0];0].   
\end{align}
By \cite[Proposition~4.3]{Budd_tight_2025}, see also \cite[Lemma~6]{budd2025random}, the series $M_p(x)$ have the same radius of convergence $x_c > 0$ and
\begin{align*}
	M_1(x) \sim -\left(\frac{4\pi^2}{c_0}\right)^2 x_c, \qquad \frac{M_0(x)^2}{2M_1(x)} \sim x-x_c\qquad \text{as }x\to x_c. 
\end{align*}
Hence 
\begin{align*}
	\hat{X}(0;x\delta_0] \sim \frac{c_0}{4\pi^2 \sqrt{2x_c}} \frac{1}{\sqrt{x_c-x}}, \quad \hat{X}''(0;x\delta_0] \sim \frac{1}{3}\frac{1}{x_c-x}\qquad \text{as }x\to x_c. 
\end{align*}
By standard transfer theorems and the $\Delta$-analyticity of $M_0(x)$ and $M_1(x)$ (see \cite[Lemma~6]{budd2025random}), we obtain the asymptotic estimate
\begin{align}
	\Var D(\mathcal{S}_n) = \frac{[x^n]\hat{X}''(0;x\delta_0]}{4 [x^n]\hat{X}(0;x\delta_0]} &\stackrel{n\to\infty}{\sim} \frac{\frac{1}{3}x_c^{-n-1}}{4\frac{c_0}{4\pi^2 \Gamma(1/2)\sqrt{2n}}x_c^{-n-1}} = \frac{4\pi^2}{3 c_0} \sqrt{\frac{\pi n}{8}},
\end{align}
from which the claim follows by rearranging.
\end{proof}

\subsection{Discussion}\label{sec:discussion}

\paragraph{Relation with planar maps.} Like in the previous works \cite{budd2020irreducible,budd2023topological,budd2025random}, the presented results are largely inspired by the combinatorics of maps (also known as ribbon graphs). 
Here we treat the relation between Weil-Petersson volumes of hyperbolic surfaces and enumeration of maps mostly as an analogy, interpreting the Weil-Petersson measure as a continuous version of the counting measure of maps.
Nonetheless, we point out that there is more to it than an analogy, as  already exemplified in the early nineties by Kontsevich's proof of Witten's conjecture \cite{Witten_Two_1991,Kontsevich1992}, which involved a precise relation between the combinatorics of (metric) maps and intersection numbers on the moduli spaces.
From the point of view of the Weil-Petersson measure this relation can be understood via a limit in which the boundary lengths of the hyperbolic surfaces become very large, such that the geometry of the surface literally starts resembling a metric map \cite{Do_asymptotic_2010,Andersen_Kontsevich_2020,Talbott_Critical_2025}. 
In another direction, the moduli space of genus-$0$ hyperbolic surfaces with cusps is known to be bijectively related to irreducible metric maps due to work of Rivin \cite{Rivin1992,Rivin1996}, see \cite{Charbonnier2017} and the discussion in \cite{budd2020irreducible}.

\begin{figure}[h!]
	\centering
	\includegraphics[width=.9\linewidth]{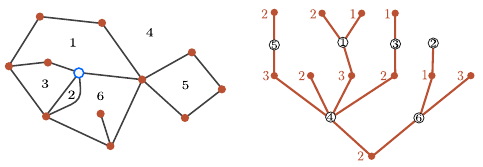}
	\caption{The Bouttier-Di~Francesco-Guitter bijection relates bipartite planar maps with a distinguished vertex (shown here in blue), $n$ labeled faces, which are properly bicolored plane trees with $n$ labeled white vertices and red vertices that are decorated by the graph distances.\label{fig:bdfg}
	}
\end{figure}

It was argued in \cite{budd2025random} that the spine construction for punctured hyberbolic spheres is analogous to the classical Cori-Vauquelin-Schaeffer bijection \cite{Cori_Planar_1981,Schaeffer_Conjugaison_1998} for planar quadrangulations, i.e.\ planar maps with all faces of degree $4$.
One may view the extension presented in this work to the case of hyperbolic spheres with geodesic boundaries as an analogue of the more general Bouttier-Di~Francesco-Guitter (BDiFG) bijection \cite{Bouttier_Planar_2004} for planar maps with faces of arbitrary degree (see also \cite{Bouttier_Planar_2012,Bouttier_Planar_2019} for the convenient dual perspective of slice decomposition).
Here the prescribed lengths of the geodesic boundaries replace the prescribed degrees of the faces of the map.

To appreciate the close analogy, let us phrase the BDiFG bijection in the case of bipartite planar maps in a fashion that resembles our context.
On one side we have the set $M_{n}(\mathbf{d})$ of bipartite planar maps with $n$ labeled faces of half-degrees\footnote{Faces of a bipartite map necessarily half even degree, so it is convenient to record half of the degree.} $\mathbf{d} = (d_1,\ldots,d_n) \in \N^{n}$ and a distinguished vertex, called the origin.
They correspond bijectively to \emph{mobiles}, which are (red/white) bipartite plane trees with $n$ labeled white vertices of degree $d_1, \ldots, d_n$ and with red vertices that are decorated by a positive integer.
These red vertices correspond precisely to the vertices of the original map, except for the origin, and the decoration records the graph distance to the origin.
The only restrictions on the decoration, stemming from their origin as graph distances, is that $1$ appears at least once and that the integer label of two consecutive red neighbours of a white vertex at most increases by $1$ in counterclockwise direction.
See Figure~\ref{fig:bdfg} for an example, noting the similarity to Figure~\ref{fig:bijectiontree}.

The analogue of the double-cusped generating function \eqref{eq:Rseriesdef} is the multivariate generating function  
\begin{align}
	R(g_1,g_2,\ldots) = \sum_{n\geq 1} \frac{1}{n!} \ \sum_{d_1,\ldots,d_n \geq 1} |M_{1+n}(1,\mathbf{d})| \,g_{d_1}\cdots g_{d_n}
\end{align}
for bipartite planar maps with a distinguished vertex and a distinguished face of degree $2$, which can also be interpreted as a distinguished edge inflated to a bigon.
They are also known as elementary slices in the slice decomposition \cite{Bouttier_Planar_2012,Bouttier_Planar_2019}. 
According to \cite[Equation~(2.5)]{Bouttier_Planar_2004}, the characteristic equation for (rooted) mobiles implies the ``string equation''
\begin{align}
	R = 1 + \sum_{k=1}^\infty g_{k} \binom{2k-1}{k} R^k,
\end{align}
which is the analogue of Corollary~\ref{cor:string}.

Finally, we point the reader to \cite{Bouttier_Geodesic_2003,Bouttier_Planar_2012} for a showcase on how to use these generating functions to establish graph distance statistics, in a similar vein as our Theorem~\ref{thm:dist}. 
We note that the results for maps are more precise, since they provide control on individual distances instead of the distance difference that we address.
It is an interesting open problem to find an expression for the distance-dependent two-point function with control on $d_{\mathsf{X}}(h_1,h_2)$ for doubly-cusped spheres $\mathsf{X} \in \mathcal{M}_{0,2+n}(0,0,\mathbf{L})$.

\paragraph{Cusp-less planar hyperbolic surfaces.}
As discussed above, the geodesic boundaries and cusps of a hyperbolic surface can be viewed as the analogues of, respectively, the faces and vertices of a map. 
However, one aspect of this analogy is broken: while every map contains at least one vertex, not every hyperbolic surface contains a cusp.
Since we require an origin cusp to define the spine, this prevents us from giving a bijective interpretation of the general genus-$0$ Weil-Petersson volume $V_{0,n}(\mathbf{L})$ where each boundary length is allowed to be positive.
Inspired by the works of Bouttier-Guitter-Miermont \cite{Bouttier_Bijective_2022,Bouttier_Enumeration_2024} for maps, the third author will close this gap in \cite{Zonneveld_tree_2025} by extending the spine construction to the ``half-tight'' hyperbolic cylinders introduced in \cite{budd2023topological}. 

\paragraph{Scaling limits.} 
The tree bijection for punctured hyperbolic spheres based on the spine construction was vital for proving the scaling limit towards the Brownian sphere in \cite{budd2025random}, mirroring probabilistic methods well-established for large random planar maps and their tree encodings \cite{LeGall_Uniqueness_2013,Miermont_Brownian_2013}.
This work provides a more general tree bijection that is well suited to address similar scaling limit results for Weil-Petersson random surfaces with boundaries.
As explained in \cite[Section~1.3]{budd2023topological}, it is natural to consider the genus-$g$ \emph{Boltzmann hyperbolic surface} associated to a weight $\mu$ in the form of a Borel measure on $\mathbb{R}_{\geq 0}$ for which the partition function \eqref{eq:partfun} is finite, $F_g[\mu] < \infty$.
It is the hyperbolic surface $\mathsf{X}$ with a random number $n$ of boundaries of random lengths $\mathbf{L} \in \R_{\geq 0}^n$ obtained by first sampling $(n,\mathbf{L})$ with joint density $V_{g,n}(\mathbf{L})/(n! F_g[\mu])\rmd \mu(L_1)\cdots\rmd\mu(L_n)$ and then, conditionally on $(n,\mathbf{L})$, sampling a Weil-Petersson random surface in $\mathcal{M}_{g,n}(\mathbf{L})$.
In the case of genus-$0$ Boltzmann hyperbolic surface $\mathsf{X}$ with a distinguished cusp, Theorem~\ref{thm:introbijection} implies that $\mathsf{Spine}(\mathsf{X})$ is a random decorated tree $\tree$ with a rather explicit law, so one can attempt\footnote{Although explicit, the law of the tree is more complicated than the pure cusp-case of \cite{budd2025random}. However, there is also a significant simplification when $\mu$ admits boundaries of positive length: the random tree naturally features an embedded single-type Galton-Watson process corresponding to genealogy of boundary vertices, whereas \cite{budd2025random} had to resort to introducing certain cutpoints to identify such a process.} the same procedure as in \cite{budd2025random} to establish a scaling limit for the random metric space.
If the weight $\mu$ is not too heavy-tailed and one conditions on the number $n$ of boundaries to be large, one would expect again to obtain the Brownian sphere and the correct scaling constant can be extracted from a calculation similar to Corollary~\ref{cor:distvar}.
However, there is also the potential of witnessing different universality classes if $\mu$ is taken to be heavy-tailed, namely the stable carpets and stable gaskets appearing in the limits of random maps with large face degrees \cite{Curien_Scaling_2025,LeGall_scaling_2010}.
The necessary non-generic criticality criteria on $\mu$ have already been investigated in the works \cite{Castro_Critical_2023,Koster_universality_2022}.

\begin{figure}[h!]
    \centering
    \includegraphics[width=.32\linewidth]{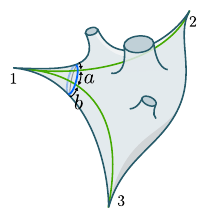}
    \caption{For $\mathsf{X}\in\mathcal{M}_{0,3+n}(0,0,0,\mathbf{L})$, the horocycle log-ratio is given by $H(\mathsf{X}) = \log(b/a)$ where $a$ and $b$ are the lengths of segments of a horocycle around the first cusp determined by the shortest geodesics to the second and third cusp.\label{fig:horocyclelength}}
\end{figure}
\paragraph{Topological recursion.}
We mentioned that Theorem~\ref{thm:dist} should somehow be related to the generalization of Mirzakhani's topological recursion introduced in \cite{budd2023topological}. 
To give a hint why this might be the case, let us discuss a slightly different statistic on triply-cusped hyperbolic spheres.
Let us denote by $\mathcal{M}^{\circ'}_{0,3+n}(0,0,0,\mathbf{L}) \subset \mathcal{M}_{0,3+n}(0,0,0,\mathbf{L})$ the full-measure subset of such surfaces that have unique length-minimizing geodesics from the second and third cusp to the first.
For $\mathsf{X}\in\mathcal{M}^{\circ'}_{0,3+n}(0,0,0,\mathbf{L})$ this pair of geodesics cuts a small horocycle $h_1$ around the first cusp into two horocyclic segments of lengths $a,b>0$, say.
Then we let the \emph{horocycle log-ratio} be $H(\mathsf{X}) = \log(b/a)$, see Figure~\ref{fig:horocyclelength}.
We claim that it has the same distribution as the distance-difference $D(\mathsf{X})$, meaning that we may also identify
\begin{align}
	\hat{X}_n(u;\mathbf{L}) = \int_{\mathcal{M}^{\circ'}_{0,3+n}(0,0,0,\mathbf{L})} e^{2 u H(\mathsf{X})} \rmd\operatorname{Vol}_\WP(\mathsf{X}).\label{eq:logratiostat}
\end{align} 

The reason is that there exists a natural measure-preserving involution (modulo interchanging labels $2$ and $3$) on $\mathcal{M}_{0,3+n}^{\circ'}(0,0,0,\mathbf{L})$ that interchanges the two statistics.
This involution amounts to cutting open the surface $\mathsf{X}$ along the pair of geodesic into an $n$-holed quadrangle, and then regluing the sides in the unique other way to achieve a triply-cusped sphere $\mathsf{X}'$ again.
The cusp formed by the two opposite vertices of the quadrangle receives labeled $1$ and it is not too difficult to see that the gluing seams ending there are the unique shortest geodesics. 
We refer to Figure~\ref{fig:cutgluepants} for a pictorial proof of the fact that $D(\mathsf{X}) = H(\mathsf{X'})$ and vice versa.

\begin{figure}[h!]
    \centering
    \includegraphics[width=\linewidth]{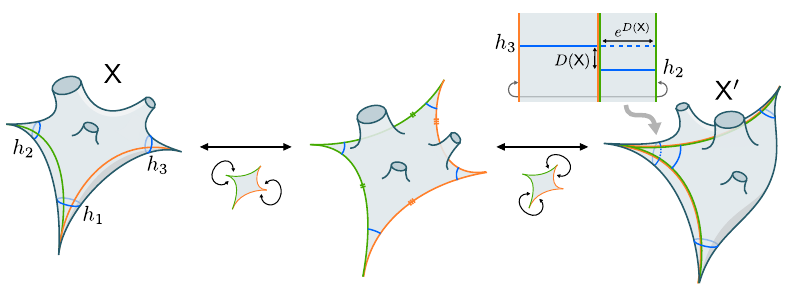}
    \caption{A pictorial proof of $D(\mathsf{X}) = H(\mathsf{X}')$. After cutting and regluing the unit-length horocycles $h_2$ and $h_3$ on $\mathsf{X}$ are mismatched on $\mathsf{X}'$ by a signed hyperbolic distance $D(\mathsf{X})$, as illustrated at the top in the upper-half plane representation. Moving $h_2$ to align with $h_1$ changes its length by a factor $e^{D(\mathsf{X})}$, so that $H(\mathsf{X}') = \log(e^{D(\mathsf{X})}/1)=D(\mathsf{X})$. \label{fig:cutgluepants}}
\end{figure}

One may interpret a triply-cusped sphere as a ``tight'' pair of pants, in the sense of \cite{budd2023topological}, with three boundaries of length $0$.
We note that the power series $\hat{X}(u;\mu]$ provides control on the partitioning of the horocycle around the initial puncture that is reminiscent of the partitioning of the initial boundary underlying the Mirzakhani-McShane identity \cite{Mirzakhani2007}.
In fact, it is possible to turn \eqref{eq:logratiostat} into a statistic for tight pairs of pants by cutting and regluing $\mathsf{X}'$ along the three shortest geodesics connecting the three cusps, similarly to what was done in \cite{Bouttier_Bijective_2022} for the case of maps.
But detailing this is beyond the scope of this work.

\subsection*{Acknowledgments}

This work is part of the VIDI programme with project number VI.Vidi.193.048, which is financed by the Dutch Research Council (NWO).

\section{Tree bijection for hyperbolic surfaces}\label{sec:bijection}

In this section we prove the bijection part of Theorem~\ref{thm:introbijection}.
We do so by establishing a stronger bijective result for the full moduli space, which involves are larger class of combinatorial trees than the ones introduced above.

\subsection{Bijection for the full moduli space}
Denote by $\treeset^{\mathrm{all}}_n$ the collection of (red and white) bicolored plane trees $\tree$ with $n$ white boundary vertices labeled $1,\ldots, n$ and a finite number of red inner vertices of degree at least $3$.
Moreover, each corner of a boundary vertex of $\tree$ is labeled \emph{ideal} or \emph{non-ideal}.
Let $\vsetinner(\tree)$, $\vsetboundary(\tree)$, $\edgeset(\tree)$ are the inner vertices, boundary vertices and edges of $\tree$ as before, but now $\cornerset(\tree)$ is the set of non-ideal corners of $\tree$ only.
The number of non-ideal corners of $\vertex \in \vsetboundary(\tree)$ is denoted $\operatorname{nonid}(\vertex)$.
For $\vertex\in \vsetinner(\tree) \cup \vsetboundary(\tree)$ the oriented edges starting at $\vertex$ are $\vec{\edge}_{\vertex,1}, \ldots, \vec{\edge}_{\vertex,\deg(\vertex)} \in \vec{\edgeset}(\tree)$ in counterclockwise order.   
Similarly, the non-ideal corners of $\bvertex \in \vsetboundary(\tree)$ are $\corner_{\bvertex,1},\ldots,\corner_{\bvertex,\operatorname{nonid}(\bvertex)}$.
For $\mathbf{L} \in \mathbb{R}_{\geq 0}^n$ we let $\treeset^{\mathrm{all}}_n(\mathbf{L}) \subset \treeset^{\mathrm{all}}_{n}$ be the subset of trees such that for each $i=1,\ldots,n$ the boundary vertex $\bvertex$ with label $i$ has exclusively ideal corners if and only if $L_i=0$.
One may check that these definitions are compatible with those in the introduction, in the sense that $\treeset_n(\mathbf{L}) \subset \treeset^{\mathrm{all}}_n(\mathbf{L})$ if for $\tree \in \treeset_n(\mathbf{L})$ we label each corner at the $i$th boundary vertex to be ideal or non-ideal depending on whether $L_i=0$ or $L_i>0$ respectively.

To $\tree\in\treeset^{\mathrm{all}}_n(\mathbf{L})$ we associate the polytope $\mathcal{A}_{\tree}(\mathbf{L})$ as follows.
Recall the notation $\Delta_k^{(y)} \coloneqq \{ \mathbf{x} \in (0,\infty)^{k} : x_1 + \cdots + x_{k} = y \}$ for the open $(k-1)$-dimensional simplex of size $y$.
Then we set
\begin{align}
	\mathcal{A}_{\tree}(\mathbf{L}) \coloneqq &\left\{ \varphi : \vec{\edgeset}(\tree) \to \R_{\geq 0} \middle|\begin{array}{l}
	\varphi(\vec{\edge}_{\bvertex,i}) = 0\text{ for }\bvertex\in\vsetboundary(\tree), 1\leq i\leq\deg(\bvertex),\\
	\varphi(\vec{\edge}_{\vertex,i}) > 0\text{ for }\vertex\in \vsetinner(\tree),  1\leq i\leq\deg(\vertex),\\
	\sum_{i=1}^{\deg(\vertex)}\varphi(\vec{\edge}_{\vertex,i}) = \pi\text{ for }\vertex\in \vsetinner(\tree), \nonumber\\ 
	\varphi(\vec{\edge}) + \varphi(\cev{\edge}) < \pi\text{ for }\edge\in \edgeset(\tree)\end{array}\right\} \\&\times \prod_{\substack{\bvertex\in \vsetboundary(\tree)\\L_{\bvertex}>0}} \left[\Delta_{\deg(\bvertex)}^{(L_{\bvertex}/2)} \times \Delta_{\operatorname{nonid}(\bvertex)}^{(L_{\bvertex}/2)}\right]\times \prod_{\substack{\bvertex\in \vsetboundary(\tree)\\L_{\bvertex}=0}} \Delta^{(1)}_{\deg(\bvertex)}.\label{eq:polytopedef}
\end{align}
It is a polytope of dimension 
\begin{align}
	\dim\mathcal{A}_\tree(\mathbf{L}) &= \sum_{\vertex\in\vsetinner(\tree)} (\deg(\vertex)-1) +  \sum_{\substack{\bvertex\in\vsetboundary(\tree)\\L_{\bvertex}>0}} (\deg(\bvertex)+\operatorname{nonid}(\bvertex)-2) + \sum_{\substack{\bvertex\in\vsetboundary(\tree)\\L_{\bvertex}=0}} (\deg(\bvertex)-1)\nonumber\\
	&= 2n-4 + \sum_{\vertex\in\vsetinner(\tree)} (3-\deg(\vertex)) + \sum_{\substack{\bvertex\in\vsetboundary(\tree)\\L_\bvertex> 0}} (\operatorname{nonid}(\bvertex)-\deg(\bvertex)) + \sum_{\substack{\bvertex\in\vsetboundary(\tree)\\ L_\bvertex=0}} (1-\deg(\bvertex)), \label{eq:dimpolytope}
\end{align}
where we used that $\tree$ is a combinatorial tree and therefore 
\begin{align}
	\sum_{\vertex\in\vsetinner(\tree)} (2-\deg(\vertex)) + \sum_{\bvertex\in\vsetboundary(\tree)} (2-\deg(\bvertex)) = 2.\label{eq:treeidentity}
\end{align}
Since the summands in \eqref{eq:dimpolytope} are all non-positive, it follows that $\dim \mathcal{A}_\tree \leq 2n-4$ with equality if and only if $\tree \in \treeset_n(\mathbf{L})$ belongs to the subclass of trees introduced in Section~\ref{sec:mainresults}, meaning that each inner vertex has degree $3$, each boundary vertex corresponding to a cusp has degree $1$ and each corner at a boundary vertex corresponding to a geodesic boundary is non-ideal.

The main result of this section will be the following.

\begin{theorem}\label{thm:fullbijection}
	For $n\geq 2$ and $L_1,\ldots,L_n \geq 0$, the Bowditch-Epstein-Penner spine constructions determines a bijection 
	\begin{align*}
		\spinebijection : \mathcal{M}_{0,1+n}(0, \mathbf{L}) \to \bigsqcup_{\tree \in \treeset^{\mathrm{all}}_n(\mathbf{L})} \mathcal{A}_{\tree}(\mathbf{L}).
	\end{align*}
\end{theorem}

\subsection{Generalizing the Bowditch-Epstein-Penner spine construction}

Let $\mathsf{X}\in\Mod_{0,n+1}(0,\mathbf{L})$ be a genus-$0$ hyperbolic surface with a distinguished cusp, that we call the \emph{origin}, and $n\geq 2$ additional boundaries of lengths $\mathbf{L} = (L_1,\ldots, L_n)$ (or cusps whenever $L_i=0$).
The universal cover of $\mathsf{X}$ can be represented as a convex domain $P$ in the Poincar\'e upper-half plane $\mathbb{H}$.
Denote by $\Gamma$ the Fuchsian group $\Gamma\subset \mathrm{PSL}(2,\R)$ such that $\mathsf{X} = P/\Gamma$.
It is convenient to consider the extended surface $\check{\mathsf{X}} = \mathbb{H}/\Gamma \supset X$, corresponding to the surface $\mathsf{X}$ in which to each geodesic boundary of positive length we glue a \emph{funnel} with a geodesic boundary of the same length.
The surface $\check{\mathsf{X}}$ thus has a complete hyperbolic metric with a \emph{boundary at infinity} corresponding to each geodesic boundary of $\mathsf{X}$, and $\mathsf{X}$ corresponds to the convex core of $\check{\mathsf{X}}$.
A point on a boundary at infinity is called an \emph{ideal point}, which can be regarded as a common endpoint of limiting parallel infinite geodesics running into the funnel.

\begin{figure}[h!]
    \centering
    \includegraphics[width=.8\linewidth]{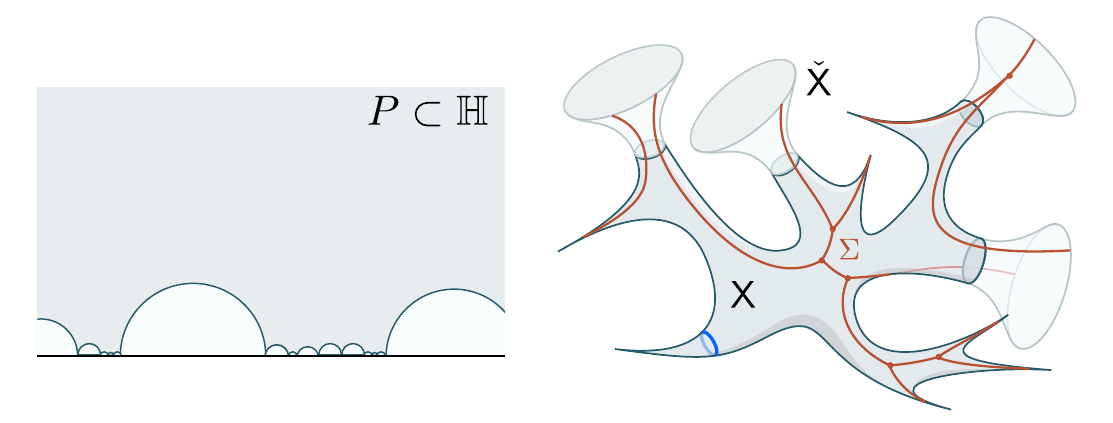}
    \caption{Illustration of the universal cover $P$ of the surface $\mathsf{X}$. Note that a boundary geodesic of $\mathsf{X}$ lifts to a geodesic side of $P$. The surface $\check{\mathsf{X}} = \mathbb{H}/\Gamma$ is extended with funnels. The spine of $\check{\mathsf{X}}$ is illustrated in red.\label{fig:treebijection-1}}
\end{figure}

For any point $x \in \mathsf{X}$ we can make sense of the shortest geodesics from $x$ to the origin as follows.
Let $h$ be a horocycle around the origin that separates $x$ from the cusp and $d_{\mathsf{X}}(x,h)$ the hyperbolic distance between $x$ and $h$.
Let $w(x)$ be the number of distinct geodesics from $x$ to $h$ of length $d_{\mathsf{X}}(x,h)$.
Since each of these geodesics meets $h$ perpendicularly and thus continues into the cusp, the extended geodesics and their number $w(x)$ is independent of the choice of horocycle $h$.

The \emph{spine} $\Sigma$ of $\check{\mathsf{X}}$ is defined as the subset of points with at least two shortest geodesics to the root,
\begin{equation}
    \Sigma = \{ x \in \check{\mathsf{X}} : w(x) \geq 2 \} \subset \check{\mathsf{X}}.
\end{equation}
The points with three or more geodesics will be called \emph{inner vertices} 
\begin{equation}
    V = \{ x\in \check{\mathsf{X}} : w(x) \geq 3\} \subset \Sigma.
\end{equation}
The following lemma is a variant of \cite[Lemma~2.2.1]{Bowditch_Natural_1988} by Bowditch \& Epstein.

\begin{lemma}\label{lem:spineproperties}
    The spine $\Sigma$ satisfies the following properties.
    \begin{enumerate}
        \item $V$ is a finite set.
        \item $\Sigma \setminus V = \{ x\in\check{\mathsf{X}} : w(x) = 2\}$ consists of a finite union of open geodesic arcs. Tracking each arc in each direction it either ends at a point in $V$ or at an ideal point inside a funnel.
        \item Each point in $x\in V$ is the endpoint of exactly $w(x)\geq 3$ arcs.
        \item The boundary at infinity of each boundary or cusp (except the origin) contains the endpoint of at least one arc.
        \item The compactified spine $\overline{\Sigma}$, obtained by joining the endpoints of all arcs ending on the boundary at infinity of a single funnel or cusp, is a tree. 
    \end{enumerate}
\end{lemma}

\begin{proof}
    Let us fix a horocycle $h$ around the origin that is short enough not to self-intersect and to be contained in $\mathsf{X} \subset \check{\mathsf{X}}$ (by the collar lemma, considering the unit-length horocycle suffices).
    Then $h$ does not intersect the spine $\Sigma$ and it lifts to a countable collection $\mathcal{C}$ of disjoint (and non-nested) horocycles in $\mathbb{H}$.
	Let $x\in \check{\mathsf{X}}$ be a point separated from the origin by $h$ and $y\in\mathbb{H}$ a lift of $x$.
    Then $w(x)$ counts the number of horocycles in $\mathcal{C}$ that are closest to $x$, say at distance $r$.
    The number $w(x)$ is necessarily finite, since every hyperbolic disk, in particular the ball of radius $r$ around $y$, can meet only finitely many disjoint horodisks.
    
    Let $\varepsilon>0$ be such that there are still only $w(x)$ horocycles in $\mathcal{C}$ within distance $r+2\varepsilon$ of $y$.
    Then the shape of the spine can be established within an $\varepsilon$-neighborhood of $x$, i.e. $\Sigma \cap \mathsf{Ball}(x,\varepsilon)$.
    This is particularly conveniently seen in the Poincar\'e disk $\mathbb{D}$ when $y$ is positioned at the origin (Fig.~\ref{fig:spinenbhd}).
    \begin{figure}[h!]
        \centering
        \includegraphics[width=\linewidth]{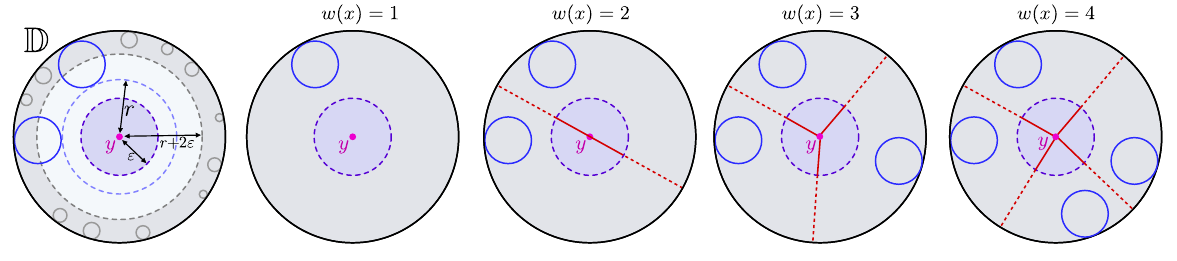}
        \caption{The first figure illustrates the collection $\mathcal{C}$ of disjoint horocycles. The blue ones are at closest distance $r$ from $y$, while the gray ones are at distance at least $r+2\varepsilon$. The other figures illustrate the spine neighbourhoods $\Sigma \cap \mathsf{Ball}(x,\varepsilon)$ for different number $w(x)$.\label{fig:spinenbhd}}
    \end{figure}
	In particular, $V$ is locally finite and $\Sigma\setminus V$ is composed of a locally-finite collection of geodesic arcs, and at least three arcs meet at every point of $V$.

	Let $\mathsf{X}^{\circ} \subset \mathsf{X} \subset \check{\mathsf{X}}$ be the compact surface obtained by removing the interiors of small horocycles around the cusps of $\mathsf{X}$.
	It follows that $V \cap \mathsf{X}^\circ$ is finite and $(\Sigma \setminus V) \cap \mathsf{X}^{\circ}$ consists of finitely many geodesic arcs.
	In particular, there are only finitely many points $\Sigma \subset \partial \mathsf{X}^{\circ}$ on the boundary of $\mathsf{X}^\circ$ with more than one shortest geodesic to $h$.
	Now we examine a single funnel or horocyclic region of $\check{\mathsf{X}} \setminus \mathsf{X}^\circ$, which we can represent in $\mathbb{H}$ with a hyperbolic respectively parabolic generator $g$.
	It follows that the cut-locus in this region is determined by a non-zero finite number (up to translation by $g$) lifts of $h$.
	From this is it follows easily that the cut-locus in this region consists of a finite number of geodesic arcs as well, and that at least one of the arcs ends on an ideal point.
	Combined this proves the first four properties.

    For the complement of the spine $\check{\mathsf{X}} \setminus \Sigma = \{x\in\check{x}: w(x)=1\}$ we obtain a retraction onto the horodisk around the origin, by tracing the unique shortest geodesics to the origin.
    Hence, $\check{\mathsf{X}} \setminus \Sigma$ is topologically an open punctured disk. 
    This in turn implies that the compactified spine $\overline{\Sigma}$ is connected and simply-connected, so must be a tree.
\end{proof}

We can rephrase this lemma as follows.

\begin{lemma}\label{lem:surfacetocombinatorialtree}
Let $\tree$ be the bicolored plane graph with red inner vertices $\vsetinner(\tree) = V$, white boundary vertices $\vsetboundary(\tree)$ corresponding to the $n$ boundaries at infinity, and edge set $\edgeset(\tree)$ determined by the arcs in $\Sigma \setminus V$. 
We designate each corner of a boundary vertex to be ideal respectively non-ideal depending on whether the two adjacent arcs in $\Sigma \setminus V$ are limiting parallel (i.e.\ end on the same ideal point) respectively ultraparallel (i.e.\ end on distinct ideal points).
Then $\tree\in \treeset_n^{\mathrm{all}}(\mathbf{L})$.
\end{lemma}
\begin{proof}
Lemma~\ref{lem:spineproperties} implies $\tree\in \treeset_n^{\mathrm{all}}$.
Moreover, all arcs entering a cusp necessarily end on the same ideal point so the corners of the corresponding boundary vertex are exclusively ideal, and the converse holds as well.
This shows that $\tree\in \treeset_n^{\mathrm{all}}(\mathbf{L})$.
\end{proof}
\subsection{A canonical tiling of $\check{\mathsf{X}}$}

We will now show that, with the help of the spine $\Sigma$, we can canonically tile $\check{\mathsf{X}}$ by hyperbolic triangles and wedges.
For any $x\in \check{\mathsf{X}} \setminus \Sigma$ there exists a unique geodesic from $x$ to the origin that does not intersect $\Sigma$.
Extending this geodesic beyond $x$, either it hits $\Sigma$ for the first time at some point $\mathsf{end}(x) \in \Sigma$ or it is an infinite geodesic disjoint from $\Sigma$ approaching an ideal point $\mathsf{end}(x)$ on a boundary at infinity.
By convention, we further let $\mathsf{end}(x) = x$ for $x\in \Sigma$. 

\begin{figure}[h!]
    \centering
    \includegraphics[width=.4\linewidth]{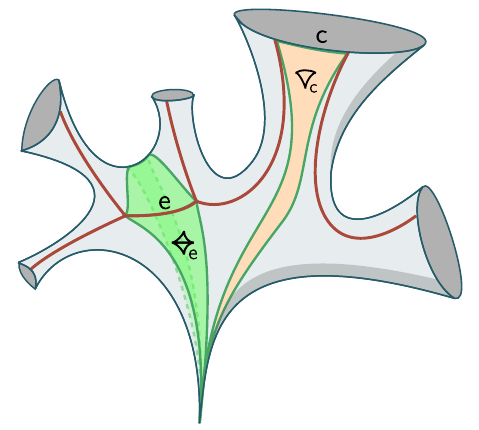}
    \caption{Example of a quadrilateral ${\protect\hypdiamond_\edge}$ associated to an edge $\edge\in\edgeset(\tree)$ and a wedge ${\protect\hypwedge_\corner}$ associated to a non-ideal corner $\corner\in\cornerset(\tree)$. They are bounded by inner ribs and boundary ribs respectively.\label{fig:hyptriangulation}}
\end{figure}

With some abuse of notation we identify an edge $\edge \in \edgeset(\tree)$ with the corresponding arc in $\Sigma \setminus \vsetinner(\tree)$, so that we can associate to it the subset $\hypdiamond_\edge \coloneqq \{ x\in  \check{\mathsf{X}} : \mathsf{end}(x) \in \edge\}$.
It is an (open) hyperbolic quadrilateral in $\check{\mathsf{X}}$ corresponding to the union of $\edge$ and two (open) hyperbolic triangles that share a side along $\edge$, see Figure~\ref{fig:hyptriangulation}.
By definition of the spine, these triangles have the same side lengths (with respect to a chosen horocycle at the origin) and therefore they are congruent via reflection in $\edge$.

If $\corner \in \cornerset(\tree)$ is a non-ideal corner, the two neighbouring edges $\edge$ and $\edge'$ end on distinct ideal points on the boundary at infinity, and therefore we may naturally associate to it an open interval of the boundary at infinity that we denote by $\corner$ as well.
We then consider the subset $\hypwedge_\corner \coloneqq \{ x\in  \check{\mathsf{X}} : \mathsf{end}(x) \in \corner\}$, which has the geometry of an \emph{ideal wedge}, i.e.\ a region isometric to $\{x+iy:0<x<1,y>0\}\subset \mathbb{H}$, see Figure~\ref{fig:hyptriangulation}.

Finally, the geodesics to the origin that bound the wedges $\hypwedge_\corner$ and quadrilaterals $\hypdiamond_\edge$ are called the \emph{ribs} of $\check{\mathsf{X}}$.
We distinguish \emph{inner ribs} and \emph{boundary ribs} depending on whether their endpoint is in $V$ or is an ideal point.
The union of ribs (including the points of $V$) is denoted $\mathsf{Ribs}(\check{\mathsf{X}})$.
We have thus established the following lemma.

\begin{lemma}\label{lem:surfacepartition}
	The surface $\check{\mathsf{X}}$ partitions into 
	\begin{align*}
		\mathsf{Ribs}(\check{\mathsf{X}})\cup\bigcup_{\edge\in \edgeset(\tree)} \hypdiamond_\edge \cup \bigcup_{\corner\in \cornerset(\tree)} \hypwedge_\corner.
	\end{align*}
\end{lemma}

\subsection{Tree labels}

We need to establish which information to record together with the combinatorial tree $\tree$ in order for the hyperbolic surface $\mathsf{X}$ to be uniquely characterized. 
The wedges $(\hypwedge_c)_{c\in\cornerset(\tree)}$ have a unique (intrinsic) geometry, while the pairs of congruent triangles making up the quadrilaterals $(\hypdiamond_\edge)_{\edge\in\edgeset(\tree)}$ are determined by their angles.
For $\edge \in \edgeset(\tree)$ we let $\varphi(\vec{\edge}), \varphi(\cev{\edge}) \in [0, \pi)$ be the angle in either one of the triangles of $\hypdiamond_\edge$ at the corner corresponding to the starting point of $\vec{\edge}$ and $\cev{\edge}$ respectively.

\begin{lemma}\label{lem:anglevar}
	The angles $\varphi : \vec{\edgeset}(\mathfrak{t}) \to [0,\pi)$ satisfy
	\begin{enumerate}[label = (\roman*)]
		\item $\varphi(\vec{\edge}) = 0$ if and only if $\vec{\edge}$ starts at a boundary vertex;
		\item $\varphi(\vec{\edge}) + \varphi(\cev{\edge}) < \pi$ for $\edge\in \edgeset(\tree)$;
		\item $\sum_{j=1}^{\deg(\vertex)} \varphi(\vec{\edge}_{\vertex,j}) = \pi$ for $\vertex \in \vsetinner(\tree)$.
	\end{enumerate} 
\end{lemma}
\begin{proof}
	The first property follows from the fact that the angle $\varphi(\vec{\edge})$ vanishes only when the arc starts at an ideal point, hence when $\vec{\edge}$ starts at a boundary vertex.
	The second property is a consequence of the fact that angles in a hyperbolic triangle add up to less than $\pi$, while the last property stems from the fact that the angles around $v\in \vsetinner(\tree)$ should add up to $2\pi$ (and each angle $\varphi(\vec{\edge}_{\vertex,j})$ contributes twice because the quadrilateral consists of two congruent triangles).
\end{proof}

The combinatorial structure of $\tree$ dictates which sides of the quadrilaterals and wedges are paired at the ribs of $\mathsf{X}$.
The inner ribs are half-infinite so there is no ambiguity in the gluing of the adjacent sides of the quadrilaterals.
The boundary ribs, however, are bi-infinite so there is a shearing degree of freedom in the gluing of the neighboring pair of quadrilaterals or quadrilateral-wedge pair.

\begin{figure}[h!]
    \centering
    \includegraphics[width=.9\linewidth]{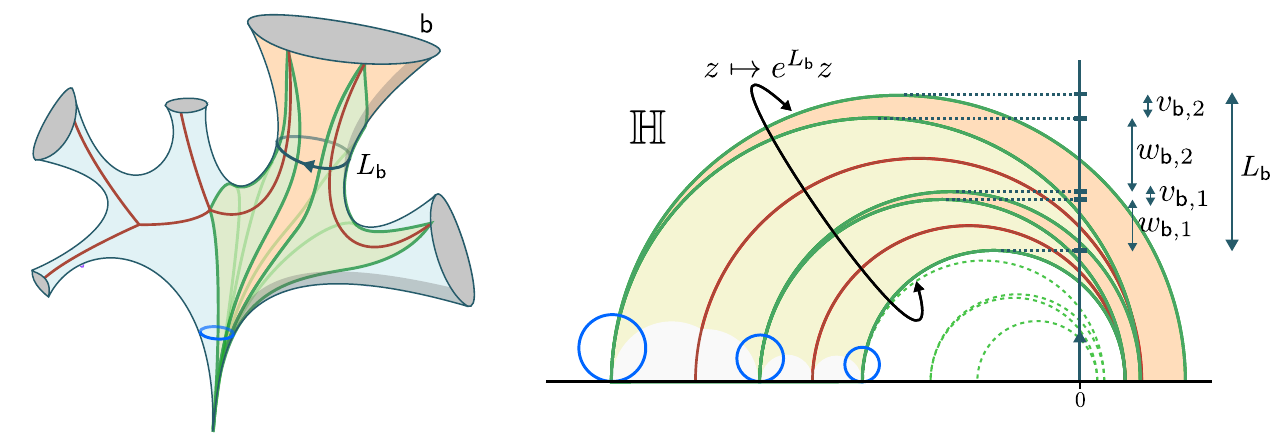}
    \caption{The neighbourhood of a funnel, tiled by hyperbolic quadrangles and wedges, can be conveniently represented in the upper-half plane $\mathbb{H}$ by aligning the closed geodesic with the vertical axis.\label{fig:treebijectionboundary}}
\end{figure}

To fix the gluing unambiguously, let us zoom in on the neighborhood of a funnel of perimeter $L_\bvertex > 0$ associated to a boundary vertex $\bvertex\in \vsetboundary(\tree)$, meaning that $L_\bvertex = L_i$ if $\bvertex$ carries label $i$, see Figure~\ref{fig:treebijectionboundary} for an example where $\deg(\bvertex) = 2$.
By convention we orient the closed geodesic of length $L_\bvertex$, corresponding to the boundary of the original surface $\mathsf{X}$, so that the origin is on the left and the boundary at infinity on the right.
Viewed in the universal cover in the upper-half plane $\mathbb{H}$, we can conjugate the Fuchsian group $\Gamma$ to assure that the geodesic is oriented upwards along the imaginary axis and the associated hyperbolic generator is $z \mapsto e^{L_\bvertex} z$.
The boundary ribs ending at $\bvertex$ lift to a sequence of disjoint semicircles in $\mathbb{H}$ that each intersect the imaginary axis.
For each $j=1,\ldots,\deg(\bvertex)$, if $r' > r'> 0$ are the radii corresponding to the ribs bounding the quadrilateral $\hypdiamond_{\vec{\edge}_{\bvertex,j}}$, we set
\begin{align}
	w_{\bvertex,j} = \log\frac{r'}{r} 
\end{align}
to be the log-ratio.
Similarly, for each $j=1,\ldots,\operatorname{nonid}(\bvertex)$, we let $v_{\bvertex,j} = \log\frac{r'}{r}$ be the log-ratio of the radii $r'>r>0$ associated to the ribs that bound the wedge $\hypwedge_{\corner_{\bvertex,j}}$.
Alternatively, we can associate to each boundary rib a marking on the boundary geodesic by projecting the top of the semicircle horizontally onto the vertical axis. 
These markings give rise to a canonical partition of the boundary into intervals of hyperbolic length $w_{\bvertex,1},\ldots,w_{\bvertex,\deg{\bvertex}}, v_{\bvertex,1}, \ldots, v_{\bvertex,\operatorname{nonid}(\bvertex)}$.

\begin{lemma}\label{lem:bdryvar}
	If $\bvertex \in \vsetboundary(\mathfrak{t})$ is the boundary vertex with label $i$ and $L_\bvertex > 0$, then these labels satisfy 
	\begin{align}
		\sum_{j=1}^{\deg(\bvertex)} w_{\bvertex,j} = \sum_{j=1}^{\operatorname{nonid}(\bvertex)} v_{\bvertex,j} = \frac{L_\bvertex}{2}\label{eq:vwsum},
	\end{align}
	and they uniquely determine the gluing of the quadrilaterals $\hypdiamond_{\vec{\edge}_{\bvertex,j}}$ and wedges $\hypwedge_{\corner_{\bvertex,j}}$ incident to $\bvertex$.
\end{lemma}
\begin{proof}
Pairs of consecutive boundary ribs around $\bvertex$ are asymptotically parallel, on the right (in $\mathbb{H}$) if they bound the same quadrilateral and on the left if they bound the same wedge.
By construction, the radii of two lifts of a rib to $\mathbb{H}$ that are related by $z \mapsto e^{L_\bvertex} z$ differ in radius by a factor
\begin{align*}
	\exp\left(\sum_{j=1}^{\deg(\bvertex)} w_{\bvertex,j} + \sum_{j=1}^{\operatorname{nonid}(\bvertex)} v_{\bvertex,j}\right) = e^{L_\bvertex}.
\end{align*}  
Another relation follows from examining the lifts of the horocycle $h$ around the origin, shown in blue in Figure~\ref{fig:treebijectionboundary}.
Using that the sides of the quadrilateral $\hypdiamond_{\vec{\edge}_{\bvertex,j}}$ incident to $\bvertex$ are of equal length, the lifts of $h$ in $\mathbb{H}$ correspond to circles tangent to the horizontal axis that differ in radius by a factor $e^{2 w_{\bvertex,j}}$.
In particular, the lifts of $h$ at the endpoints of two lifts of a rib related by $z \mapsto e^{L_i} z$ differ in radius by a factor
\begin{align*}
	\exp\left(2\sum_{j=1}^{\deg(\bvertex)} w_{\bvertex,j}\right) = e^{L_\bvertex}
\end{align*}  
as well.
Combining the last two relations proves \eqref{eq:vwsum}.

\begin{figure}[h!]
    \centering
    \includegraphics[width=.6\linewidth]{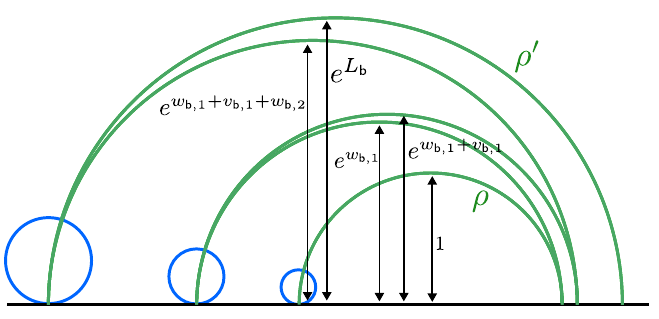}
    \caption{The reconstruction of Figure~\ref{fig:treebijectionboundary} with the help of the boundary parameters $(w_{\bvertex,j})_{j=1}^{\deg(\bvertex)}$ and $(v_{\bvertex,j})_{j=1}^{\operatorname{nonid}(\bvertex)}$. \label{fig:treebijectionboundary-reconstruct}}
\end{figure}

Knowledge of $(w_{\bvertex,j})_{j=1}^{\deg(\bvertex)}$ and $(v_{\bvertex,j})_{j=1}^{\operatorname{nonid}(\bvertex)}$ is sufficient to reconstruct the picture in $\mathbb{H}$ up to scaling (see Figure~\ref{fig:treebijectionboundary-reconstruct}): choosing the first rib $\rho$ as a unit-radius semicircle, the radii and positions of the other (lifts of) ribs are uniquely determined by the log-ratios of the radii and the tangency requirements.
The unique hyperbolic generator that fixes $\infty$ and maps $\rho$ to its sibling $\rho'$ of radius $e^{L_i}$ also fixes a unique vertical geodesic that intersects all ribs.
Translating all ribs horizontally to align this geodesic with the vertical axis, we reproduce the original representation of the neighbourhood of $\bvertex$ in $\mathbb{H}$ up to rescaling.
In particular, it unambiguously fixes the gluing prescription of the quadrilaterals and wedges. 
\end{proof}

The situation in the case $L_\bvertex = 0$ is slightly different. 
As already observed in Lemma~\ref{lem:surfacetocombinatorialtree}, all corners of $\bvertex$ in this case are ideal so the cusp region consists of $\deg(\bvertex)$ hyperbolic quadrilaterals only.
Considering the unit-length horocycle around $\bvertex$, it is then natural to define $w_{\bvertex,j}$ to be the length of horocyclic segment contained in $\hypdiamond_{\vec{e}_{\bvertex,j}}$.  

\begin{lemma}\label{lem:cuspvar}
	If $\bvertex \in \vsetboundary(\mathfrak{t})$ corresponds to a cusp, $L_\bvertex = 0$, then these labels satisfy 
	\begin{align}
		\sum_{j=1}^{\deg(\bvertex)} w_{\bvertex,j} = 1 \label{eq:wcuspsum}
	\end{align}
	and they uniquely determine the gluing of the quadrilaterals $\hypdiamond_{\vec{\edge}_{\bvertex,j}}$ incident to $\bvertex$.
\end{lemma}
\begin{proof}
	Let us consider the neighbourhood of $\bvertex$ in $\mathbb{H}$ with the associated parabolic generator $z \mapsto z+1$ so that the cusp is located at $\infty$, see Figure~\ref{fig:treebijectioncusp}.
	In this case the unit-length horocycle is located at height one and the parameter $w_{\bvertex,j}$ measures the width of the vertical strip corresponding to $\hypdiamond_{\vec{e}_{\bvertex,j}}$.
	Since this clearly amounts to a partition of the unit-interval, \eqref{eq:wcuspsum} follows.
	Moreover, the picture and therefore the gluing prescription of the quadrilaterals is fully characterized by $(w_{\bvertex,j})_{j=1}^{\deg(\bvertex)}$.
	\begin{figure}[h!]
    \centering
    \includegraphics[width=.45\linewidth]{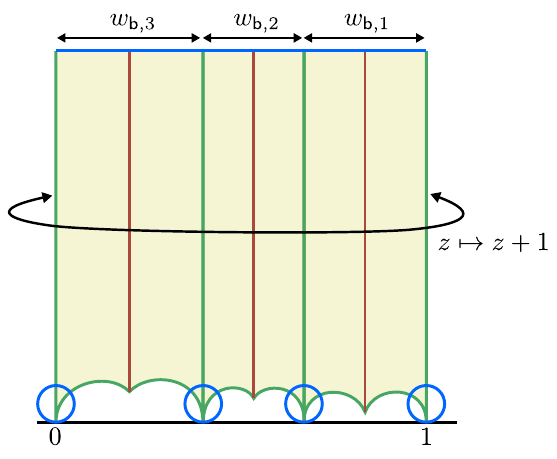}
    \caption{Example of the cusp region in $\mathbb{H}$ for a vertex $\bvertex$ with $L_{\mathsf{b}}=0$ and $\deg(\bvertex) = 3$.\label{fig:treebijectioncusp}}
\end{figure}
\end{proof}

We are now ready to properly define the mapping $\spinebijection$ of Theorem~\ref{thm:fullbijection}.
If $\mathsf{X} \in \mathcal{M}_{0,1+n}(0,\mathbf{L})$, then by Lemma~\ref{lem:surfacetocombinatorialtree} we have that the tree structure of the spine gives a bicolored tree $\tree \in \treeset^{\mathrm{all}}_n(\mathbf{L})$.
Lemma~\ref{lem:anglevar}, Lemma~\ref{lem:bdryvar} and Lemma~\ref{lem:cuspvar} together imply that 
\begin{align}
	\lambda \coloneqq \left( (\varphi(\vec{\edge}))_{\vec{\edge}\in\vec{\edgeset}(\tree)}, \left((w_{\bvertex,j})_{j=1}^{\deg(\bvertex)},(v_{\bvertex,j})_{j=1}^{\operatorname{nonid}(\bvertex)}\right)_{\bvertex\in \vsetboundary(\tree)}\right) \in \polytope(\mathbf{L}).\label{eq:pointinpolytope}
\end{align}
Moreover, each boundary rib is incident to a boundary vertex so, according to Lemma~\ref{lem:bdryvar} and Lemma~\ref{lem:cuspvar}, the gluing of the neighbouring quadrilaterals or quadrilateral-wedge pairs is uniquely determined.
We thus conclude the following.

\begin{proposition}
	For $n \geq 2$ and $\mathbf{L} \in \R_{\geq 0}^n$, associating to $\mathsf{X} \in \mathcal{M}_{0,1+n}(0, \mathbf{L})$ the pair $\spinebijection(\mathsf{X})=(\tree,\lambda)$ consisting of the tree $\tree$ and the point \eqref{eq:pointinpolytope} in the polytope $\polytope(\mathbf{L})$ determines a well-defined injective mapping $\spinebijection : \mathcal{M}_{0,1+n}(0, \mathbf{L}) \to \bigsqcup_{\tree \in \treeset^{\mathrm{all}}_n(\mathbf{L})} \mathcal{A}_{\tree}(\mathbf{L})$. 
\end{proposition}

\subsection{Bijection proof}

With our description of the mapping $\spinebijection$ in place, it is straightforward to formulate the inverse construction.
Given a $\tree \in \treeset_n^{\mathrm{all}}(\mathbf{L})$ and a point $\lambda \in \mathcal{A}_\tree(\mathbf{L})$, let us construct a hyperbolic surface $\mathsf{X} = \mathsf{Glue}(\tree,\lambda)$.
To this end, we introduce the parameters associated to $\lambda$ as in \eqref{eq:pointinpolytope} and associate to each edge $\edge \in \edgeset(\tree)$ the unique quadrilateral $\hypdiamond_\edge$ composed of congruent triangles with angles $(\varphi(\vec{\edge}),\varphi(\cev{\edge}),0)$ and to each non-ideal corner $\corner \in \cornerset(\tree)$ a wedge $\hypwedge_\corner$.
The incidence relations of $\tree$ prescribe which pairs of sides of these constituents should be glued to result in a rib of the surface.
As explained in the previous section, this gluing is unambiguous for the inner ribs, while the shearing for the boundary ribs is determined by the local structure around the boundary vertices $\vsetboundary(\tree)$.    
The result is a complete hyperbolic surface $\check{\mathsf{X}}$ and its convex core (i.e.\ the surface obtained by amputating funnels at their closed geodesic) is a hyperbolic surface with a distinguished cusp and $n$ labeled geodesic boundaries or cusps, see Figure~\ref{fig:glueconstruction}.

\begin{figure}[h!]
    \centering
    \includegraphics[width=.8\linewidth]{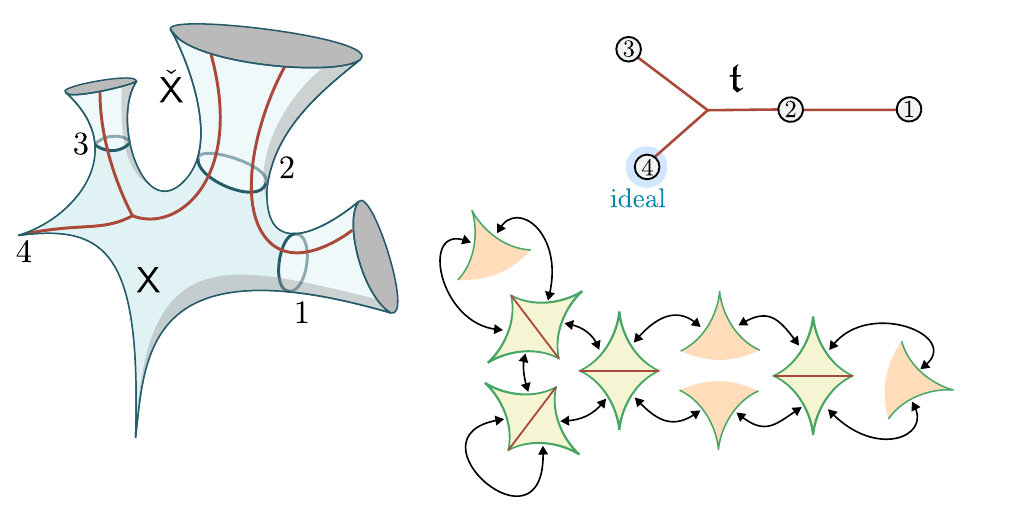}
    \caption{An example of a tree $\tree\in \treeset_n^{\mathrm{all}}(\mathbf{L})$ with $n = 4$ and $L_1,L_2,L_3>0$, $L_4=0$, together with an illustration of the gluing prescription (bottom left) and the resulting surface $\mathsf{X} = \mathsf{Glue}(\tree,\lambda)\in \mathcal{M}_{0,1+n}(\mathbf{L})$. \label{fig:glueconstruction}}
\end{figure}

\begin{lemma}
	This determines a well-defined mapping $\mathsf{Glue} : \bigsqcup_{\tree \in \treeset^{\mathrm{all}}_n(\mathbf{L})} \mathcal{A}_{\tree} \to \mathcal{M}_{0,1+n}(0, \mathbf{L})$.
\end{lemma}
\begin{proof}
	Because of the planarity of the gluing prescription, $\check{\mathsf{X}}$ is clearly a hyperbolic metric on an $(n+1)$-holed sphere.
	We only need to check that the origin corresponds to a cusp and that the boundary with label $i$ is a cusp if $L_i=0$ or a funnel with the appropriate perimeter when $L_i > 0$.
	The latter conditions are granted by the constructions described in Lemma~\ref{lem:bdryvar} and Lemma~\ref{lem:cuspvar} because they apply to any choice of $(w_{\bvertex,j})_{j=1}^{\deg(\bvertex)}\in \Delta_{\deg(\bvertex)}$ and $(v_{\bvertex,j})_{j=1}^{\operatorname{nonid}(\bvertex)}\in \Delta_{\operatorname{nonid}(\bvertex)}$. 

	In order for the origin to be a cusp, it is sufficient that there exists a matching choice of horocycles for all ideal vertices of the quadrilaterals and wedges that will be identified with the origin.
	This can be achieved by selecting an arbitrary edge $\edge\in \edgeset(\tree)$ and endowing the two ideal vertices of the quadrilateral $\hypdiamond_\edge$ with an arbitrary symmetric choice of horocycles.
	Then iteratively we can adjust the horocycles on the remaining quadrilaterals and wedges as follows.
	For each boundary vertex $\bvertex \in \vsetboundary(\tree)$, a choice of symmetric horocycles on a single incident quadrilateral uniquely determines the remaining horocycles (see Lemma~\ref{lem:bdryvar}), which will again be symmetric for each quadrilateral. 
	Similarly, for each inner vertex $\vertex \in \vsetinner(\tree)$ the horocycles of one of the $\deg(v)$ quadrilaterals determines uniquely a choice of symmetric horocycles for all them.
	Since $\tree$ has the structure of a tree, the iteration determines unique matching horocycles everywhere.
\end{proof}

To prove Theorem~\ref{thm:fullbijection} it remains to show that $\mathsf{Glue}$ and $\spinebijection$ are inverses of each other.

\begin{proof}[Proof of Theorem~\ref{thm:fullbijection}]
	That $\mathsf{Glue}\circ \spinebijection(X) = X$ for every $\mathsf{X} \in \mathcal{M}_{0,1+n}(0,\mathbf{L})$ follows directly from the construction of $\mathsf{Glue}$ and our discussion in the last subsection.
	It remains to verify that for $\mathsf{X} = \mathsf{Glue}(\tree,\lambda)$ the spine construction gives $\spinebijection(X) = (\tree, \lambda)$ again.
	In particular, we need to convince ourselves that the subset $\Sigma\subset \check{\mathsf{X}}$ given by the union of diagonals of the quadrilaterals $\hypdiamond_\edge$ for $\edge\in \edgeset(\tree)$ together with the points corresponding to the inner vertices is the cut-locus of $\check{\mathsf{X}}$.
	
	By the gluing prescription, $\check{\mathsf{X}} \setminus \Sigma$ has convex boundary and is topologically a punctured disk. 
	Therefore, every point $x \in \check{\mathsf{X}} \setminus \Sigma$ admits a unique geodesic to the origin contained in $\check{\mathsf{X}} \setminus \Sigma$.
	Moreover, for every $x \in \Sigma$ there are a multiple such geodesics that are contained in $\check{\mathsf{X}} \setminus \Sigma$ except for their endpoint and all of them have the same length.
	Let now $x \in \check{\mathsf{X}}$ and $\gamma$ be a geodesic from $x$ to the origin that intersects $\Sigma$ away from its starting point.
	Then letting $y \in \Sigma$ be the last such intersection, we know that the final segment of $\gamma$ is a geodesic from $y$ to the origin that is contained in $\check{\mathsf{X}} \setminus \Sigma$.
	We can thus find another path from $x$ to the origin of the same length by replacing this final segment by one of the other such geodesics starting at $y$.
	But by construction this path is not geodesic because of the bend at $y$, showing that $\gamma$ is not of minimal length.
	This proves that the geodesics to the origin in $\check{\mathsf{X}} \setminus \Sigma$ are length minimizing, and therefore that $\Sigma$ is the cut-locus. 

	Hence, the combinatorial tree identified in the spine construction $\spinebijection(X)$ is precisely $\tree$, the quadrilaterals $(\hypdiamond_\edge)_{\edge\in \edgeset(\tree)}$ and wedges $(\hypwedge_\corner)_{\corner\in\cornerset(\tree)}$ are exactly the ones appearing in the partition of $\check{\mathsf{X}}$ described in Lemma~\ref{lem:surfacepartition}.
	Since the tree labels are intrinsic to the geometry of the quadrilaterals or the funnel/cusp region, they clearly match $\lambda$. 
\end{proof}

\section{Weil--Petersson Poisson structure and measure}

We start with some background on the Weil-Petersson structures on Teichm\"uller space.
Let $S$ be a fixed surface of genus $0$ with $n+1$ boundaries labeled $0,\ldots,n$.
For $\mathbf{L} \in \R^{n+1}_{\geq 0}$, the \emph{Teichm\"uller space} $\mathcal{T}_{0,n+1}(\mathbf{L})$ consists of pairs $(\mathsf{X},f)$ where $\mathsf{X}$ is a hyperbolic surface with geodesic boundaries and/or cusps and where $f : S \to X$ is an orientation-preserving diffeomorphism such that for each $i=0,\ldots,n$, $f$ shrinks the $i$th boundary to a cusp if $L_i=0$ or maps the boundary to a geodesic boundary of $\mathsf{X}$ of length $L_i > 0$.
Two pairs $(\mathsf{X},f)$ and $(\mathsf{Y},g)$ are regarded equivalent, if there exists an isometry $h : \mathsf{X} \to \mathsf{Y}$ such that $h \circ f : S \to \mathsf{Y}$ is isotopic to $g : S \to \mathsf{Y}$ (allowing boundary reparametrization).

The Teichm\"uller space $\mathcal{T}_{0,n+1}(\mathbf{L})$ comes with a canonical symplectic structure given by the Weil-Petersson 2-form $\omega_\WP$.
It takes on a particularly simple form \cite{Wolpert_symplectic_1983} in Fenchel-Nielsen coordinates associated to a pants decomposition of $S$,
\begin{align}
	\omega_{\WP} = \sum_{\gamma} \rmd \ell_\gamma \wedge \rmd \tau_\gamma
\end{align}
involving the lengths $\ell_\gamma$ and twists $\tau_\gamma$ associated to the closed geodesics $\gamma$ of the pants decomposition.
The associated Weil-Petersson Poisson structure is determined by the bivector
\begin{align}
	\pi_{\WP} = \sum_{\gamma} \frac{\partial}{\partial\ell_\gamma} \wedge \frac{\partial}{\partial\tau_\gamma},\label{eq:poissonwp}
\end{align}
Both the symplectic and the Poisson structure are mapping class group invariant, and therefore descend to moduli space $\mathcal{M}_{0,1+n}(\mathbf{L})$.
However, these expressions are not very useful for our purposes because the Fenchel-Nielsen coordinates are not easily related to the tree decorations.

In \cite{budd2025random} it was observed that the Penner coordinates \cite{Penner_decorated_1987} on $\mathcal{T}_{0,1+n}(\mathbf{0})$, associated to carefully chosen triangulations of the cusped surfaces, are adapted well to the tree decoration.
Although this coordinate system does not generalize easily to the case of geodesic boundaries, the closely related \emph{shear coordinates} do.
As we will explain in the next sections, it is more convenient to work with the Weil-Petersson Poisson structure than with the symplectic structure, because the former has a simple expression in terms of shears on Teichm\"uller space and in terms of the tree decoration $\mathcal{A}_\tree(\mathbf{L})$.

\subsection{Thurston--Fock shear coordinates}

A triangulation $\triangulation$ of $S$ is a maximal set of disjoint arcs starting and ending on the boundaries such that each connected component of the complement is a hexagon with three sides corresponding to arcs interleaved with three sides corresponding to segments of the boundary, see the left-hand side of Figure~\ref{fig:sheartriangulation} for an example.

\begin{figure}[h!]
    \centering
    \includegraphics[width=\linewidth]{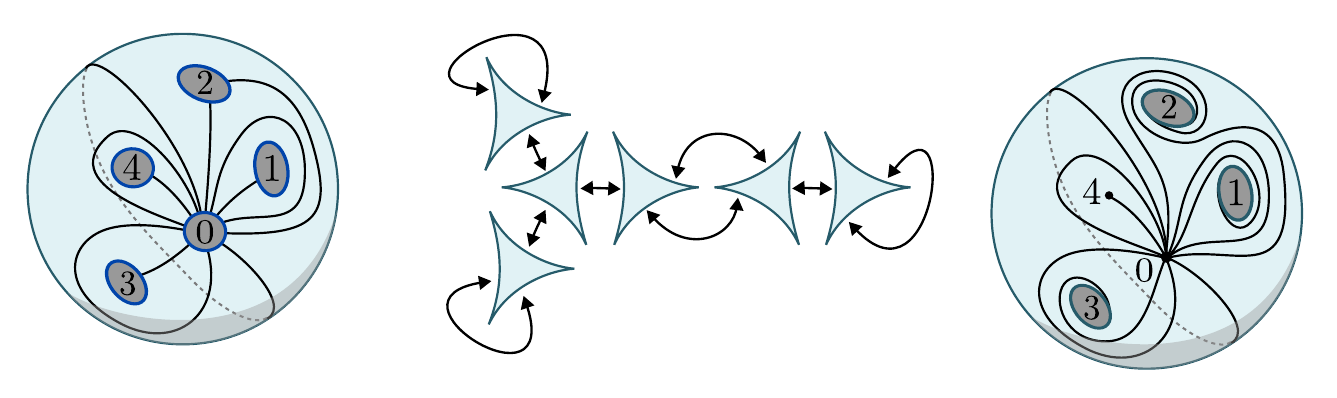}
    \caption{Given a triangulation of $S$ and a hyperbolic surface $\mathsf{X}$, one can consider the ideal triangulation of $\mathsf{X}$ consisting of arcs that spiral into to the geodesic boundaries.\label{fig:sheartriangulation}}
\end{figure}

Given a triangulation $\triangulation$ and a hyperbolic surface $(\mathsf{X},f)\in \mathcal{T}_{0,n+1}(\mathbf{L})$, each arc $\alpha \in \triangulation$ determines an arc $f(\alpha)$ on $\mathsf{X}$ starting and ending on a geodesic boundary or cusp.
The homotopy class of such arcs, where one allows sliding the endpoints along the boundaries, contains a unique geodesic representative $\gamma^\perp_\alpha$ that meets the boundaries perpendicularly at the endpoints (or runs into the cusp).
This gives a canonical tiling of $\mathsf{X}$ by right-angled hyperbolic hexagons (or hyperbolic $k$-gons with $6-k$ ideal vertices and $2k-6$ right angles, for $k=3,4,5$, depending on the number of adjacent cusps).   
But one can also canonically associate to $\alpha$ a geodesic $\gamma^{\mathcal{S}}_\alpha$ on $\mathsf{X}$ that spirals into the boundaries at the endpoints of $\alpha$ in a clockwise fashion (of course no spiralling happens when it ends at a cusp). 
Informally, we think of $\gamma^{\mathcal{S}}_\alpha$ as the limiting geodesic obtained from $\gamma^\perp_\alpha$ by moving both endpoints clockwise along the boundary at constant rate, see the right-hand side of Figure~\ref{fig:sheartriangulation}. 
The result is a tiling of $\mathsf{X}$ by $2n-2$ ideal triangles, see Figure~\ref{fig:pantsexample} for the simple case of $n=2$ corresponding to a pair of pants.

\begin{figure}[h!]
    \centering
    \includegraphics[width=\linewidth]{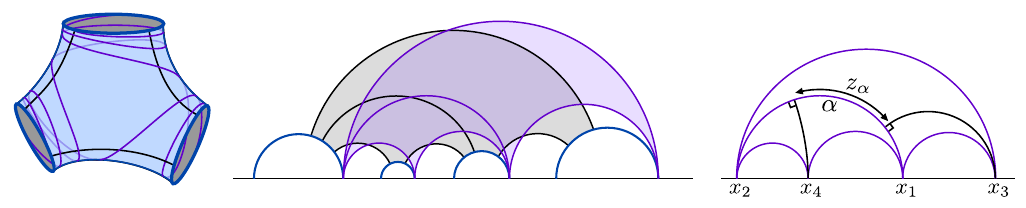}
    \caption{Left: example surface $\mathsf{X}$ in the case $n=2$, corresponding to a pair of pants, with the perpendicular geodesics $\gamma_\alpha^{\perp}$ of a triangulation in black and the spiralling geodesics $\gamma_\alpha^{\mathcal{S}}$ in purple. Middle: the corresponding fundamental domains for $\mathsf{X}$ in $\mathbb{H}$ consisting either of a pair of right-angled hexagons (black) or a pair of ideal triangles (purple). Right: the definition of the shear coordinate as a signed hyperbolic distance along $\alpha$. In this example $z_\alpha < 0$. \label{fig:pantsexample}}
\end{figure}

This ideal triangulation gives rise to the Thurston-Fock shear coordinates on $\mathcal{T}_{0,n+1}(\mathbf{L})$, first introduced by Thurston \cite{Thurston_Minimal_1998} and generalized by Fock \cite{Fock_Description_1993}. 
See \cite[Section~2.4]{Penner_Decorated_2012} for a recent overview.
The shear $z_\alpha$ associated to the arc $\alpha \in \triangulation$ can be defined as follows. 
The geodesic $\gamma^{\mathcal{S}}_\alpha$ on $\mathsf{X}$ is adjacent to a pair of ideal triangles, see right-hand side of Figure~\ref{fig:pantsexample}. 
If $x_1,x_2,x_3,x_4 \in \R\cup \{\infty\}$ are the positions of the ideal vertices of these triangles lifted to $\mathbb{H}$, such that $\alpha$ points from $x_1$ to $x_2$ with $x_3$ on the right and $x_4$ on the left, then the shear is the logarithmic cross ratio
\begin{align}
	z_\alpha = \log\frac{(x_3-x_1)(x_4-x_2)}{(x_3-x_2)(x_1-x_4)}.
\end{align}
Equivalently, $z_\alpha$ can be understood as the signed hyperbolic distance along $\gamma^{\mathcal{S}}_\alpha$ between the orthogonal projections onto $\gamma^{\mathcal{S}}_\alpha$ of the opposite vertices (corresponding to $x_3$ and $x_4$).
By \cite[Theorem~4.4]{Penner_Decorated_2012}, the mapping $(\mathsf{X},f) \mapsto (z_\alpha)_{\alpha\in \triangulation}$ is a real-analytic homeomorphism
\begin{align}
	\mathsf{Shear} : \mathcal{T}_{0,n+1}(\mathbf{L}) \to \mathcal{Z}_{\triangulation}(\mathbf{L}) = \left\{ (z_\alpha)_{\alpha\in \triangulation} \in \R^{3n-3} : \sum_{j} z_{\alpha_{i,j}} = L_i\text{ for }i=0,\ldots,n\right\} \subset \R^{3n-3}, \label{eq:shearcoor}
\end{align}
where for each $i$ the sum runs over the arcs $\alpha_{i,1},\alpha_{i,2},\ldots\in\triangulation$ that start at the $i$th boundary (counting twice the arcs that start and end there).
Since the shears obey $n+1$ independent constraints, $\mathcal{Z}(\mathbf{L})$ has dimension $2n-4$.

We have seen that the Weil-Petersson symplectic form is easily expressed in Fenchel-Nielsen coordinates.
In the case of shear coordinates, however, it is the dual Weil-Petersson Poisson structure that admits a simple expression.
Fock introduced \cite{Fock_Dual_1997} (see also \cite[Definition~4.9]{Penner_Decorated_2012}) the bivector\footnote{Note that the factor of $\frac{1}{2}$ does not appear in \cite{Fock_Dual_1997} or \cite{Penner_Decorated_2012}. Several different conventions are used in the literature for the normalization of the Weil-Petersson symplectic or Poisson structure. Our choice is the one compatible with the symplectic structure $\omega_{\mathrm{WP}} = -2 \sum_{\text{triangles}} \left(\rmd \log \lambda_\alpha \wedge \rmd \log \lambda_\beta +\cdots\right)$ in Penner's lambda-lengths in the case of punctured surfaces, as well as the Weil-Petersson volume computations of Mirzakhani \cite{Mirzakhani2007}. Note for instance that \cite[Theorem~3.3.6]{Penner_Weil_1992} and \cite[Equation~(2.8)]{Chekhov_Fenchel_2020} differ by a factor of $2$.} 
\begin{align}
	\pi_{\triangulation} = \frac{1}{2} \sum_{\text{triangles}} \left(\frac{\partial}{\partial z_{\alpha}} \wedge \frac{\partial}{\partial z_{\beta}} + \frac{\partial}{\partial z_{\beta}} \wedge \frac{\partial}{\partial z_{\gamma}}+ \frac{\partial}{\partial z_{\gamma}} \wedge \frac{\partial}{\partial z_{\alpha}}\right)\label{eq:thurstonfockbivector}
\end{align}
on $\mathcal{Z}_{\triangulation}(\mathbf{L})$, where the sum runs over the triangles of $\triangulation$ and the bounding arcs of each triangle are $\alpha$, $\beta$, $\gamma$ in clockwise order.
It was later \cite{Fock_Quantum_1999,Chekhov_Quantizing_2004} demonstrated to agree with the Weil-Petersson Poisson structure $\pi_{\WP}$ from \eqref{eq:poissonwp}.

\subsection{Relation to the spine construction}\label{sec:spinetoshear}

From here on we restrict to the case $L_0=0$ and $\mathbf{L} = (L_1,\ldots,L_n) \in \R_{\geq 0}^n$. 
We now explain how to associate to any hyperbolic surface $(\mathsf{X},f)\in \mathcal{T}_{0,n+1}(0,\mathbf{L})$ a canonical triangulation of $S$, in a similar spirit to \cite[Section~3]{budd2025random} but using the generalized spine and shear coordinates instead of Penner coordinates.

Recall the construction of the spine $\Sigma \subset \check{\mathsf{X}}$ of Lemma~\ref{lem:spineproperties} and the associated partition of $\check{\mathsf{X}}$ into quadrangles and wedge of Lemma~\ref{lem:surfacepartition}.
For each edge $\edge \in \edgeset(\tree)$ we take an arc in the quadrangle $\hypdiamond_\edge$ crossing $\edge$ and starting and ending at the origin. 
For each corner $\corner$ of a boundary vertex $\bvertex$ we draw an arc from the origin to the boundary of $\mathsf{X}$ corresponding to $\bvertex$: we can take it to be the intersection of $\mathsf{X}$ with the corresponding boundary rib if $\corner$ is ideal, or with any geodesic in the wedge $\hypwedge_\corner$ if $\corner$ is non-ideal. 
By construction all these arcs are disjoint, but the ones of the first type are not necessarily all contained in $\mathsf{X} \subset \check{\mathsf{X}}$.
However, we may shorten them (e.g.\ to the unique geodesic representative in their homotopy classes) without introducing intersections, such that they are contained in $\mathsf{X}$.
If all inner vertices are of degree $3$ then by pulling the arcs back to $S$ via $f^{-1}$ we obtain a triangulation of $S$ that we denote by $\operatorname{\triangulation}(\mathsf{X},f)$.
If $\tree$ has an inner vertex $\vertex \in \vsetinner(\tree)$ of degree $\deg(\vertex) = k > 3$, then the arcs associated to the $k$ quadrangles surrounding $\vertex$ bound a $k$-gon, but we can choose to triangulate this $k$-gon in an arbitrary deterministic fashion depending only on the combinatorics of $\tree$.
Hence, the mapping $(\mathsf{X},f) \to \operatorname{\triangulation}(\mathsf{X},f)$ can be made well-defined on all of $\mathcal{T}_{0,n+1}(0,\mathbf{L})$, in such a way that the combinatorics of the triangulation depends only on the isometry class, meaning that for two different markings $f,g : S\to X$ we have that 
\begin{align}
\operatorname{\triangulation}(X,g) = g^{-1} \circ f (\operatorname{\triangulation}(\mathsf{X},f)).\label{eq:triangMCG}
\end{align}

For every $\tree\in\treeset^{\mathrm{all}}(\mathbf{L})$, one may canonically associate shear coordinates to $\mathcal{A}_\tree(\mathbf{L})$ as follows.
Since $\mathcal{A}_\tree(\mathbf{L})$ is nonempty, we can take a representative $\mathsf{X} \in \mathsf{Glue}(\tree,\mathcal{A}_\tree(\mathbf{L})) \subset \mathcal{M}_{0,1+n}(\mathbf{L})$ and a marking $f: S\to X$ and let $\triangulation_\tree = \operatorname{\triangulation}(\mathsf{X},f)$.
Then Theorem~\ref{thm:fullbijection} implies that $\mathsf{Glue}$ bijectively relates $\mathcal{A}_\tree(\mathbf{L})$ to the set $\operatorname{\triangulation}^{-1}(\triangulation_\tree) \subset \mathcal{T}_{0,n+1}(\mathbf{L})$ of marked surfaces sharing the same canonical triangulation.
Composition with the shear coordinates \eqref{eq:shearcoor} thus gives an injective mapping 
\begin{align}
	\mathsf{Shear}_\tree : \mathcal{A}_\tree(\mathbf{L}) \to \mathcal{Z}_{\triangulation_\tree}(0,\mathbf{L}) \subset \R^{3n-3}.
\end{align}
Thanks to \eqref{eq:triangMCG}, when viewing the shears as indexed by the edges of the combinatorial triangulation, this mapping is independent of the choice of representative $(\mathsf{X},f)$.

\begin{figure}[h!]
    \centering
    \includegraphics[width=.9\linewidth]{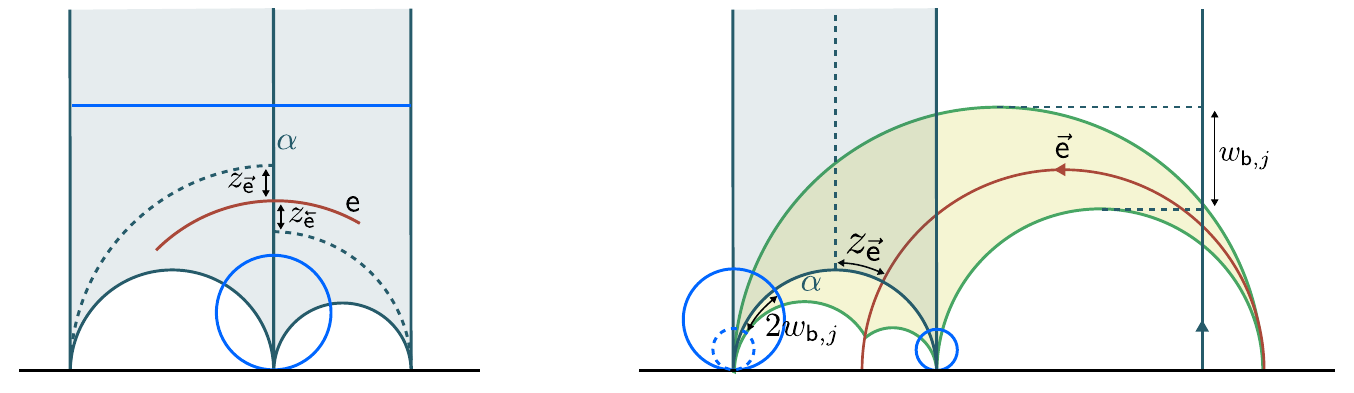}
    \caption{Left: the half-shears measure the (signed) distances from the orthogonal projections onto the midpoint of $\alpha$ (in this illustration they are negative). Right: illustration in $\mathbb{H}$ of the case that $\vec{\edge}$ starts at a boundary vertex. In this case the arc in ${\protect\triangulation_\tree}$ dual to $\edge$ is adjacent to an ideal triangle (in purple) with a vertex at $\infty$. The half-shear $z_{\vec{\edge}}$ is equal to $w_{\bvertex,j}$.\label{fig:shearw}}
\end{figure}

Let us identify the mapping $\mathsf{Shear}_\tree$ explicitly in terms of the coordinates of $\mathcal{A}_\tree(\mathbf{L})$.
For this it is convenient to introduce the \emph{half-shear} $z_{\vec{\edge}}$ associated to an orientation of an edge $\edge\in\edgeset(\tree)$ of the spine. 
If $\alpha$ is the geodesic arc associated to the quadrangle $\hypdiamond_\edge$, then $\alpha$ comes with a canonical marking at its midpoint, measured with respect to the horocycle at the origin. 
In case $\alpha$ and $\edge$ intersect, this corresponds precisely to their intersection point (see left side of Figure~\ref{fig:shearw}).  
We consider the ideal triangle adjacent to $\alpha$ on the side corresponding to starting point of $\vec{\edge}$.
Then the half-shear $z_{\vec{\edge}}$ is the signed hyperbolic distance along $\alpha$ between its midpoint and the orthogonal projection onto $\alpha$ of the opposite vertex of this ideal triangle. 
It should be clear that then the shear $z_\edge$ of $\alpha$ equals the sum of the half-shears of the orientations of $\edge$,
\begin{align}
	z_\edge = z_{\vec{\edge}} + z_{\cev{\edge}}.
\end{align}
The expression for $z_{\vec{\edge}}$ depends on the type of vertex at which $\vec{\edge}$ starts.

If $\vec{\edge}$ starts at a boundary vertex $\bvertex$ with $L_\bvertex >0$, meaning that $\vec{\edge} = \vec{\edge}_{\bvertex,j}$ for some $j$, then we may examine the quadrangle $\hypdiamond_\edge$ in the funnel region as in Figure~\ref{fig:treebijectionboundary}.
The right side of Figure~\ref{fig:shearw} illustrates this quadrangle as well as the associated ideal triangle, which has a vertex at $\infty$ because the two bounding arcs are asymptotically parallel to the boundary geodesic, i.e.\ the vertical axis in this picture.
The orthogonal projection of this vertex onto $\alpha$ is therefore simply the top of the semicircle.
Since the boundary ribs of $\hypdiamond_\edge$ correspond to semicircles whose radii differ by a factor $e^{w_{\bvertex,j}}$,
the horocycle $h$ at the origin lifts to a pair of horocycles that differ in size by a hyperbolic distance $2 w_{\bvertex,j}$.
It follows directly that the midpoint of $\alpha$ is at (signed) distance
\begin{align}
	z_{\vec{\edge}} = w_{\bvertex,j}\label{eq:shearw}
\end{align}
from the top.

If $\vec{\edge}$ starts at a boundary vertex $\bvertex$ corresponding to a cusp, $L_\bvertex = 0$, then the corresponding ideal triangle in $\triangulation_\tree$ aligns with the boundary ribs of $\hypdiamond_\edge$.
Due to symmetry (see Figure \ref{fig:treebijectioncusp}) the half-shear $z_{\vec{\edge}} = 0$ vanishes in this case.

\begin{figure}[h!]
    \centering
    \includegraphics[width=.7\linewidth]{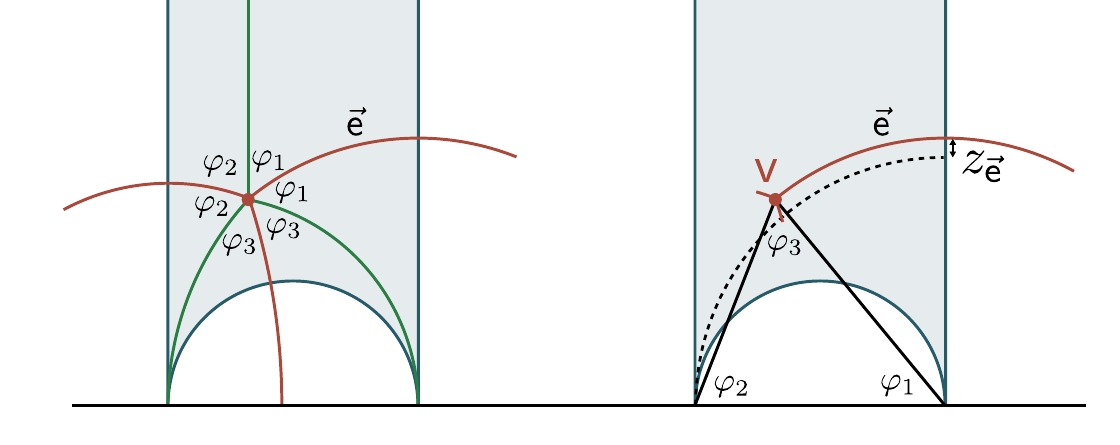}
    \caption{Here we use the abbreviated notation $\varphi_j = \varphi(\vec{\edge}_{\vertex,j})$, $j=1,2,3$.\label{fig:sheartriangle}}
\end{figure}

If $\vec{\edge}$ starts at an inner vertex $\vertex\in \vsetinner(\tree)$ of degree $3$, let us assume for concreteness that $\vec{\edge} = \vec{\edge}_{\vertex,1}$ and examine the angles $\varphi(\vec{\edge}_{\vertex,j})$, $j=1,2,3$, around $\vertex$. 
If we draw the arc dual to $\vec{\edge}$ vertically in $\mathbb{H}$, then the ideal triangle dual to $\vertex$ looks like the one in Figure~\ref{fig:sheartriangle}.
Elementary trigonometry shows that $\varphi(\vec{\edge}_{\vertex,j})$, $j=1,2,3$ are precisely the angles of the Euclidean triangle formed by $\vertex$ and the two ideal vertices on the horizontal axis.
In particular, $e^{z_{\vec{\edge}}}$ is the ratio of lengths of the right side to the bottom side.
Hence, the sine law tells us that
\begin{align}
	z_{\vec{\edge}_{\vertex,1}} = \log\frac{\sin\varphi(\vec{\edge}_{\vertex,2})}{\sin\varphi(\vec{\edge}_{\vertex,3})}.
\end{align}
Of course, this formula holds more generally under cyclic permutation of the indices $1,2,3$. 

\begin{figure}[h!]
    \centering
    \includegraphics[width=.6\linewidth]{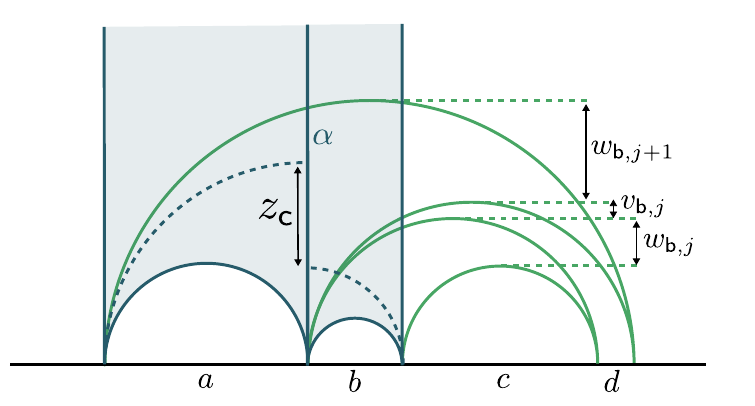}
    \caption{To compute the shear associated to an arc $\alpha$ that spirals into a boundary, we examine the construction of Figure~\ref{fig:treebijectionboundary-reconstruct}.
	The lengths of the segments on the horizontal axis obey \eqref{eq:intervalsforshear}.
	\label{fig:shearv}}
\end{figure}

For each corner $\corner$ of a boundary vertex $\bvertex$ we have a geodesic arc $\alpha$ spiralling into the boundary associated to $\bvertex$.
Assuming $\corner$ is non-ideal, the construction of Lemma~\ref{lem:bdryvar} and Figure~\ref{fig:treebijectionboundary-reconstruct} allows us to establish the geometry of the two neighbouring ideal triangles, see Figure~\ref{fig:shearv}.
The hyperbolic distances $z_\corner$, $w_{\bvertex,j}$, $v_{\bvertex,j}$ and $w_{\bvertex,j+1}$ can easily be expressed in terms of the segment lengths $a,b,c,d$ indicated in the figure as
\begin{align}
	e^{z_\corner} = \frac{b}{a}, \quad e^{w_{\bvertex,j}} = \frac{b+c}{c},\quad e^{v_{\bvertex,j}} = \frac{b+c+d}{b+c},\quad e^{w_{\bvertex,j+1}} = \frac{a+b+c+d}{b+c+d}.\label{eq:intervalsforshear}
\end{align}
Solving for $z_\corner$ gives 
\begin{align}
	z_\corner = -v_{\bvertex,j} + \log \frac{1-e^{-w_{\bvertex,j}}}{e^{w_{\bvertex,j+1}}-1}.\label{eq:shearv}
\end{align}
This formula remains valid if $\corner$ is ideal as long as we admit that $v_{\bvertex,j} = 0$ by convention.
In case $L_{\bvertex} = 0$, we can refer to Figure~\ref{fig:treebijectioncusp} to see that $z_\corner = \log(w_{\bvertex,j}/w_{\bvertex,j+1})$.

\begin{lemma}
	$\mathsf{Shear}_\tree$ is real-analytic for each $\tree\in\treeset^{\mathrm{all}}_n(\mathbf{L})$. Its image is of full dimension in $\mathcal{Z}_{\triangulation_\tree}(0,\mathbf{L})$ if and only if $\tree \in \treeset_n(\mathbf{L})$.
\end{lemma}
\begin{proof}
	Assuming all inner vertices are of degree $3$, we have provided explicit formulas above for the shear coordinates comprising $\mathsf{Shear}_\tree$, which are analytic by inspection.
	In case of an inner vertex $\vertex$ of degree $k>3$, the associated geodesic arcs form a triangulated ideal $k$-gon.
	It is straightforward to check that the shears of the diagonals and the half-shears of the sides of this $k$-gon are analytic functions of the $k$ angles $\varphi(\vec{\edge}_{\vertex,j})$, $j=1,\ldots,k$.
	We omit the details, because we won't need the formulas in the following.

	Recall that $\dim \mathcal{Z}(0,\mathbf{L}) = 2n-4$.
	Comparing with the dimension of $\mathcal{A}_\tree(\mathbf{L})$ given in \eqref{eq:dimpolytope}, we see that $\dim \mathcal{Z}_{\triangulation_\tree}(0,\mathbf{L}) = \dim \mathcal{A}_\tree(\mathbf{L})$ if and only if $\tree \in \treeset_n(\mathbf{L})$.
	Together with injectivity and analyticity, this verifies the claim.
\end{proof}

\begin{proposition}\label{prop:fullwpmeasure}
The subset $\mathcal{M}_{0,1+n}^\circ(0,\mathbf{L}) \coloneqq \mathsf{Glue}(\bigsqcup_{\tree\in \treeset_n(\mathbf{L})}\mathcal{A}_{\tree}(\mathbf{L})) \subset \mathcal{M}_{0,1+n}(0,\mathbf{L})$ is of full Weil-Petersson measure.
\end{proposition}
\begin{proof}
	It follows from the previous lemma, that the complement $\mathcal{M}_{0,1+n}(0,\mathbf{L}) \setminus \mathcal{M}_{0,1+n}^\circ(0,\mathbf{L})$ is covered by finitely many real-analytic hypersurfaces in $\mathcal{T}_{0,n+1}(0,\mathbf{L})$.
	These have vanishing Weil-Petersson measure.
\end{proof}

From here on we will restrict to $\tree \in \treeset_n(\mathbf{L})$.
This means that all inner vertices are of degree $3$ and all corners of those boundary vertices for which $L_i > 0$ are non-ideal.
In this case we choose the ordering of edges and corners around a boundary vertex $\bvertex$ such that for $j=1,\ldots,\deg(\bvertex)$ the corner $\corner_{\bvertex,j}$ sits between the edges $\vec{\edge}_{\bvertex,j}$ and $\vec{\edge}_{\bvertex,j+1}$ (where we use the convention $w_{\bvertex,\deg(\bvertex)+1} = w_{\bvertex,1}$).

\begin{figure}[h!]
    \centering
    \includegraphics[width=\linewidth]{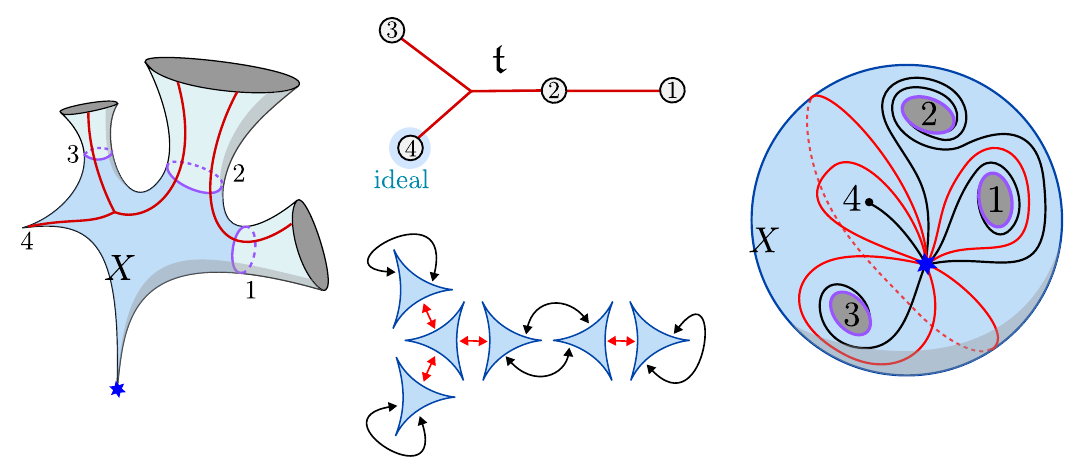}
    \caption{The tree $\tree$ gives rise to a canonical ideal triangulation of $\mathsf{X}$ that is illustrated on the right, where edges of the triangulation that are dual to edges of $\tree$ are colored in red.\label{fig:spineshears}}
\end{figure}

\subsection{Weil--Petersson Poisson structure on the tree decoration}

For $\tree \in \treeset_n(\mathbf{L})$ we introduce the bivector $\pi_\tree \in \bigwedge^2 \R^{6n-6}$ as 
\begin{equation}\label{eq:polytopebivector}
	\pi_\tree = \frac{1}{2}\sum_{\vertex\in\vsetinner(\tree)} \sum_{j=1}^3 \frac{\partial}{\partial \varphi(\vec{\edge}_{\vertex,j})}\wedge \frac{\partial}{\partial \varphi(\vec{\edge}_{\vertex,j+1})} + \frac{1}{2}\sum_{\bvertex \in \vsetboundary(\tree)} \sum_{j=1}^{\deg(\bvertex)} \left(\frac{\partial}{\partial w_{\bvertex,j}}\wedge \frac{\partial}{\partial v_{\bvertex,j}} + \frac{\partial}{\partial v_{\bvertex,j}} \wedge \frac{\partial}{\partial w_{\bvertex,j+1}}\right),
\end{equation}
where it is understood that $\vec{\edge}_{\vertex,4} = \vec{\edge}_{\vertex,1}$ and $w_{\bvertex,\deg(\bvertex)+1} = w_{\bvertex,1}$.
We may equally think of $\pi_\tree \in \bigwedge^2 T\R^{6n-6}$ as a constant bivector field on the manifold $\R^{6n-6}$.

\begin{lemma}
	For each $\tree\in\treeset_{n}(\mathbf{L})$, the bivector $\pi_\tree$ defines a Poisson structure $\{ f, g\}_{\tree} = \pi_\tree(\rmd f,\rmd g)$ on $\R^{6n-6}$ with symplectic leaves of dimension $2n-4$. The polytope $\mathcal{A}_\tree(\mathbf{L})$ of \eqref{eq:polytopedef} is (an open subset of) such a symplectic leaf and thus comes with a corresponding symplectic form $\omega_\tree$.
\end{lemma}
\begin{proof}
	The bracket $\{ f, g\}_{\tree}$ of a constant bivector field automatically satisfies the Jacobi identity and therefore determines a Poisson structure.
	It clearly satisfies 
	\begin{align*}
		\forall_{\bvertex\in \vsetboundary(\tree)}\forall_{1\leq j\leq \deg(\bvertex)} \,\left\{ \varphi(\vec{\edge}_{\bvertex,j}), \,\cdot\, \right\}_{\tree} &= 0, \quad &\forall_{\vertex\in \vsetinner(\tree)} \,\left\{ \sum_{j=1}^3\varphi(\vec{\edge}_{\vertex,j}), \,\cdot\, \right\}_{\tree} &= 0, \\
		\forall_{\bvertex\in \vsetboundary(\tree)} \,\left\{ \sum_{j=1}^{\deg(\bvertex)}v_{\bvertex,j}, \,\cdot\, \right\}_{\tree} &= 0, \quad &\forall_{\bvertex\in \vsetboundary(\tree)} \,\left\{ \sum_{j=1}^{\deg(\bvertex)}w_{\bvertex,j}, \,\cdot\, \right\}_{\tree} &= 0,
	\end{align*}
	and it is easily checked that these span the kernel of $\pi_\tree$ (seen as a mapping $\pi_\tree^\sharp: T^*\R^{6n-6} \to T \R^{6n-6}$).
	From the definition \eqref{eq:polytopedef}, we observe that $\mathcal{A}_\tree(\mathbf{L})$ corresponds to the non-empty open subset $\{ \varphi(\vec{\edge}) + \varphi(\cev{\edge}) < \pi\text{ for }\edge \in \edgeset(\tree),\, v_{\bvertex,j}, w_{\bvertex, j} > 0\text{ for }\bvertex\in\vsetboundary(\tree), j = 1, \ldots,\deg(\bvertex)\}$ of a symplectic leaf (namely a translation of the image of $\pi_\tree^\sharp$).
	The corresponding symplectic form on $\mathcal{A}_\tree(\mathbf{L})$ is related to $\pi_\tree$ via $\omega_\tree(\pi_\tree^\sharp \alpha,\pi_\tree^\sharp \beta) = - \pi_\tree(\alpha,\beta)$, see e.g.\ \cite[Prop.~2.24]{Crainic_Lectures_2021}.
\end{proof}

We are primarily interested in the volume form on $\mathcal{A}_\tree(\mathbf{L})$ induced by its symplectic structure.
Luckily, it can be computed directly from the bivector $\pi_\tree$.

\begin{lemma}\label{lem:volumeformpolytope}
	Up to an irrelevant sign the Liouville volume form $\mu_\tree = \pm \omega_\tree^{n-2} / (n-2)!$ on $\mathcal{A}_\tree(\mathbf{L})$ is equal to
	\begin{align*}
		\mu_\tree = 2^{n-2}\bigwedge_{\vertex\in \vsetinner(\tree)} \left(\rmd \varphi(\vec{\edge}_{\vertex,1}) \wedge\rmd \varphi(\vec{\edge}_{\vertex,2})\right)\wedge \bigwedge_{\bvertex\in \vsetboundary(\tree)} \bigwedge_{j=1}^{\deg(\bvertex)-1}\left(\rmd v_{\bvertex,j}\wedge\rmd w_{\bvertex,j}\right).
	\end{align*}
\end{lemma}
\begin{proof}
	Clearly $(\varphi(\vec{\edge}_{\vertex,1}),\varphi(\vec{\edge}_{\vertex,2}))_{\vertex\in \vsetinner(\tree)},((v_{\bvertex,j},w_{\bvertex,j})_{j=1}^{\deg(\bvertex)-1})_{\bvertex \in \vsetboundary(\tree)}$ constitutes a proper coordinate chart on $\mathcal{A}_\tree(\mathbf{L})$.
	In these coordinates, the non-degenerate Poisson bivector on $\mathcal{A}_\tree(\mathbf{L})$ is
	\begin{align*}
		\pi_\tree|_{\mathcal{A}_\tree(\mathbf{L})} = \frac{1}{2}\sum_{\vertex\in\vsetinner(\tree)} \frac{\partial}{\partial \varphi(\vec{\edge}_{\vertex,1})}\wedge \frac{\partial}{\partial \varphi(\vec{\edge}_{\vertex,2})} + \frac{1}{2}\sum_{\bvertex \in \vsetboundary(\tree)} \sum_{j=1}^{\deg(\bvertex)-1} \left(\frac{\partial}{\partial w_{\bvertex,j}}\wedge \frac{\partial}{\partial v_{\bvertex,j}} + \frac{\partial}{\partial v_{\bvertex,j}} \wedge \frac{\partial}{\partial w_{\bvertex,j+1}}\right).
	\end{align*}
	From the simple block structure we see that the top exterior power $(\pi_\tree|_{\mathcal{A}_\tree(\mathbf{L})})^{n-2} / (n-2)! \in \bigwedge^{2n-4} T \mathcal{A}_\tree(\mathbf{L})$ is simply
	\begin{align*}
		\frac{(\pi_\tree|_{\mathcal{A}_\tree(\mathbf{L})})^{n-2}}{(n-2)!} = 2^{2-n}\bigwedge_{\vertex\in \vsetinner(\tree)} \left(\frac{\partial}{\partial \varphi(\vec{\edge}_{\vertex,1})} \wedge\frac{\partial}{\partial \varphi(\vec{\edge}_{\vertex,2})}\right)\wedge \bigwedge_{\bvertex\in \vsetboundary(\tree)} \bigwedge_{j=1}^{\deg(\bvertex)-1}\left(\frac{\partial}{\partial v_{\bvertex,j}}\wedge\frac{\partial}{\partial w_{\bvertex,j}}\right).
	\end{align*}
	Since $\pi_\tree$ and $\omega_\tree$ are constant in these coordinates, the natural pairing of $\bigwedge^{2n-4} T \mathcal{A}_\tree(\mathbf{L})$ and $\bigwedge^{2n-4} T^* \mathcal{A}_\tree(\mathbf{L})$ shows that $\mu_\tree$ is $2^{n-2}$ times the standard $(2n-4)$-dimensional volume form in these coordinates. 
\end{proof}

\begin{proposition}\label{prop:poissonisomorphism}
	For each $\tree\in\treeset_{n}(\mathbf{L})$, the mapping $\mathsf{Glue}|_{\mathcal{A}_\tree(\mathbf{L})}$ is a Poisson isomorphism onto its image $\mathsf{Glue}(\tree,\mathcal{A}_\tree(\mathbf{L}))$ equipped with the Weil-Petersson Poisson structure.
\end{proposition}
\begin{proof}
	It suffices to compute the Poisson brackets $\{ \,\cdot\,,\,\cdot\,\}_\tree$ of the components of $\mathsf{Shear}_\tree$ and check that they match with the bracket $\{ \,\cdot\,,\,\cdot\,\}_{\triangulation}$ associated to \eqref{eq:thurstonfockbivector} in Thurston-Fock coordinates.
	Since \eqref{eq:polytopebivector} is local to the vertices of $\tree$, it is clear that $\{ z_{\edge},z_{\edge'}\}_\tree = 0$ unless the edges $\edge,\edge' \in \edgeset(\tree)$ share an inner vertex.
	Assuming they do and they correspond to $\vec{\edge}_{\vertex,1}$ and $\vec{\edge}_{\vertex,2}$, we compute
	\begin{align}
		\{ z_{\edge},z_{\edge'}\}_\tree &= \{ z_{\vec{\edge}_{\vertex,1}},z_{\vec{\edge}_{\vertex,2}}\}_\tree = \left\{ \log\frac{\sin\varphi(\vec{\edge}_{\vertex,2})}{\sin\varphi(\vec{\edge}_{\vertex,3})}, \log\frac{\sin\varphi(\vec{\edge}_{\vertex,3})}{\sin\varphi(\vec{\edge}_{\vertex,1})}\right\}_\tree \\
		&= \frac{1}{2}\left(\cot\varphi(\vec{\edge}_{\vertex,1})\cot\varphi(\vec{\edge}_{\vertex,2})+\cot\varphi(\vec{\edge}_{\vertex,2})\cot\varphi(\vec{\edge}_{\vertex,3})+\cot\varphi(\vec{\edge}_{\vertex,3})\cot\varphi(\vec{\edge}_{\vertex,1})\right) \label{eq:pre-Hermite} \\
		&= \frac{1}{2} \label{eq:post-Hermite} = \{ z_{\edge},z_{\edge'}\}_{\triangulation},
	\end{align}
	where the step from \eqref{eq:pre-Hermite} to \eqref{eq:post-Hermite} is due to a version of Hermite's cotangent identity, since $\sum_{j=1}^3 \varphi(\vec{\edge}_{\vertex,j}) = \pi$.

	For an edge $\edge\in\edgeset(\tree)$ and a corner $\corner\in\cornerset(\tree)$, we observe from the structure of \eqref{eq:polytopebivector} and the expressions \eqref{eq:shearw} and \eqref{eq:shearv} that $\{ z_{\edge},z_{\corner}\}_\tree = 0$ unless $\edge$ is adjacent to $\corner$.
	If they correspond to $\vec{\edge}_{\bvertex,j}$ and $\corner_{\bvertex,j}$ respectively, then 
	\begin{align*}
		\{ z_{\edge},z_{\corner}\}_\tree &= \{ z_{\vec{\edge}_{\bvertex,j}},z_{\corner}\}_\tree = \left\{ w_{\bvertex,j},-v_{\bvertex,j} + \log \frac{1-e^{-w_{\bvertex,j}}}{e^{w_{\bvertex,j+1}}-1}\right\}_\tree = -\frac{1}{2} = \{ z_{\edge},z_{\corner}\}_{\triangulation}, 
	\end{align*}
	In case they correspond to $\vec{\edge}_{\bvertex,j+1}$ and $\corner_{\bvertex,j}$, we instead have 
	\begin{align*}
		\{ z_{\edge},z_{\corner}\}_\tree &= \{ z_{\vec{\edge}_{\bvertex,j}},z_{\corner}\}_\tree = \left\{ w_{\bvertex,j+1},-v_{\bvertex,j} + \log \frac{1-e^{-w_{\bvertex,j}}}{e^{w_{\bvertex,j+1}}-1}\right\}_\tree = \frac{1}{2} = \{ z_{\edge},z_{\corner}\}_{\triangulation}.
	\end{align*}
	
	Finally, if $\corner,\corner' \in \cornerset(\tree)$ are two corners, then clearly $\{ z_{\corner},z_{\corner'}\}_\tree = 0$ unless $\corner$ and $\corner'$ are adjacent at the same boundary vertex $\bvertex$.
	If they correspond to $\corner_{\bvertex,j}, \corner_{\bvertex,j+1}$, and $\deg(\bvertex) > 2$, then
	\begin{align*}
		\{ z_{\corner},z_{\corner'}\}_\tree &= \left\{-v_{\bvertex,j} + \log \frac{1-e^{-w_{\bvertex,j}}}{e^{w_{\bvertex,j+1}}-1},-v_{\bvertex,j+1} + \log \frac{1-e^{-w_{\bvertex,j+1}}}{e^{w_{\bvertex,j+2}}-1}\right\}_\tree \\
		&= -\frac{1}{e^{w_{\bvertex,j+1}}-1} \{ v_{\bvertex,j}, w_{\bvertex,j+1}\}_\tree + \frac{e^{w_{\bvertex,j+1}}}{e^{w_{\bvertex,j+1}}-1} \{ w_{\bvertex,j+1}, v_{\bvertex,j+1}\}_\tree \\
		&= \frac{1}{2} = \{ z_{\corner},z_{\corner'}\}_{\triangulation}.
	\end{align*}
	If $\deg(\bvertex)=2$ then both Poisson brackets vanish due to the constraint $z_\corner + z_{\corner'} = L_{\bvertex}/2$.
	This finishes the proof.
\end{proof}

Finally we have all the ingredients for our main theorem.
\begin{proof}[Proof of Theorem~\ref{thm:introbijection}]
	Proposition~\ref{prop:fullwpmeasure} gives the existence of the subset $\mathcal{M}_{0,1+n}^\circ(0,\mathbf{L}) \subset \mathcal{M}_{0,1+n}(0,\mathbf{L})$ of full measure and Theorem~\ref{thm:fullbijection} provides the bijection restricted to $\mathcal{M}_{0,1+n}^\circ(0,\mathbf{L})$. 
	It is a Poisson isomorphism by Proposition~\ref{prop:poissonisomorphism}, so the Weil-Petersson volume form on $\mathcal{M}_{0,1+n}^\circ(0,\mathbf{L})$ agrees with the pull-back along $\mathsf{Spine}$ of the Liouville volume form on $\mathcal{A}_\tree(\mathbf{L})$.
	According to Lemma~\ref{lem:volumeformpolytope} the latter agrees with the (rescaled) Lebesgue measure \eqref{eq:measurepolytope}. 
\end{proof}

\section{Computation of Weil--Petersson volumes}

As an application of our tree bijection we turn to the computation of Weil-Petersson volumes.
After dealing with the case of fixed $n$ we will turn attention to the generating functions.

\subsection{Polytope volumes via inclusion-exclusion}\label{sec:inclexcl}

As announced in the introduction, the main difficulty in computing the volumes $\operatorname{Vol}_\tree(\mathcal{A}_\tree(\mathbf{L}))$ comes from the Delaunay condition $\varphi(\vec{\edge}) + \varphi(\cev{\edge}) > \pi$, which is there to ensure that the edge $\edge$ of the spine has positive length.
In the volume computation of \cite[Section~4.3]{budd2025random} this was addressed, roughly speaking, by considering the generating function of volumes associated to rooted trees and introducing a ``catalytic variable'' to control the angle $\varphi(\vec{\edge})$ at the root.
Here we follow a complementary approach via inclusion-exclusion to reverse the Delaunay inequality.

If $\tree\in\treeset_n(\mathbf{L})$ and $\mathsf{A} \subset \edgeset(\tree)$ is a subset of the edges of $\tree$, that we call \emph{anti-Delaunay edges}, then we introduce the new polytope
\begin{align}
	\mathcal{A}^{\mathrm{anti}}_{\tree,\mathsf{A}}(\mathbf{L}) \coloneqq \left\{ \varphi : \vec{\edgeset}(\tree) \to \R_{\geq 0} \middle|\begin{array}{l}
	\varphi(\vec{\edge}_{\bvertex,i}) = 0\text{ for }\bvertex\in\vsetboundary(\tree), 1\leq i\leq\deg(\bvertex),\\
	\varphi(\vec{\edge}_{\vertex,i}) > 0\text{ for }\vertex\in \vsetinner(\tree),  1\leq i\leq 3,\\
	\varphi(\vec{\edge}_{\vertex,1})+\varphi(\vec{\edge}_{\vertex,2})+\varphi(\vec{\edge}_{\vertex,3}) = \pi\text{ for }\vertex\in \vsetinner(\tree), \\ 
	\varphi(\vec{\edge}) + \varphi(\cev{\edge}) > \pi\text{ for }\edge\in \mathsf{A}\end{array}\right\} \times \prod_{\bvertex\in \vsetboundary(\tree)} \left(\Delta_{\deg(\bvertex)}(L_\bvertex/2)\right)^2.\label{eq:antidelaunaypolytope}
\end{align}
Note that it differs from $\mathcal{A}_{\tree}(\mathbf{L})$ only in the last condition on the angles, in particular there is no condition for the remaining edges $\edgeset(\tree)\setminus \mathsf{A}$.
Observe also that an edge incident to a boundary vertex cannot obey the anti-Delaunay condition, so from here on we assume that the edges $\mathsf{A}$ have their endpoints at the inner vertices $\vsetinner(\tree)$ so that $\mathcal{A}^{\mathrm{anti}}_{\tree,\mathsf{A}}(\mathbf{L})$ is non-empty.
It is naturally equipped with same volume measure $\operatorname{Vol}_\tree$ defined in \eqref{eq:measurepolytope}.

Since the polytopes $\mathcal{A}^{\mathrm{anti}}_{\tree,\mathsf{A}}(\mathbf{L})$ have finite volume for all $\mathsf{A}$, an application of the inclusion-exclusion principle immediately yields the relation 
\begin{align}
 \operatorname{Vol}_\tree(\mathcal{A}_\tree(\mathbf{L})) = \sum_{\mathsf{A} \subset \edgeset(\tree)} (-1)^{|\mathsf{A}|} \operatorname{Vol}_\tree(\mathcal{A}^{\mathrm{anti}}_{\tree,\mathsf{A}}(\mathbf{L})).\label{eq:inclusionexclusion}
\end{align}

Recall from Section~\ref{sec:mainresults} the set $\widetilde{\treeset}_n(\mathbf{L}) \supset \treeset_n(\mathbf{L})$ of bicolored trees whose inner vertices are allowed to have degree $3$ or larger.
Given $\tree$ and $\mathsf{A}$, we define the \emph{anti-Delaunay tree} $\tilde{\tree}\in \widetilde{\treeset}_n(\mathbf{L})$ to be the combinatorial tree obtained from $\tree$ by contracting the edges $\mathsf{A}$.
To each inner vertex $\vertex\in\vsetinner(\tilde{\tree})$ of $\tilde{\tree}$ one can associate a subtree $\tree_\vertex$ of $\tree$ including all edges of $\tree$ incident to one of the vertices that was merged into $\vertex$.
For concreteness one may treat the leafs of $\tree_\vertex$ as if they correspond to univalent boundary vertices, e.g.~cusps.
See Figure~\ref{fig:antidelaunaytree} for an example.

\begin{figure}[h!]
    \centering
    \includegraphics[width=.9\linewidth]{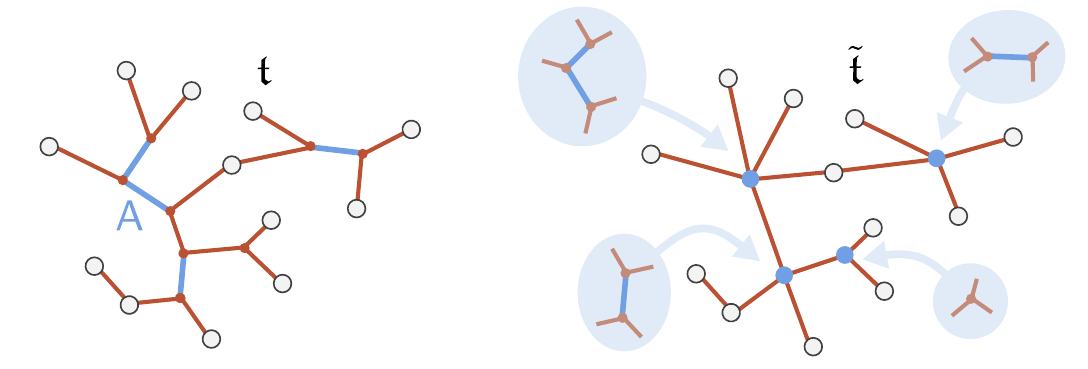}
    \caption{An example of a tree $\tree\in\treeset_n(\mathbf{L})$ with a choice $\mathsf{A}\subset\edgeset(\tree)$ of anti-Delaunay edges and the corresponding anti-Delaunay tree $\tilde{\tree}\in\widetilde{\treeset}_n(\mathbf{L})$.\label{fig:antidelaunaytree}}
\end{figure}

\begin{lemma}\label{lem:antifactorization}
	The polytope $\mathcal{A}_{\tree,\mathsf{A}}^{\mathrm{anti}}(\mathbf{L})$ factors into a Cartesian product of polytopes associated to the (boundary and inner) vertices of the anti-Delaunay tree $\tilde{\tree}$,
	\begin{align}
		\mathcal{A}_{\tree,\mathsf{A}}^{\mathrm{anti}}(\mathbf{L})\cong \prod_{\vertex\in\vsetinner(\tilde{\tree})} \mathcal{A}_\vertex \prod_{\bvertex\in\vsetboundary(\tilde{\tree})} \mathcal{A}_\bvertex.
	\end{align}
	The volume factors correspondingly as  
	\begin{align}
		\operatorname{Vol}_\tree(\mathcal{A}^{\mathrm{anti}}_{\tree,\mathsf{A}}(\mathbf{L})) = 2^{n-2}\prod_{\vertex\in\vsetinner(\tilde{\tree})}\frac{\pi^{2\deg(\vertex)-4}}{(2\deg(\vertex)-4)!} \prod_{\bvertex\in\vsetboundary(\tilde{\tree})} \frac{(L_\bvertex/2)^{2\deg(\bvertex)-2}}{((\deg(\bvertex)-1)!)^2}.\label{eq:antivolume}
	\end{align}
	In particular, it depends only on the degrees in $\tilde{\tree}$ and not on the structure of the trees $(\tree_\vertex)_{\vertex\in\vsetinner(\tilde{\tree})}$.
\end{lemma}
\begin{proof}
	Since the only conditions in \eqref{eq:antidelaunaypolytope} that are not local to a vertex of $\tree$ are the anti-Delaunay conditions for the edges $\mathsf{A}$, the factorization should be clear for the contracted tree $\tilde{\tree}$.
	Since $\mathcal{A}_\bvertex = \left(\Delta_{\deg(\bvertex)}(L_\bvertex/2)\right)^2$ is a product of two simplices, its Euclidean volume is
	\begin{align*}
		\left(\frac{(L_\bvertex/2)^{\deg(\bvertex)-1}}{(\deg(\bvertex)-1)!}\right)^2.
	\end{align*}

	For $\vertex\in\vsetinner(\tilde{\tree})$, say of degree $\deg(\vertex)=k+2$, we claim that $\mathcal{A}_\vertex \cong \Delta_{2k+1}^{(\pi)}$ is also a simplex of dimension $2k-2$ after a (volume-preserving) change of variables.
	Note that $\tree_\vertex$ has $k$ inner vertices and $k+2$ leaves.
	Let us designate one of the leafs as the root and consider all oriented edge $\vec{\edge}_0,\ldots,\vec{\edge}_{2k}$ of $\tree_\vertex$ that point away from the root, where by convention $\vec{\edge}_0$ starts at the root.
	The anti-Delaunay condition turns into the requirement that $\varphi(\vec{\edge}_i) > \varphi(\vec{\edge}_j) + \varphi(\vec{\edge}_k)$ whenever $\vec{\edge}_j$ and $\vec{\edge}_k$ are the two children of $\vec{\edge}_i$ for $i>0$, keeping in mind the natural genealogical interpretation of the binary tree $\tree_\vertex$.
	This makes sense for the root as well ($i=0$) if we use the convention $\varphi(\vec{\edge}_0) = \pi$.
	It motivates us to introduce $\widetilde{\varphi}(\vec{\edge}_i)>0$ for $i=0, \ldots, 2k$ to be the angle $\varphi(\vec{\edge}_i)$ minus the sum of the angles of its children, if any.
	Then it is easy to see that $(\widetilde{\varphi}(\vec{\edge}_i))_{i=0}^{2k} \in \Delta_{2k+1}^{(\pi)}$, see Figure~\ref{fig:antidelaunay}.

	\begin{figure}[h!]
    \centering
    \includegraphics[width=.9\linewidth]{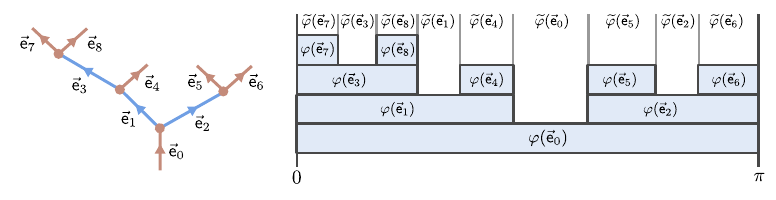}
    \caption{Example of a subtree $\tree_\vertex$ for $\deg(\vertex)=6$ (so $k=4$) rooted at the bottom. The diagram on the right illustrates the partition $\sum_{i=0}^{2k}\widetilde{\varphi}(\vec{\edge}_i) = \pi$. \label{fig:antidelaunay}}
	\end{figure}

	The volume measure (without the power of $2$) on $\mathcal{A}_\vertex$ is the Lebesgue measure in both sets of angles,
	\begin{align*}
		\rmd\varphi(\vec{\edge}_1)\cdots\rmd\varphi(\vec{\edge}_{2k}) = \rmd\widetilde{\varphi}(\vec{\edge}_1)\cdots\rmd\widetilde{\varphi}(\vec{\edge}_{2k}).
	\end{align*}
	Hence $\mathcal{A}_\vertex$ contributes a volume $\pi^{2k}/(2k)!$.
	Taking a product over all vertices of $\tilde{\tree}$ and incorporating the conventional normalization $2^{n-2}$ gives the claimed volume \eqref{eq:antivolume}.
\end{proof}

It remains to sum over $\tree$ and $\mathsf{A}$ to obtain the Weil-Petersson volumes of Theorem~\ref{thm:antidelaunay}.

\begin{proof}[Proof of Theorem~\ref{thm:antidelaunay}]
Combining Corollary~\ref{cor:volpolynomial} and the inclusion-exclusion sum \eqref{eq:inclusionexclusion}, we have
\begin{align}
	V_{0,1+n}(0,\mathbf{L}) = \sum_{\tree\in\treeset_n(\mathbf{L})} \sum_{\mathsf{A}\subset \edgeset(\tree)} (-1)^{|\mathsf{A}|} \operatorname{Vol}_\tree(\mathcal{A}^{\mathrm{anti}}_{\tree,\mathsf{A}}(\mathbf{L})).
\end{align} 
By the previous lemma the volume on the right-hand side depends only on the contracted tree $\tilde{\tree}$.
Counting the pairs $(\tree,\mathsf{A})$ resulting in the same such tree $\tilde{\tree}$ is straightforward: for each $\vertex\in \vsetinner(\tilde{\tree})$ of degree $\deg(\vertex)=k+2$ the subtree $\tree_{\vertex}$ can be any of the Catalan number $\binom{2k}{k}/(k+1)$ of binary trees with $k$ inner vertices. 
Using also that $|\mathsf{A}| = \sum_{\vertex\in\vsetinner(\tilde{\tree})} (\deg(\vertex)-3)$, plugging in \eqref{eq:antivolume} leads to the expression
\begin{align*}
	V_{0,1+n}(0,\mathbf{L}) &= 2^{n-2}\sum_{\tilde{\tree}\in\widetilde{\treeset}_n(\mathbf{L})} \prod_{\vertex\in\vsetinner(\tilde{\tree})}\frac{(-1)^{\deg(\vertex)-3}\pi^{2\deg(\vertex)-4}}{(\deg(\vertex)-1)!(\deg(\vertex)-2)!} \prod_{\bvertex\in\vsetboundary(\tilde{\tree})} \frac{(L_\bvertex/2)^{2\deg(\bvertex)-2}}{((\deg(\bvertex)-1)!)^2}\\
	&= \sum_{\tilde{\tree}\in\widetilde{\treeset}_n(\mathbf{L})} \prod_{\vertex\in\vsetinner(\tilde{\tree})}\frac{2^{\deg(\vertex)-2}(-1)^{\deg(\vertex)-3}\pi^{2\deg(\vertex)-4}}{(\deg(\vertex)-1)!(\deg(\vertex)-2)!} \prod_{\bvertex\in\vsetboundary(\tilde{\tree})} \frac{2^{\deg(\bvertex)-1}(L_\bvertex/2)^{2\deg(\bvertex)-2}}{((\deg(\bvertex)-1)!)^2},
\end{align*} 
where we used the combinatorial identity \eqref{eq:treeidentity} in the second equality.
Recalling the definitions
\begin{align*}
	\gamma_k=\frac{(-1)^k\pi^{2k-2}}{(k-1)!}, \quad t_k(L)=\frac{2}{k!}\left(\frac{L}{2}\right)^{2k},
\end{align*}
this corresponds precisely the claimed formula
\begin{align*}
	V_{0,1+n}(0,\mathbf{L})
	&=\sum_{\tilde{\tree}\in\widetilde{\treeset}_n(\mathbf{L})} \prod_{\vertex\in\vsetinner(\tilde{\tree})} \frac{2^{\deg(\vertex)-2}}{(\deg(\vertex)-1)!}\,\gamma_{\deg(\vertex)-1}\prod_{\bvertex \in \vsetboundary(\tilde{\tree})} \frac{2^{\deg(\bvertex)-2}}{(\deg(\bvertex)-1)!}\,t_{\deg(\bvertex)-1}(L_\bvertex).
\end{align*}
\end{proof}

We emphasize that the relation with decorated anti-Delaunay trees involves inclusion-exclusion, and therefore we loose the strictly bijective correspondence. 
Effectively we have related the Weil--Petersson measure on $\mathcal{M}_{0,1+n}(0,\mathbf{L})$ to a signed measure on 
\begin{align}
	\bigsqcup_{(\tilde{\tree},\mathsf{A})} \mathcal{A}^{\mathrm{anti}}_{\tilde{\tree},\mathsf{A}}(\mathbf{L}),
\end{align}
that is easier to work.
Besides volume computations, we can use it for integration against the  Weil--Petersson measure as long as the integrand extends appropriately to trees with anti-Delaunay edges. 
In Section \ref{sec:distances}, we will see that this is the case for the distance statistic $D$.

\subsection{The string equation}\label{sec:stringeq}

We start by introducing the generating functions of Weil-Petersson volumes a bit more carefully than we did in the introduction.
Following \cite{budd2023topological}, we introduce a weight $\mu\in\R[L^2]^*$ as a linear functional on the ring of real even polynomials.
Then one may define the partition function of genus-$g$ Weil--Petersson volumes as
\begin{align}
	F_g[\mu]=\sum_{n\geq 0}\frac{\mu^{\otimes n}(V_{g,n})}{n!},
\end{align}
where $\mu^{\otimes n}\in(\R[L^2]^*)^{\otimes n}\cong \R[L_1^2,\ldots,L_n^2]^*$ is defined by $\mu^{\otimes n}(f_1(L_1)\dots f_n(L_n))=\mu(f_1)\dots\mu(f_n)$, extended for general even polynomials by linearity.
We often use the physicist notation
\begin{align}
	\mu^{\otimes n}(f)=\int f(\mathbf{L})\, \rmd{\mu(L_1)}\dots\rmd{\mu(L_n)},
\end{align}
which can be interpreted literally whenever $\mu$ arises from a measure on $\R$ with moments of arbitrary order.
In this notation we retrieve the formula \eqref{eq:partfun} given in the introduction,
\begin{align}
	F_g[\mu]=\sum_{n\geq3}\frac{1}{n!}\int V_{g,n}(\mathbf{L})\, \rmd{\mu(L_1)}\dots\rmd{\mu(L_n)}.
\end{align}

Note that $F_g[\mu]$ only depends on the even moments $\int L^{2k}\rmd{\mu(L)}$, so we can formulate $F_g$ as a function of the times $t_k[\mu]$ defined as
\begin{align}
	t_k[\mu]=\int t_k(L) \rmd{\mu(L)}= \int \frac{2}{k!}\left(\frac{L}{2}\right)^{2k} \rmd{\mu(L)},
\end{align}
with $t_k(L)$ as in Theorem~\ref{thm:antidelaunay}.
To avoid discussion about convergence, we interpret $F_g$ as a multivariate formal power series, either in\footnote{In this interpretation, we will often drop the argument $[\mu]$ in our notation.} $(t_0[\mu],t_1[\mu],\ldots)$ or in $\mu$. 
The latter can be made more precise by fixing $\mu$ and interpreting $F_g[x\mu]$ as a univariate formal power series in $x$.

Denoting the delta measure at $L$ by $\delta_L$, we can define the functional derivative $\frac{\delta}{\delta\mu(L)}$ on such power series as 
\begin{align}
	\frac{\delta}{\delta\mu(L)}P[\mu]=	\left.\frac{\partial}{\partial x} P[\mu+x\delta_L]\right|_{x=0},
\end{align}
which gives 
\begin{align}
	\frac{\delta}{\delta\mu(L)}F_g[\mu]=\sum_{n\geq0}\frac{1}{n!}\int V_{g,1+n}(L,\mathbf{L})\, \rmd{\mu(L_1)}\dots\rmd{\mu(L_n)},
\end{align}
where we used the fact that $V_{g,n}$ is symmetric.
We can relate this functional derivative to the formal power series in the times by
\begin{align}
	\frac{\delta}{\delta\mu(L)}=\sum_{k=0}^{\infty}\frac{2}{k!}\left(\frac{L}{2}\right)^{2k}\frac{\partial}{\partial t_k}.
\end{align}

Since we are dealing with genus $0$ and require an origin cusp for the tree bijection, we will be mostly interested in 
\begin{align}
	\frac{\delta}{\delta\mu(0)}F_0[\mu]=2 \frac{\partial}{\partial t_0}F_0(t_0,t_1,\ldots)=\sum_{n\geq2}\frac{1}{n!}\int V_{0,n+1}(0,\mathbf{L})\, \rmd{\mu(L_1)}\dots\rmd{\mu(L_n)}.
\end{align}
Theorem~\ref{thm:antidelaunay} then implies
\begin{align}\label{eq:WP_generating_from_tree}
	\frac{\delta}{\delta\mu(0)}F_0[\mu]=\sum_{n\geq2}\frac{1}{n!}\sum_{\tilde{\tree}\in\widetilde{\treeset}_n} \prod_{\vertex\in\vsetinner(\tilde{\tree})} \frac{2^{\deg(\vertex)-2}}{(\deg(\vertex)-1)!}\,\gamma_{\deg(\vertex)-1}\prod_{\bvertex \in \vsetboundary(\tilde{\tree})} \frac{2^{\deg(\bvertex)-2}}{(\deg(\bvertex)-1)!}\,t_{\deg(\bvertex)-1}[\mu].
\end{align}

In order to exhibit an equation for the generating function, it is convenient to distinguish another cusp, besides the origin, that we may use to root the tree.
We therefore consider the generating function $R[\mu]$ of the twice-punctured sphere, defined by
\begin{align}
	R[\mu]=\frac{\delta^2}{\delta\mu(0)^2}F_0[\mu]=\sum_{n\geq1}\frac{1}{n!}\int V_{0,n+2}(0,0,\mathbf{L})\, \rmd{\mu(L_1)}\dots\rmd{\mu(L_n)}.
\end{align}
It follows from \eqref{eq:WP_generating_from_tree} that
\begin{align}
	R[\mu]&=\sum_{n\geq1}\frac{1}{n!}\sum_{\tilde{\tree}\in\widetilde{\treeset}_{1+n}} \prod_{\vertex\in\vsetinner(\tilde{\tree})} \frac{2^{\deg(\vertex)-2}}{(\deg(\vertex)-1)!}\,\gamma_{\deg(\vertex)-1}\prod_{\substack{\bvertex \in \vsetboundary(\tilde{\tree})\\\bvertex\neq\text{root}}} \frac{2^{\deg(\bvertex)-2}}{(\deg(\bvertex)-1)!}\,t_{\deg(\bvertex)-1}[\mu]\\
	&=\sum_{\tilde{\tree}\in\widetilde{\treeset}^{\text{unlab}}} \prod_{\vertex\in\vsetinner(\tilde{\tree})} \frac{2^{\deg(\vertex)-2}}{(\deg(\vertex)-1)!}\,\gamma_{\deg(\vertex)-1}\prod_{\substack{\bvertex \in \vsetboundary(\tilde{\tree})\\\bvertex\neq\text{root}}} \frac{2^{\deg(\bvertex)-2}}{(\deg(\bvertex)-1)!}\,t_{\deg(\bvertex)-1}[\mu],
\end{align}
where $\widetilde{\treeset}^{\text{unlab}}$ is the set of bicolored plane trees with inner vertices of degree at least $3$ and white vertices (unlabeled this time) of any degree and rooted on a white leaf.

\begin{proof}[Proof of Corollary~\ref{cor:string}]
Since $\tilde{\tree}\in \widetilde{\treeset}^{\text{unlab}}$ admits a bijective decomposition at the vertex adjacent to the root, which may be an inner vertex of degree $d+1 \geq 3$ or a boundary vertex with degree $d+1\geq 1$, we immediately deduce the string equation
\begin{align}
	R[\mu]&= \sum_{d \geq 2} \frac{2^{d-1}}{d!} \gamma_{d}\, R[\mu]^{d} + \sum_{d \geq 0} \frac{2^{d-1}}{d!} t_{d}[\mu] R[\mu]^{d}.
\end{align}
To see that is solves $Z(R[\mu];\mu] = 0$, it suffices to expand the Bessel functions,
\begin{align*}
	Z(r;\mu] &= \frac{\sqrt{r}}{\sqrt{2}\pi}J_1\left(2\pi\sqrt{2r}\right)-\int I_0\left(L\sqrt{2r}\right)\dd{\mu(L)} \\
	&= \sum_{d\geq 1} \frac{2^{d-1}}{d!} (-1)^{d+1} \frac{\pi^{2d-2}}{(d-1)!}r^d - \sum_{d\geq 0} \frac{2^{d-1}}{d!}  \int  \frac{2}{d!} \left(\frac{L}{2}\right)^{2d} \rmd \mu(L)\\
	&= r - \sum_{d \geq 2} \frac{2^{d-1}}{d!} \gamma_{d}\, r^{d} - \sum_{d \geq 0} \frac{2^{d-1}}{d!} t_{d}[\mu] r^{d}.
\end{align*}
\end{proof}

\section{Geodesic distances in triply cusped spheres}\label{sec:distances}

So far we have used the tree bijection to compute the Weil--Petersson volume of $\mathcal{M}_{0,n+1}(0,\mathbf{L})$, which amounts to integrating the constant function against the measure $\operatorname{Vol}_{\WP}$. 
More generally, one may try to integrate any statistic $f : \mathcal{M}_{0,n+1}(0,\mathbf{L})\to \R$ against $\operatorname{Vol}_{\WP}$ by considering the corresponding function
\begin{align}
	f \circ \mathsf{Glue}(\tree, \cdot) : \mathcal{A}_\tree(\mathbf{L}) \to \R
\end{align}
on the polytope $\mathcal{A}_\tree(\mathbf{L})$ for $\tree\in \treeset_n(\mathbf{L})$.
Moreover, if $f \circ \mathsf{Glue}(\tree, \cdot)$ extends to an integrable function $f_\tree : \mathcal{A}_{\tree,\emptyset}^{\mathrm{anti}}(\mathbf{L}) \to\R$ on the larger polytope without Delaunay conditions on the angles, then the inclusion-exclusion principle of Section~\ref{sec:inclexcl} shows that
\begin{align}
	\int_{\mathcal{M}_{0,1+n}(0,\mathbf{L})} f(\mathsf{X}) \rmd \operatorname{Vol}_{\WP}(\mathsf{X}) = \sum_{\tree\in \treeset_n(\mathbf{L})} \sum_{\mathsf{A}\subset \edgeset(\tree)} (-1)^{|\mathsf{A}|} \int_{\mathcal{A}_{\tree,\mathsf{A}}^{\mathrm{anti}}(\mathbf{L})} f_\tree(\lambda) \rmd \operatorname{Vol}_\tree(\lambda),\label{eq:integrationantidelaunay}
\end{align}
where $\operatorname{Vol}_\tree$ is the properly normalized Lebesgue measure of Lemma~\ref{lem:volumeformpolytope}.

In this section we focus on the distance difference $D : \mathcal{M}_{0,3+n}(0,0,0,\mathbf{L})\to \R$. 
Recall from the introduction that $D(\mathsf{X})$ is defined as $D(\mathsf{X}) = d_{\mathsf{X}}(h_0,h_1) - d_{\mathsf{X}}(h_0,h_2)$ in terms of the unit-length horocycles $h_0,h_1,h_2$ around the first three cusps, where we have changed the indices to be compatible with the labeling of the tree $\tree$.
We will first make sure $D \circ \mathsf{Glue}(\tree, \cdot)$ can be reasonably expressed in terms of the decoration of the spine $\tree \in\treeset_{2+n}(0,0,\mathbf{L})$.
Note that in this case $h_0$ is the horocycle around the origin and $h_1, h_2$ are horocycles associated to the boundary vertices with labels $1$ and $2$ of $\tree$. 

\subsection{Distances in terms of tree decoration}

Given an edge $\edge\in\edgeset(\tree)$ and a point $x \in \mathsf{X}$ on this edge, we let $\alpha \in (0,\pi/2]$ be the angle between $\edge$ and any of the two shortest geodesic from $x$ to the origin.
Then we define the \emph{edge distance} to be
\begin{align}
	d_\edge = d_{\mathsf{X}}(x,h_0) + \log \sin \alpha.\label{eq:edgedist}
\end{align}
It follows from hyperbolic trigonometry that this is independent of the choice of the point $x$, and that $d_\edge$ is half the length of the unique geodesic arc starting and ending at the origin dual to $\edge$.
In case $\edge$ intersects this arc, $d_\edge = d_{\mathsf{X}}(\edge,h_0)$.

If $\edge \in\edgeset(\tree)$ is incident to a boundary vertex corresponding to a cusp with unit-length horocycle $h_i$, then it is easily verified that its edge distance agrees with the distance from $h_i$ to the origin,
\begin{align}
	d_\edge=d_{\mathsf{X}}(h_i,h_0).
\end{align} 
Therefore, if we let $\edge_1, \ldots, \edge_k$ for some $k\geq 1$ be the unique sequence of edges connecting the first two boundary vertices in $\tree$, then
\begin{align}
	D(\mathsf{X}) = d_{\edge_1} - d_{\edge_k} = \sum_{k=1}^{k-1} (d_{\edge_k} - d_{\edge_{k+1}}).
\end{align}
If $\edge_{k} = \edge_{\vertex,1}$ and $\edge_{k+1} = \edge_{\vertex,2}$ share an inner vertex $\vertex$, then the formula \eqref{eq:edgedist} with $x$ at the position of $\vertex$ implies that
\begin{align}
	d_{\edge_k} - d_{\edge_{k+1}} = \log \sin \varphi(\edge_{\vertex,1}) - \log \sin \varphi(\edge_{\vertex,2}).\label{eq:ddiffinner}
\end{align}
Similarly, we will see that for two consecutive edges sharing a boundary vertex $\bvertex$, the difference $d_{\edge_k} - d_{\edge_{k+1}}$ can be expressed in terms of the boundary vertex decoration $(w_{\bvertex,j},v_{\bvertex,j})_{j=1}^{\deg(\bvertex)}$.

In the following we will concentrate on the function $f(\mathsf{X}) = \exp(2 u D(\mathsf{X}))$ for fixed $u\in \R$ of sufficiently small absolute value.
By \eqref{eq:ddiffinner}, the formula extends to an integrable function $f_\tree : \mathcal{A}_{\tree,\emptyset}^{\mathrm{anti}}(0,0,\mathbf{L}) \to\R$ because the angles $\varphi(\vec{\edge})$ are restricted to $(0,\pi)$.
We may thus invoke the integration formula \eqref{eq:integrationantidelaunay}.

To evaluate the right-hand side we combine the factorization of $f_\tree$ with that of the polytope $\mathcal{A}_{\tree,\mathsf{A}}^{\mathrm{anti}}(0,0,\mathbf{L})$ from Lemma~\ref{lem:antifactorization}.
Denoting by $\vertex_1,\ldots,\vertex_m \in \vsetinner(\tilde{\tree}) \cup \vsetboundary(\tilde{\tree})$ the vertices of the contracted tree $\tilde{\tree}$ encountered along the path between the first two boundary vertices, we obtain 
\begin{align*}
	\int_{\mathcal{A}_{\tree,\mathsf{A}}^{\mathrm{anti}}(\mathbf{L})} f_\tree(\lambda) \rmd \operatorname{Vol}_\tree(\lambda) &= \prod_{\substack{\vertex\in\vsetinner(\tilde{\tree})\\\vertex \notin \{\vertex_1,\ldots,\vertex_m\} }}\frac{2^{\deg{\vertex}-2}\pi^{2\deg(\vertex)-4}}{(2\deg(\vertex)-4)!} \prod_{\substack{\bvertex\in\vsetboundary(\tilde{\tree})\\\bvertex \notin \{\vertex_1,\ldots,\vertex_m\} }} \frac{2^{\deg{\vertex}-1}(L_\bvertex/2)^{2\deg(\bvertex)-2}}{((\deg(\bvertex)-1)!)^2} \prod_{i=1}^m E_{\vertex_i}(u), 
\end{align*}
where $E_{\vertex_i}(u)$ is the integral of $\exp(2u (d_{\edge} - d_{\edge'}))$, involving the edge distances of the entrance edge $\edge$ and the exit edge $\edge'$, over the polytope $\mathcal{A}_{\vertex_i}$.

\begin{figure}
	\centering
	\includegraphics[width=.8\linewidth]{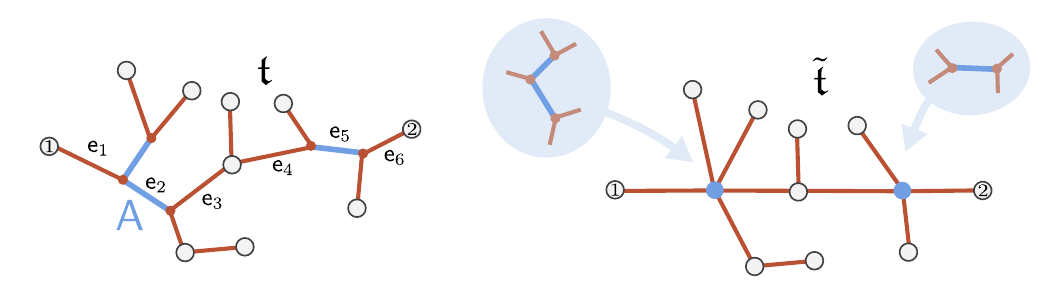}
	\caption{
		\label{fig:distancefactorization}
	}
\end{figure}

Suppose $\vertex_i \in \vsetboundary(\tilde{\tree})$ is a boundary vertex of degree $\deg(\vertex_i) = k \geq 2$.
Suppose the exit edge $\edge'$ is the $\ell$th edge, for $2\leq \ell \leq k$, if we start counting from the entrance edge $\edge$ in clockwise order around $\vertex_i$.
Then $E_{\vertex_i}(u)$ depends only on $k$, $\ell$ and the associated boundary length $L_{\vertex_i}$,
\begin{align}
	E_{\vertex_i}(u) = E_{k,\ell}(L_{\vertex_i},u).
\end{align}

Suppose otherwise that $\vertex_i \in \vsetinner(\tilde{\tree})$ is an anti-Delaunay vertex of degree $\deg(\vertex_i) = k \geq 3$.
Recall from Section~\ref{sec:inclexcl} that there is an associated subtree $\tree_{\vertex_i}$ with $k$ leafs, of which two are distinguished as the entrance and exit leaf.
Then $E_{\vertex_i}(u)$ depends only on this subtree $\tree_{\vertex_i}$ with distinguished leafs,
\begin{align}
	E_{\vertex_i}(u) = E_{\tree_{\vertex_i}}(u).
\end{align}

Recalling the definition of $\hat{X}_n(u;\mathbf{L})$ from \eqref{eq:Xhatn} in the introduction, the integration formula \eqref{eq:integrationantidelaunay} gives
\begin{align}
	\hat{X}_n(u;\mathbf{L})&=  \sum_{\tree\in \treeset_{2+n}(0,0,\mathbf{L})} \sum_{\mathsf{A}\subset \edgeset(\tree)} \Bigg[ (-1)^{|\mathsf{A}|} \prod_{\substack{\vertex\in\vsetinner(\tilde{\tree})\\\vertex \notin \{\vertex_1,\ldots,\vertex_m\} }}\frac{2^{\deg(\vertex)-2}\pi^{2\deg(\vertex)-4}}{(2\deg(\vertex)-4)!} \nonumber\\
	&\mkern200mu \prod_{\substack{\bvertex\in\vsetboundary(\tilde{\tree})\\\bvertex \notin \{\vertex_1,\ldots,\vertex_m\} }} \frac{2^{\deg(\vertex)-1}(L_\bvertex/2)^{2\deg(\bvertex)-2}}{((\deg(\bvertex)-1)!)^2} \prod_{i=1}^m E_{\vertex_i}(u)\Bigg] \nonumber \\
	&= \sum_{\tilde{\tree}\in \widetilde{\treeset}_{2+n}(0,0,\mathbf{L})} \prod_{\substack{\vertex\in\vsetinner(\tilde{\tree})\\\vertex \notin \{\vertex_1,\ldots,\vertex_m\} }}\frac{2^{\deg(\vertex)-2}\gamma_{\deg(\vertex)-1}}{(\deg(\vertex)-1)!} \prod_{\substack{\bvertex\in\vsetboundary(\tilde{\tree})\\\bvertex \notin \{\vertex_1,\ldots,\vertex_m\} }} \frac{2^{\deg(\bvertex)-2}t_{\deg(\bvertex)-1}(L_\bvertex)}{(\deg(\bvertex)-1)!}\,\prod_{i=1}^m \tilde{E}_{\vertex_i}(u),\label{eq:Xhatnsum}
\end{align}
with the new notation
\begin{align}
	\tilde{E}_{\vertex_i}(u) = \begin{cases} E_{\vertex_i}(u) & \vertex_i \in \vsetboundary(\tilde{\tree})\\ (-1)^{\deg(\vertex_i)-3} \sum_{\tree_{\vertex_i}}E_{\tree_{\vertex_i}}(u) & \vertex_i \in \vsetinner(\tilde{\tree})\end{cases},
\end{align}
where the sum is over all plane binary trees with a number of leafs equal to $\deg(\vertex_i)$. 

\subsection{Distance-dependent three-point function}

Following the procedure in Section~\ref{sec:stringeq}, the sum \eqref{eq:Xhatnsum} simplifies drastically at the level of the generating functions, because the contribution at each of the $m$ vertices along the path is independent of the others, 
\begin{align}
	\hat{X}(u;\mu] &= \sum_{n \geq 0} \frac{1}{n!} \int \hat{X}_n(u;\mathbf{L})\rmd \mu(L_1)\cdots \rmd \mu(L_n) \nonumber\\
	&= \sum_{m=0}^\infty \left( \hat{X}_{\mathrm{A}}(u,R[\mu]) + \int\rmd  \mu (L)\hat{X}_{\mathrm{B}}(L,u,R[\mu])\right)^m,\label{eq:hatXsum}
\end{align}
where we introduced the distance-dependent generating functions $\hat{X}_{\mathrm{A}}(u,r)$ and $\hat{X}_{\mathrm{B}}(L,u,r)$ around an anti-Delaunay vertex and boundary vertex respectively via
\begin{align}
	\hat{X}_{\mathrm{A}}(u,r) \coloneqq \sum_{\tree} (-1)^{|\vsetinner(\tree)|-1} r^{|\vsetinner(\tree)|} E_\tree(u), \qquad \hat{X}_{\mathrm{B}}(L,u,r) \coloneqq  \sum_{k=2}^\infty \sum_{\ell = 2}^k r^{k-2} E_{k,\ell}(L,u).
\end{align}
Here the sum in $\hat{X}_{\mathrm{A}}(u,r)$ is over all plane binary trees with a marked entrance and exit leaf, and $\vsetinner(\tree)$ denotes the non-leaf vertices of $\tree$.

In the next two subsections we will establish the following explicit formulas for $\hat{X}_{\mathrm{A}}(u,r)$ and $\hat{X}_{\mathrm{B}}(L,u,r)$.

\begin{proposition}\label{prop:hatXA}
	For an anti-Delaunay vertex we have the formal power series identity
	\begin{align*}
		\hat{X}_{\mathrm{A}}(u,r) = 1-\frac{2\pi u}{\sin(2\pi u)}\sum_{p=0}^\infty \frac{u^{2p}}{(2p+1)!!} \frac{\partial^{p+1}}{\partial r^{p+1}}\frac{\sqrt{r}}{\sqrt{2}\pi}J_1(2\pi\sqrt{2r}).
	\end{align*}
\end{proposition}
\begin{proposition}\label{prop:hatXB}
	For a boundary vertex we have the formal power series identity
	\begin{align*}
		\hat{X}_{\mathrm{B}}(L,u,r) = \frac{2\pi u}{\sin(2\pi u)}\sum_{p=0}^\infty \frac{u^{2p}}{(2p+1)!!}\frac{\partial^{p+1}}{\partial r^{p+1}}  I_0(L\sqrt{2r}).
	\end{align*}
\end{proposition}

Before turning to their proofs, we can conclude Theorem~\ref{thm:dist}.

\begin{proof}[Proof of Theorem~\ref{thm:dist}]
Comparing the formulas in Proposition~\ref{prop:hatXA} and Proposition~\ref{prop:hatXB} with the definition \eqref{eq:etadef} of $\eta(u;\mu]$, we find that
\begin{align}
	\hat{X}_{\mathrm{A}}(u,r) + \int \rmd\mu(L) \hat{X}_{\mathrm{B}}(L,u,r) = 1 - \frac{2\pi u}{\sin(2\pi u)} \eta(u;\mu].
\end{align}
Hence the sum in \eqref{eq:hatXsum} evaluates to
\begin{align*}
	\hat{X}(u;\mu] = \sum_{m=0}^\infty \left(1 - \frac{2\pi u}{\sin(2\pi u)} \eta(u;\mu]\right)^m = \frac{\sin(2\pi u)}{2\pi u \,\eta(u;\mu]}
\end{align*}
as claimed.
\end{proof}

\subsection{Around an inner vertex}

This subsection is devoted to proving Proposition~\ref{prop:hatXA}.
We start by identifying an explicit integral expression for $\hat{X}_{\mathrm{A}}(u,r)$.

\begin{lemma}\label{lem:hatX_A_as_angle_integral}
	\begin{align}
		\hat{X}_{\mathrm{A}}(u,r)= - \sum_{p=1}^\infty \int \rmd A \prod_{i=1}^p \left(-4F(r,\beta_i-\alpha_i)\left(\frac{\sin\alpha_i}{\sin\beta_i}\right)^{2u} \right),
	\end{align}
	where $\rmd A = \rmd \alpha_1\rmd \beta_1 \cdots \rmd \alpha_p \rmd\beta_p$ on $\{0<\alpha_1<\beta_1 < \cdots < \alpha_p < \beta_p <\pi \}$ and
	\begin{align}
		F(r,\theta)=\frac{\sqrt{r}}{\sqrt{2}\theta} J_1(2\theta\sqrt{2r}).
	\end{align}
\end{lemma}

\begin{proof}
	Let us start by giving an interpretation to the series $F(r,\theta)$. 
	From the proof of Lemma~\ref{lem:antifactorization} it follows that the polytope of an anti-Delaunay tree with $k$ inner vertices and angle $\varphi(\vec{\edge}_0) = \theta$ on the root instead of $\pi$, is related to simplex $\Delta^{(\theta)}_{2k+1}$ of size $\theta$.
	Therefore its volume (including the power of $2$ in the normalization) is $2^{k} \theta^{2k} / (2k)!$.
	Hence, the generating function of such trees with weight $r$ per non-root leaf and weight $-1$ per inner vertex is 
	\begin{align}
		\sum_{k=0}^\infty (-1)^k\frac{(2k)!}{k!(k+1)!} 2^{k} \frac{\theta^{2k}}{(2k)!} r^{k+1} = \frac{\sqrt{r}}{\sqrt{2}\theta} J_1(2\theta \sqrt{2r}) = F(r,\theta).
	\end{align}

	\begin{figure}[h]
	\centering
	\includegraphics[width=.4\linewidth]{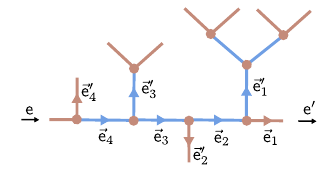}
	\caption{An example of a plane binary tree $\tree$ with a marked entrance and exit, in this case with $p=4$ vertices on the path.
		\label{fig:antidelaunaytree-entranceexit}
	}
\end{figure}
	Recall that we should sum over all plane binary trees $\tree$ with a marked entrance and exit leaf.
	Let $p \geq 1$ be the number of inner vertices of $\tree$ along the path from entrance to exit and let $\vec{\edge}_1,\ldots,\vec{\edge}_p$ be the edges making up the path in reverse order, pointing towards the exit (see Figure~\ref{fig:antidelaunaytree-entranceexit}).
	Let further $\vec{\edge}_i'$ be the oriented edge pointing away from the path and sharing its starting point with $\vec{\edge}_i$.
	Then we write for $i=1,\ldots,p$,  
	\begin{align}
		\alpha_i = \varphi(\vec{\edge}_{i}), \quad \beta_i = \varphi(\vec{\edge}_{i}) + \varphi(\vec{\edge}_{i}').
	\end{align}
	From our discussion in the proof of Lemma~\ref{lem:antifactorization} it follows that 
	\begin{align}
		0 < \alpha_1 < \beta_1 < \cdots < \alpha_p < \beta_p < \pi.
	\end{align}
	Moreover
	\begin{align}
		\exp(2u (d_{\edge}-d_{\edge'})) = \prod_{i=1}^p \left(\frac{\sin( \varphi(\vec{\edge}_{i})+\varphi(\vec{\edge}_{i}'))}{\sin \varphi(\vec{\edge}_i)}\right)^{2u} = \prod_{i=1}^p \left(\frac{\sin\beta_{i}}{\sin\alpha_i}\right)^{2u}.
	\end{align}
	Since we can bijectively encode $\tree$ by the $p$-tuple of subtrees whose roots are $\vec{\edge}_i$, $i=1,\ldots,p$, and those subtrees have angle $\varphi(\vec{\edge}_i') = \beta_i - \alpha_i$ on their root, we find
	\begin{align}
		\hat{X}_{\mathrm{A}}(u,r) = - \sum_{\tree} (-1)^{|\vsetinner(\tree)|} r^{|\vsetinner(\tree)|} E_\tree(u) &= - \sum_{p=1}^\infty (-1)^p 2^p \int\rmd A \prod_{i=1}^p \left(2F(r,\beta_i-\alpha_i)\left(\frac{\sin\beta_{i}}{\sin\alpha_i}\right)^{2u}\right),
	\end{align}
	where the factor of $2^p$ is due to the normalization of the volume measure $\operatorname{Vol}_\tree$ compared to $\rmd A$ and the factor of $2$ in the product takes into account that the subtrees can be on the left or the right of the path.
\end{proof}

To evaluate the integral in the previous lemma, we need the following two somewhat curious integration identities.

\begin{lemma}\label{lem:integral_ident1}
	For any $0\leq \alpha_1 \leq \beta_1\leq \cdots \leq \alpha_p \leq \beta_p \leq \pi$ and $u\in[0,1/2)$ we have the identity
	\begin{equation}\label{eq:sineratioaverage}
		\frac{1}{\pi} \int_0^\pi \rmd \gamma \prod_{i=1}^p \left(\frac{\sin(\alpha_i-\gamma)}{\sin(\beta_i-\gamma)}\right)^{2u}\ind_{\{\gamma \notin [\alpha_i,\beta_i]\}} = \frac{\sin\left(2u\left[\pi - \sum_{i=1}^p(\beta_i-\alpha_i)\right]\right)}{\sin(2\pi u)}.
	\end{equation}
\end{lemma}
\begin{proof}
	Let's assume $u \in (0,1/2)$, the case $u=0$ being obvious. The function
	\begin{equation*}
		g(z) = \prod_{i=1}^p \left(\frac{\sin(\alpha_i-z)}{\sin(\beta_i-z)}\right)^{2u} = \exp\left(2u\sum_{i=1}^p \log\left(\frac{\sin(\alpha_i-z)}{\sin(\beta_i-z)}\right)\right)
	\end{equation*}
	is $\pi$-periodic and holomorphic in $\C \setminus \bigcup_{\substack{1\leq i \leq p\\n\in \Z}}[\alpha_i+n \pi,\beta_i+n \pi]$. Therefore, the integral 
	\begin{equation*}
		\frac{1}{\pi}\int_0^\pi g(\gamma+iy)\rmd\gamma,\qquad y\in \R\setminus\{0\}
	\end{equation*}
	depends only on the sign of $y$.
	Hence
	\begin{equation*}
		\lim_{y\nearrow\infty}\frac{1}{\pi}\int_0^\pi \rmd\gamma\frac{e^{2\pi i u}g(\gamma+iy)-e^{-2\pi i u}g(\gamma-iy)}{2i} = \lim_{y\searrow 0}\frac{1}{\pi}\int_0^\pi \rmd\gamma \frac{e^{2\pi i u}g(\gamma+iy)-e^{-2\pi i u}g(\gamma-iy)}{2i}.
	\end{equation*}
	Since $g(\gamma \pm i y) \to e^{\pm 2u i \sum_{i=1}^p (\alpha_i - \beta_i)}$ as $y\to\infty$, the left-hand side equals
	\begin{equation*}
		\sin\left(2u\left[\pi - \sum_{i=1}^p(\beta_i-\alpha_i)\right]\right).
	\end{equation*}
	If $\gamma \in [\alpha_i,\beta_i]$ then $g(\gamma + i0) = e^{-4\pi i u} g(\gamma - i0)$, so the integrand on the right-hand side vanishes along the cuts. 
	The integral therefore evaluates to
	\begin{equation*}
		\frac{\sin(2\pi u)}{\pi} \int_0^\pi \rmd \gamma \prod_{i=1}^p \left(\frac{\sin(\alpha_i-\gamma)}{\sin(\beta_i-\gamma)}\right)^{2u} \ind_{\{\gamma \notin [\alpha_i,\beta_i]\}},
	\end{equation*}
	which proves the claimed expression.
\end{proof}
		
\begin{lemma}\label{lem:integral_ident2}
	For any analytic function $f : \R \to \R$ and $u\in(-1/2,1/2)$ we have
	\begin{align}
		Z_p[f]&:=\int \rmd A \prod_{i=1}^p \left(\frac{\sin\alpha_i}{\sin\beta_i}\right)^{2u} f(\beta_i-\alpha_i) \nonumber\\
		&= \int_0^\pi\rmd \theta_0 \frac{\operatorname{sinc}\left(2 \theta_0 u\right)}{\operatorname{sinc}(2\pi u)} \frac{\theta_0^p}{p!} \int_{\R_{>0}^p} \rmd\theta_1\cdots\rmd\theta_p \delta\left(\pi-\sum_{i=0}^p\theta_i\right)\prod_{i=1}^p f(\theta_i),
	\end{align}
	where $\operatorname{sinc}(x) = \sin(x)/x$ and $\rmd A=\rmd\alpha_1\rmd\beta_1\cdots\rmd\alpha_p\rmd\beta_p$ on $\{0<\alpha_1<\beta_1<\cdots<\alpha_p<\beta_p<\pi\}$.

	Furthermore, if $f$ has Laplace transform	\begin{equation}
		\hat{f}(y) = \int_0^\infty \rmd \theta e^{-y \theta} f(\theta),
	\end{equation}
	the previous expression can be summarized as the formal power series identity
	\begin{equation}
		\sum_{p=0}^\infty Z_p[f] = \frac{\pi}{\sin(2\pi u)}\mathcal{L}^{-1}\left\{ \arctan\left(\frac{2u}{y-\hat{f}(y)}\right)\right\}(\pi),
	\end{equation}
	where $\mathcal{L}^{-1}\left\{\hat{f}(y)\right\}(\theta)=f(\theta)$ is the corresponding inverse Laplace transform.
\end{lemma}
\begin{proof}
	Due to the shift symmetry in the Lebesgue measure and the $\pi$-periodicity of the integrand, we have
	\begin{align}
		 Z_p[f]&= \frac{1}{\pi-\sum_{i=1}^p(\beta_i-\alpha_i)} \int_0^\pi \rmd\gamma \int \rmd A \prod_{i=1}^p \left(\frac{\sin(\alpha_i-\gamma)}{\sin(\beta_i-\gamma)}\right)^{2u} f(\beta_i-\alpha_i)\ind_{\{\gamma \notin [\alpha_i,\beta_i]\}}.
	\end{align}
	We can now use Lemma~\ref{lem:integral_ident1} to give us
	\begin{align}
		Z_p[f]&=\int \rmd A  \frac{\operatorname{sinc}\left(2u\left[\pi - \sum_{i=1}^p(\beta_i-\alpha_i)\right]\right)}{\operatorname{sinc}(2\pi u)}\prod_{i=1}^p f(\beta_i-\alpha_i)\nonumber\\    
		&= \int_0^\pi \rmd\theta_0\frac{\operatorname{sinc}\left(2u\theta_0\right)}{\operatorname{sinc}(2\pi u)}\int \rmd A \,\delta\left(\pi-\theta_0-\sum_{i=0}^p(\beta_i-\alpha_i)\right)\prod_{i=1}^p f(\beta_i-\alpha_i)\nonumber\\
		&= \int_0^\pi\rmd \theta_0 \frac{\operatorname{sinc}\left(2u \theta_0\right)}{\operatorname{sinc}(2\pi u)} \frac{\theta_0^p}{p!} \int_{\R_{>0}^p} \rmd\theta_1\cdots\rmd\theta_p \delta\left(\pi-\theta_0-\sum_{i=0}^p\theta_i\right)\prod_{i=1}^p f(\theta_i),\nonumber
	\end{align}
	which proves the first equation.

	Using the Laplace transform, we get
	\begin{align}
		Z_p[f]&= \int_0^\pi\rmd \theta_0 \frac{\operatorname{sinc}(2u \theta_0)}{\operatorname{sinc}(2\pi u)} \frac{\theta_0^p}{p!} \mathcal{L}^{-1}\left\{\hat{f}(y)^p\right\}(\pi-\theta_0)\nonumber\\
		&=  \int_0^\infty\rmd \theta_0 \frac{\operatorname{sinc}(2u \theta_0)}{\operatorname{sinc}(2\pi u)} \frac{\theta_0^p}{p!} \mathcal{L}^{-1}\left\{e^{-\theta_0 y}\hat{f}(y)^p\right\}(\pi).
	\end{align}
	In particular, summing over $p$,
	\begin{align}
		\sum_{p=0}^\infty Z_p[f] &= \int_0^\infty\rmd \theta_0 \frac{\operatorname{sinc}(2u \theta_0)}{\operatorname{sinc}(2\pi u)}\mathcal{L}^{-1}\left\{e^{-\theta_0(y-\hat{f}(y))}\right\}(\pi)\nonumber\\
		&=\frac{\pi}{\sin(2\pi u)}\mathcal{L}^{-1}\left\{ \arctan\left(\frac{2u}{y-\hat{f}(y)}\right)\right\}(\pi), 
	\end{align}
	where we used $\int_0^\infty e^{-ax}\sin(bx)\frac{\dd{x}}{x}=\arctan(b/a)$.
\end{proof}

We now have all the ingredients to prove proposition~\ref{prop:hatXA}.
\begin{proof}[Proof of proposition~\ref{prop:hatXA}]
	Let us apply Lemma~\ref{lem:integral_ident2} to 
	\begin{align}
		f(\theta)=-4 F(r,\theta) = \frac{-2\sqrt{2r}}{\theta} J_1(2\theta\sqrt{2r}),
	\end{align}
	because then, using Lemma~\ref{lem:hatX_A_as_angle_integral}, 
	\begin{align}
		\hat{X}_A(u,r) = - \sum_{p=1}^\infty Z_p[f] = 1 - \sum_{p=0}^\infty Z_p[f].\label{eq:XhatZsum}
	\end{align}
	The Laplace transform of $f$ is
	\begin{align*}
		\hat{f}(y) = \int \rmd \theta e^{-y \theta} f(\theta) = y - \sqrt{y^2+8r}.
	\end{align*}
	Using the series expansion
	\begin{align*}
		\arctan\left(\frac{2u}{y-\hat{f}(y)}\right) &= \arctan\left(\frac{2u}{\sqrt{y^2+8r}}\right)=\sum_{p=0}^\infty \frac{(-1)^p}{2p+1} (2u)^{2p+1} (y^2+8r)^{-p-1/2},
	\end{align*}
	and the inverse Laplace transform 
	\begin{equation*}
		\mathcal{L}^{-1}\left\{ (y^2+8r)^{-p-1/2}\right\}(\pi) = \left(\frac{\pi}{2\sqrt{2r}}\right)^p \frac{1}{(2p-1)!!} J_{p}(2\pi\sqrt{2r}), 
	\end{equation*}
	we conclude that
	\begin{align}
		\sum_{p=0}^\infty Z_p[f] &= \frac{\pi}{\sin(2\pi u)}\mathcal{L}^{-1}\left\{ \arctan\left(\frac{2u}{\sqrt{y^2+8r}}\right)\right\}(\pi)\nonumber\\
		&=\frac{2\pi u}{\sin(2\pi u)}\sum_{p=0}^\infty \frac{u^{2p}}{(2p+1)!!}  \left(-\frac{\sqrt{2}\pi}{\sqrt{r}}\right)^p J_{p}(2\pi\sqrt{2r}) \nonumber\\
		&= \frac{2\pi u}{\sin(2\pi u)}\sum_{p=0}^\infty \frac{u^{2p}}{(2p+1)!!} \frac{\partial^{p+1}}{\partial r^{p+1}}\frac{\sqrt{r}}{\sqrt{2}\pi}J_1(2\pi\sqrt{2r}).
	\end{align}
	Together with \eqref{eq:XhatZsum} this gives the claimed formula.
\end{proof}

\subsection{Around a boundary vertex}
We now analyse the distance differences when passing through a boundary vertex $\bvertex$ of degree $k \geq 2$.
To lighten notation in this section, we write $w_j = w_{\bvertex,j}$, $v_j = v_{\bvertex,j}$ and $d_j = d_{\vec{e}_{\bvertex,j}}$ for $j=1,\ldots,k$. 
Moreover, we assume the ordering of edges around $\bvertex$ is chosen such that $\vec{\edge}_{\bvertex,1}$ is the entrance edge, connecting to boundary vertex $1$, and $\vec{\edge}_{\bvertex,\ell}$ is the exit edge, connecting to the boundary vertex with label $2$. 
Recall that Proposition~\ref{prop:hatXB} involves the function $E_{k,\ell}(L,u)$ given by the integral of $e^{2u(d_1 - d_\ell)}$ over the polytope $(\Delta_k^{(L/2)})^2$.

\begin{figure}
	\centering 
	\includegraphics[width=.9\linewidth]{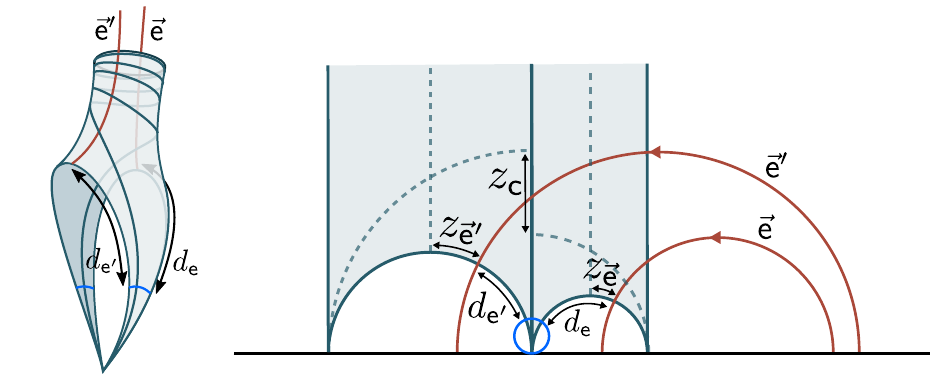}
	\caption{Left: illustration of the two ideal triangles spiralling into a boundary associated to $\bvertex$ of degree $\deg(\bvertex)=2$. Right: the relation between distances to the origin and the shears and half-shears. Note that $z_\corner$ is negative here, while $z_{\vec{\edge}},z_{\vec{\edge}'}$ are positive.\label{fig:distbdry}}
\end{figure}

Recall from Section~\ref{sec:spinetoshear} the construction  of the half-shear $z_{\vec{\edge}}$ associated to an oriented edge $\vec{\edge}$ starting at a boundary vertex (Figure~\ref{fig:shearw}), and the shear $z_{\corner}$ associated to a corner $\corner$ (Figure~\ref{fig:shearv}).
If $\vec{\edge} = \vec{\edge}_{\bvertex,j}$ and $\vec{\edge}'=\vec{\edge}_{\bvertex,j+1}$ share a corner $\corner = \corner_{\bvertex,j}$ the picture looks like that of Figure~\ref{fig:distbdry}.
Comparing to the indicated distances $d_{\edge} = d_j$ and $d_{\edge'} = d_{j+1}$ to the origin horocycle, it should be clear that we have the relation 
\begin{align}
	d_{\edge'} + z_{\vec{\edge}'} = d_{\edge} - z_{\vec{\edge}} - z_\corner.
\end{align}
Together with \eqref{eq:shearw} and \eqref{eq:shearv}, this gives
\begin{equation}
	e^{d_{j+1} - d_{j}} = e^{-w_{j+1}-w_{j}} e^{-z_\corner} = e^{v_{j}} \frac{1-e^{-w_{j+1}}}{e^{w_{j}}-1}.
\end{equation}
Taking products on both sides, we can compare $d_\ell$ to $d_1$ for $\ell = 1,\ldots,k$, 
\begin{equation}
	e^{2u (d_{\ell} - d_1)} = e^{2u \sum_{j=1}^{\ell-1} v_j} e^{-2u \sum_{j=2}^{\ell} w_{j}} \left(1-e^{-w_{\ell}}\right)^{2u}\left(1-e^{-w_{1}}\right)^{-2u}.
\end{equation}
Thanks to the separation of variables we can express $E_{k,\ell }(L,u)$ as 
\begin{align}
	E_{k,\ell }(L,u)= 2^{k-1} I_{k,\ell}(u,L)\tilde{I}_{k,\ell}(u,L),\label{eq:XBfactor}
\end{align}
where the factor $2^{k-1}$ comes from the normalization of the measure $\operatorname{Vol}_\tree$ and we have introduced
\begin{align}
	I_{k,\ell}(u,L) &:= \int_{\Delta_k^{(L/2)}} \rmd v_1\cdots \rmd v_{k-1} e^{2u \sum_{i=1}^{\ell-1} v_i}.\\
	\tilde{I}_{k,\ell}(u,L) &:= \int_{\Delta_k^{(L/2)}} \rmd w_1\cdots \rmd w_{k-1}e^{-2u \sum_{i=2}^{\ell} w_i}\left(1-e^{-w_\ell}\right)^{-2u}\left(1-e^{-w_1}\right)^{2u}.
\end{align}

\begin{proof}[Proof of Proposition~\ref{prop:hatXB}]
	Recall that we should compute 
	\begin{align*}
		\hat{X}_{\mathrm{B}}(L,u,r) = \sum_{k=2}^\infty \sum_{\ell = 2}^k r^{k-2} E_{k,\ell}(L,u) = \sum_{k=2}^\infty \sum_{\ell = 2}^k r^{k-2} 2^{k-1} I_{k,\ell}(u,L)\tilde{I}_{k,\ell}(u,L)
	\end{align*}
	We can switch from integrating over the simplex to the full $\R^k_{>0}$ by Laplace transforming,
	\begin{align}
		\int_0^\infty \rmd L e^{-Lz} I_{k,\ell}(u,L) &= 2 \int_0^\infty \rmd v_1\cdots \rmd v_k e^{-2z\sum_{i=1}^k v_i +2u \sum_{i=1}^{\ell-1} v_i} \nonumber\\
		&= 2^{1-k} (z-u)^{1-\ell}z^{\ell-k-1} \label{eq:Ikllaplace}\\
		&= 2^{1-k}\sum_{m=0}^{\infty} \binom{\ell+m-2}{m} u^m z^{-k-m}.\nonumber
	\end{align}
	Taking the inverse Laplace tranform yields
	\begin{align}
		I_{k,\ell}(u,L) = \sum_{m=0}^{\infty}  u^m L^{k+m-1} \frac{2^{1-k}}{(k+m-1)!}\binom{\ell+m-2}{m}.\label{eq:Iklexpr} 
	\end{align}
	Similarly for $\tilde{I}_{k,\ell}(u,L)$ the Laplace transform factorizes,
	\begin{align}
		\int_0^\infty \rmd L e^{-Lz} \tilde{I}_{k,\ell}(u,L) &= 2 \int_0^\infty \rmd w_1\cdots \rmd w_k e^{-2z\sum_{i=1}^kw_i - 2u \sum_{i=2}^{\ell} w_i} \left(1-e^{-w_\ell}\right)^{-2u}\left(1-e^{-w_1}\right)^{2u}\nonumber\\
		&= 2^{3-k}(z+u)^{2-\ell}z^{\ell-k} \int_0^\infty \rmd w_1 e^{-2z w_1}(1-e^{-w_1})^{2u}\int_0^\infty \rmd w_\ell e^{-2(z+u) w_\ell}(1-e^{-w_\ell})^{-2u}\nonumber\\
		&= 2^{3-k}(z+u)^{2-\ell}z^{\ell-k} \int_0^1 \rmd x\, x^{2z-1}(1-x)^{2u}\int_0^\infty \rmd y\, y^{2(z+u)-1}(1-y)^{-2u}\nonumber\\
		&= 2^{3-k}(z+u)^{2-\ell}z^{\ell-k} \frac{\Gamma(2z)\Gamma(1+2u)}{\Gamma(2z+2u+1)} \,\frac{\Gamma(2z+2u)\Gamma(1-2u)}{\Gamma(2z+1)} \nonumber\\
		&= 2^{1-k}(z+u)^{1-\ell}z^{\ell-k-1} \Gamma(1+2u)\Gamma(1-2u),
	\end{align}
	where in the second-to-last equality we recognized Beta integrals.
	Using the reflection formula for the Gamma function
	\begin{equation}
		\Gamma(1-2u)\Gamma(1+2u) = \frac{2\pi u}{\sin 2\pi u},
	\end{equation}
	and comparing to \eqref{eq:Ikllaplace}, it follows immediately that
	\begin{equation}
		\tilde{I}_{k,\ell}(u,L) = \frac{2\pi u}{\sin 2\pi u} I_{k,\ell}(-u,L).
	\end{equation}

	Summing over $\ell$, we find
	\begin{align}
		\sum_{\ell=2}^k I_{k,\ell}(u,L)\tilde{I}_{k,\ell}(u,L) &= \frac{2\pi }{u\sin 2\pi u} \sum_{\ell=2}^k u^2I_{k,\ell}(-u,L)I_{k,\ell}(u,L)\nonumber\\
		&\!\!\stackrel{\eqref{eq:Iklexpr}}{=} \frac{2\pi 4^{1-k}L^{2k-4}}{u\sin 2\pi u} \sum_{\ell=2}^k\sum_{p=1}^\infty u^{2p}L^{2p} \sum_{m=0}^{2p-2} \frac{(-1)^m}{(k+m-1)!}\binom{\ell+m-2}{m} \breakline
		\frac{1}{(k+2p-m-3)!}\binom{\ell+2p-m-4}{2p-m-2}\nonumber\\
		&=\frac{2\pi 4^{1-k}L^{2k-4}}{u\sin 2\pi u} \sum_{p=1}^\infty u^{2p}L^{2p} 2\frac{p!}{(2p)!(k-2)!(k+p-2)!}, 
	\end{align}
	where in the last equality we use the summation identity of Lemma~\ref{lem:hypident} below.
	
	Putting everything together we obtain
	\begin{align}
		\hat{X}_{\mathrm{B}}(L,u,r) &= \sum_{k=2}^\infty r^{k-2} 2^{k-1} \sum_{\ell=2}^k I_{k,\ell}(u,L)\tilde{I}_{k,\ell}(u,L)\nonumber\\
		&= \frac{2\pi u}{\sin(2\pi u)} \sum_{k=2}^\infty \sum_{p=1}^\infty u^{2p-2}\frac{2^{p}p!}{(2p)!} \frac{r^{k-2} 2^{2-k-p}L^{2(k+p-2)}}{(k-2)!(k+p-2)!}\nonumber\\
		&= \frac{2\pi u}{\sin(2\pi u)} \sum_{p=0}^\infty \frac{u^{2p}}{(2p+1)!!} \sum_{k=2}^\infty \frac{r^{k-2} 2^{1-k-p}L^{2(k+p-1)}}{(k-2)!(k+p-1)!}\nonumber\\
		&= \frac{2 \pi u}{\sin(2\pi u)} \sum_{p=0}^\infty \frac{u^{2p}}{(2p+1)!!}\frac{\partial^{p+1}}{\partial r^{p+1}} \sum_{k=1-p}^\infty  \frac{2^{1-k-p}L^{2(k+p-1)}r^{k+p-1}}{(k+p-1)!^2}\nonumber\\
		&= \frac{2\pi u}{\sin(2\pi u)} \sum_{p=0}^\infty \frac{u^{2p}}{(2p+1)!!}\frac{\partial^{p+1}}{\partial r^{p+1}}  I_0(L\sqrt{2r}),
	\end{align}
	as claimed.
\end{proof}

\begin{lemma}\label{lem:hypident}
	The following summation identity holds for every $p\geq 1$ and $k\geq 2$, 
	\begin{align*}
		\sum_{\ell=2}^k\sum_{m=0}^{2p-2} (-1)^m\frac{\binom{\ell+m-2}{m}}{(k+m-1)!} \frac{\binom{\ell+2p-m-4}{2p-m-2}}{(k+2p-m-3)!} &= 2\frac{p!}{(2p)!(k-2)!(p+k-2)!}.
	\end{align*}
\end{lemma}
\begin{proof}
	We give a proof based on known hypergeometric identities.
	Performing the sum over $m$ on the left-hand side we obtain 
	\begin{align*}
		\sum_{\ell=2}^k\frac{(\ell+2p-4)!}{(\ell-2)!(2p-2)!(k+2p-2)!(k-1)!}{_3F_2}\left[\begin{smallmatrix}
			2p-2&3-k-2p&\ell-1\\k&4-2p-\ell
		\end{smallmatrix};1\right],
	\end{align*}
	where ${_pF_q}$ is a generalized hypergeometric function.
	According to \cite[equations (2.3.3.2) and (1.1.5)]{Slater1966Generalized} this hypergeometric function evaluates to
	\begin{align}
		{_3F_2}&\left[\begin{smallmatrix}
			2p-2&3-k-2p&\ell-1\\k&4-2p-\ell
		\end{smallmatrix};1\right]=\frac{(3-p-\ell)^{(\ell-2)}(1+k-\ell)^{(\ell-2)}}{(4-2p-\ell)^{(\ell-2)}(p+k-\ell)^{(\ell-2)}}{_3F_2}\left[\begin{smallmatrix}
			2p-2&3-k-2p&1\\k&2-2p
		\end{smallmatrix};1\right]\nonumber\\
		&= \frac{(3-p-\ell)^{(\ell-2)}(1+k-\ell)^{(\ell-2)}}{(4-2p-\ell)^{(\ell-2)}(p+k-\ell)^{(\ell-2)}}
		\left({_2F_1}\left[\begin{smallmatrix}
			3-k-2p&1\\k
		\end{smallmatrix};1\right]-\frac{(3-k-2p)^{(2p-1)}}{k^{(2p-1)}}{_2F_1}\left[\begin{smallmatrix}
			2-k&1\\k+2p-1
		\end{smallmatrix};1\right]\right)\nonumber\\
		&= \frac{(3-p-\ell)^{(\ell-2)}(1+k-\ell)^{(\ell-2)}}{(4-2p-\ell)^{(\ell-2)}(p+k-\ell)^{(\ell-2)}}\frac{2(k-1)}{2p+2k-4},
	\end{align}
	where $a^{(n)}=a(a+1)(a+2)\dots(a+n-1)$ is the rising factorial.
	With this the desired sum becomes 
	\begin{align}
		&\sum_{\ell=2}^k\frac{(\ell+2p-4)!}{(\ell-2)!(2p-2)!(k+2p-2)!(k-1)!}\frac{(3-p-\ell)^{(\ell-2)}(1+k-\ell)^{(\ell-2)}}{(4-2p-\ell)^{(\ell-2)}(p+k-\ell)^{(\ell-2)}}\frac{2(k-1)}{2p+2k-4}\nonumber\\
		&=\sum_{\ell=2}^k\frac{(k+p-\ell-1)!(p+\ell-3)!}{(k-\ell)!(\ell-2)!(p-1)!(p+k-2)!(2p+k-3)!}\nonumber\\
		&= \frac{1}{(p+k-2)!(2p+k-3)!(k-2)!} {_2F_1}\left[\begin{smallmatrix}
			p+1&-k\\-k+p
		\end{smallmatrix};1\right]=2 \frac{p!}{(p+k-2)!(2p)!(k-2)!}, 
	\end{align}
	as claimed.
\end{proof}

\bibliographystyle{siam}
\bibliography{ref}	

\begin{thebibliography}{10}

\bibitem{Albenque_generic_2015}
{\sc M.~Albenque and D.~Poulalhon}, {\em A generic method for bijections between blossoming trees and planar maps}, Electron. J. Combin., 22 (2015), pp.~Paper 2.38, 44.

\bibitem{Andersen_Kontsevich_2020}
{\sc J.~E. Andersen, G.~Borot, S.~Charbonnier, A.~Giacchetto, D.~Lewa{\'n}ski, and C.~Wheeler}, {\em On the {Kontsevich} geometry of the combinatorial {Teichm\"uller} space}, arXiv preprint arXiv:2010.11806,  (2020).

\bibitem{Bernardi_bijection_2012}
{\sc O.~Bernardi and E.~Fusy}, {\em A bijection for triangulations, quadrangulations, pentagulations, etc}, J. Combin. Theory Ser. A, 119 (2012), pp.~218--244.

\bibitem{Borot_Topological_2020}
{\sc G.~Borot}, {\em Topological recursion and geometry}, Rev. Math. Phys., 32 (2020), pp.~2030007, 50.

\bibitem{Bouttier_Planar_2019}
{\sc J.~Bouttier}, {\em Planar maps and random partitions}, {Habilitation thesis}, Universit\'{e} Paris-Sud, 2019.
\newblock arXiv:1912.06855.

\bibitem{Bouttier_Geodesic_2003}
{\sc J.~Bouttier, P.~Di~Francesco, and E.~Guitter}, {\em Geodesic distance in planar graphs}, Nuclear Phys. B, 663 (2003), pp.~535--567.

\bibitem{Bouttier_Planar_2004}
\leavevmode\vrule height 2pt depth -1.6pt width 23pt, {\em Planar maps as labeled mobiles}, Electron. J. Combin., 11 (2004), pp.~Research Paper 69, 27 pp. (electronic).

\bibitem{Bouttier_Planar_2012}
{\sc J.~Bouttier and E.~Guitter}, {\em Planar maps and continued fractions}, Comm. Math. Phys., 309 (2012), pp.~623--662.

\bibitem{Bouttier_Bijective_2022}
{\sc J.~Bouttier, E.~Guitter, and G.~Miermont}, {\em Bijective enumeration of planar bipartite maps with three tight boundaries, or how to slice pairs of pants}, Ann. H. Lebesgue, 5 (2022), pp.~1035--1110.

\bibitem{Bouttier_Enumeration_2024}
\leavevmode\vrule height 2pt depth -1.6pt width 23pt, {\em Enumeration of maps with tight boundaries and the zhukovsky transformation}, arXiv preprint arXiv:2406.13528,  (2024).

\bibitem{Bowditch_Natural_1988}
{\sc B.~Bowditch and D.~Epstein}, {\em Natural triangulations associated to a surface}, Topology, 27 (1988), pp.~91--117.

\bibitem{budd2020irreducible}
{\sc T.~Budd}, {\em Irreducible metric maps and {W}eil-{P}etersson volumes}, Comm. Math. Phys., 394 (2022), pp.~887--917.

\bibitem{budd2025random}
{\sc T.~Budd and N.~Curien}, {\em Random punctured hyperbolic surfaces \& the brownian sphere}, arXiv preprint arXiv:2508.18792,  (2025).

\bibitem{Budd_tight_2025}
{\sc T.~Budd and T.~Lions}, {\em The tight length spectrum of large-genus random hyperbolic surfaces with many cusps}, arXiv preprint arXiv:2506.02611,  (2025).

\bibitem{budd2023topological}
{\sc T.~Budd and B.~Zonneveld}, {\em Topological recursion of the {W}eil-{P}etersson volumes of hyperbolic surfaces with tight boundaries}, J. Math. Phys., 65 (2024), pp.~Paper No. 092302, 34.

\bibitem{Castro_Critical_2023}
{\sc A.~Castro}, {\em Critical {JT} gravity}, J. High Energy Phys.,  (2023), pp.~Paper No. 36, 27.

\bibitem{Charbonnier2017}
{\sc S.~Charbonnier, F.~David, and B.~Eynard}, {\em Local properties of the random {Delaunay} triangulation model and topological 2d gravity},  (2017).
\newblock arXiv:1701.02580.

\bibitem{Chekhov_Quantizing_2004}
{\sc L.~Chekhov and R.~C. Penner}, {\em On quantizing {Teichm\"uller} and thurston theories}, arXiv preprint math/0403247,  (2004).

\bibitem{Chekhov_Fenchel_2020}
{\sc L.~O. Chekhov}, {\em Fenchel-{N}ielsen coordinates and {G}oldman brackets}, Uspekhi Mat. Nauk, 75 (2020), pp.~153--190.

\bibitem{Cori_Planar_1981}
{\sc R.~Cori and B.~Vauquelin}, {\em Planar maps are well labeled trees}, Canad. J. Math., 33 (1981), pp.~1023--1042.

\bibitem{Crainic_Lectures_2021}
{\sc M.~Crainic, R.~L. Fernandes, and I.~M\u{a}rcu\c{t}}, {\em Lectures on {P}oisson geometry}, vol.~217 of Graduate Studies in Mathematics, American Mathematical Society, Providence, RI, 2021.

\bibitem{Curien_Scaling_2025}
{\sc N.~Curien, G.~Miermont, and A.~Riera}, {\em The scaling limit of planar maps with large faces}, arXiv preprint arXiv:2501.18566,  (2025).

\bibitem{Do_asymptotic_2010}
{\sc N.~Do}, {\em The asymptotic {Weil}-{Petersson} form and intersection theory on $\mathcal{M}_{g,n}$},  (2010).
\newblock arXiv:1010.4126.

\bibitem{Eynard_short_2014}
{\sc B.~Eynard}, {\em A short overview of the" topological recursion"}, arXiv preprint arXiv:1412.3286,  (2014).

\bibitem{eynard2016counting}
\leavevmode\vrule height 2pt depth -1.6pt width 23pt, {\em Counting surfaces}, Progress in Mathematical Physics, 70 (2016).

\bibitem{Eynard_Invariants_2007}
{\sc B.~Eynard and N.~Orantin}, {\em Invariants of algebraic curves and topological expansion}, arXiv preprint math-ph/0702045,  (2007).

\bibitem{eynard2007weil}
\leavevmode\vrule height 2pt depth -1.6pt width 23pt, {\em Weil-petersson volume of moduli spaces, mirzakhani's recursion and matrix models}, arXiv preprint arXiv:0705.3600,  (2007).

\bibitem{Fock_Description_1993}
{\sc V.~Fock}, {\em Description of moduli space of projective structures via fat graphs}, arXiv preprint hep-th/9312193,  (1993).

\bibitem{Fock_Dual_1997}
\leavevmode\vrule height 2pt depth -1.6pt width 23pt, {\em Dual {Teichm\"uller} spaces}, arXiv preprint dg-ga/9702018,  (1997).

\bibitem{Fock_Quantum_1999}
{\sc V.~V. Fock and L.~O. Chekhov}, {\em Quantum mapping class group, pentagon relation, and geodesics}, Trudy Matematicheskogo Instituta imeni VA Steklova, 226 (1999), pp.~163--179.

\bibitem{Guth_Pants_2011}
{\sc L.~Guth, H.~Parlier, and R.~Young}, {\em Pants decompositions of random surfaces}, Geom. Funct. Anal., 21 (2011), pp.~1069--1090.

\bibitem{Hide_Short_2025}
{\sc W.~Hide and J.~Thomas}, {\em Short geodesics and small eigenvalues on random hyperbolic punctured spheres}, Comment. Math. Helv., 100 (2025), pp.~463--506.

\bibitem{Janson_Brownian_2007}
{\sc S.~Janson}, {\em Brownian excursion area, {W}right's constants in graph enumeration, and other {B}rownian areas}, Probab. Surv., 4 (2007), pp.~80--145.

\bibitem{Kontsevich1992}
{\sc M.~Kontsevich}, {\em Intersection theory on the moduli space of curves and the matrix {A}iry function}, Comm. Math. Phys., 147 (1992), pp.~1--23.

\bibitem{Koster_universality_2022}
{\sc P.~Koster}, {\em Universality classes of 2D hyperbolic Riemannian manifolds}, {Bachelor thesis}, Radboud University, 2022.

\bibitem{LeGall_Uniqueness_2013}
{\sc J.-F. Le~Gall}, {\em Uniqueness and universality of the {B}rownian map}, Ann. Probab., 41 (2013), pp.~2880--2960.

\bibitem{LeGall_scaling_2010}
{\sc J.-F. Le~Gall and G.~Miermont}, {\em On the scaling limit of random planar maps with large faces}, in X{VI}th {I}nternational {C}ongress on {M}athematical {P}hysics, World Sci. Publ., Hackensack, NJ, 2010, pp.~470--474.

\bibitem{Liu_Recursion_2009}
{\sc K.~Liu and H.~Xu}, {\em Recursion formulae of higher {W}eil-{P}etersson volumes}, Int. Math. Res. Not. IMRN,  (2009), pp.~835--859.

\bibitem{Manin_Invertible_2000}
{\sc Y.~I. Manin and P.~Zograf}, {\em Invertible cohomological field theories and {W}eil-{P}etersson volumes}, Ann. Inst. Fourier (Grenoble), 50 (2000), pp.~519--535.

\bibitem{Marckert_Limit_2006}
{\sc J.-F. Marckert and A.~Mokkadem}, {\em Limit of normalized quadrangulations: the {B}rownian map}, Ann. Probab., 34 (2006), pp.~2144--2202.

\bibitem{Miermont_Brownian_2013}
{\sc G.~Miermont}, {\em The {B}rownian map is the scaling limit of uniform random plane quadrangulations}, Acta Math., 210 (2013), pp.~319--401.

\bibitem{Mirzakhani2007}
{\sc M.~Mirzakhani}, {\em Simple geodesics and {W}eil-{P}etersson volumes of moduli spaces of bordered {R}iemann surfaces}, Invent. Math., 167 (2007), pp.~179--222.

\bibitem{Mirzakhani2007a}
\leavevmode\vrule height 2pt depth -1.6pt width 23pt, {\em Weil-{P}etersson volumes and intersection theory on the moduli space of curves}, J. Amer. Math. Soc., 20 (2007), pp.~1--23.

\bibitem{Mirzakhani_Growth_2013}
\leavevmode\vrule height 2pt depth -1.6pt width 23pt, {\em Growth of {W}eil-{P}etersson volumes and random hyperbolic surfaces of large genus}, J. Differential Geom., 94 (2013), pp.~267--300.

\bibitem{mirzakhani2017lengths}
{\sc M.~Mirzakhani and B.~Petri}, {\em Lengths of closed geodesics on random surfaces of large genus}, Commentarii Mathematici Helvetici, 94 (2019), pp.~869--889.

\bibitem{Mulase_Mirzakhanis_2008}
{\sc M.~Mulase and B.~Safnuk}, {\em Mirzakhani's recursion relations, {V}irasoro constraints and the {K}d{V} hierarchy}, Indian J. Math., 50 (2008), pp.~189--218.

\bibitem{Mullin_enumeration_1967}
{\sc R.~C. Mullin}, {\em On the enumeration of tree-rooted maps}, Canadian Journal of Mathematics, 19 (1967), pp.~174--183.

\bibitem{Penner_decorated_1987}
{\sc R.~C. Penner}, {\em The decorated {Teichm\"uller} space of punctured surfaces}, Comm. Math. Phys., 113 (1987), pp.~299--339.

\bibitem{Penner_Weil_1992}
\leavevmode\vrule height 2pt depth -1.6pt width 23pt, {\em Weil-{P}etersson volumes}, J. Differential Geom., 35 (1992), pp.~559--608.

\bibitem{Penner_Decorated_2012}
\leavevmode\vrule height 2pt depth -1.6pt width 23pt, {\em Decorated {T}eichm\"{u}ller theory}, QGM Master Class Series, European Mathematical Society (EMS), Z\"{u}rich, 2012.
\newblock With a foreword by Yuri I. Manin.

\bibitem{Rivin1992}
{\sc I.~Rivin}, {\em Intrinsic geometry of convex ideal polyhedra in hyperbolic 3-space},  (1992).
\newblock arXiv:math/0005234.

\bibitem{Rivin1996}
\leavevmode\vrule height 2pt depth -1.6pt width 23pt, {\em A {Characterization} of {Ideal} {Polyhedra} in {Hyperbolic} 3-{Space}}, Annals of Mathematics, 143 (1996), pp.~51--70.

\bibitem{Schaeffer_Conjugaison_1998}
{\sc G.~Schaeffer}, {\em Conjugaison d'arbres et cartes combinatoires al{\'e}atoires. {PhD} thesis},  (1998).

\bibitem{Slater1966Generalized}
{\sc L.~J. Slater}, {\em Generalized Hypergeometric Functions}, Cambridge University Press, 1966.

\bibitem{Talbott_Critical_2025}
{\sc H.~Talbott}, {\em Critical exponents on hyperbolic surfaces with long boundaries and the asymptotic weil-petersson form}, arXiv preprint arXiv:2501.08447,  (2025).

\bibitem{Thurston_Minimal_1998}
{\sc W.~P. Thurston}, {\em Minimal stretch maps between hyperbolic surfaces}, arXiv preprint math/9801039,  (1998).

\bibitem{Witten_Two_1991}
{\sc E.~Witten}, {\em Two-dimensional gravity and intersection theory on moduli space}, in Surveys in differential geometry ({C}ambridge, {MA}, 1990), Lehigh Univ., Bethlehem, PA, 1991, pp.~243--310.

\bibitem{Wolpert_symplectic_1983}
{\sc S.~Wolpert}, {\em On the symplectic geometry of deformations of a hyperbolic surface}, Ann. of Math. (2), 117 (1983), pp.~207--234.

\bibitem{Zonneveld_tree_2025}
{\sc B.~Zonneveld}, {\em A tree bijection for cusp-less planar hyperbolic surfaces}.
\newblock to appear.

\end{thebibliography}
	
\end{document}